\documentclass[10pt,reqno]{amsart}
\usepackage{hyperref}
\usepackage{latexsym,amsmath}
\usepackage{enumerate}
\usepackage{amsfonts}
\usepackage{amssymb}
\usepackage{geometry}
\usepackage{latexsym}
\usepackage{fixmath}
\usepackage[T1]{fontenc}
\usepackage{fourier}
\usepackage{bbm}
\usepackage{color}
\usepackage{xcolor}
\usepackage{fullpage}
\usepackage{textcomp}

\newcommand{\Zbal}{Z_\mathrm{bal}(\PHI,\beta)}
\newcommand{\dk}{d_{\mathrm{SAT}}}

\newcommand{\sign}{\mathrm{sign}}
\newcommand{\disteq}{\stacksign{d}=}
\newcommand{\BAL}{\mathrm{BAL}}
\newcommand{\vmu}{\vec\mu}
\renewcommand{\epsilon}{\eps}

\newcommand{\ex}{\Erw}

\newcommand{\vgamma}{\vec\gamma}
\newcommand{\vg}{\vec g}
\newcommand{\veta}{\vec\eta}
\newcommand{\vPi}{\vec\Pi}
\newcommand{\vSigma}{\vec\Sigma}
\newcommand{\vxi}{\vec\xi}
\newcommand{\vE}{\vec E}

\newcommand{\vd}{\vec d}
\newcommand{\vchi}{\vec\chi}
\newcommand{\vzeta}{\vec\zeta}

\newcommand\fa{\mathfrak a}
\newcommand\fb{\mathfrak b}
\newcommand\fc{\mathfrak c}
\newcommand\fd{\mathfrak d}
\newcommand\fe{\mathfrak e}
\newcommand\ff{\mathfrak f}

\newcommand\fx{\mathfrak x}

\newcommand\fB{\mathfrak B}
\newcommand\fE{\mathfrak E}
\newcommand\fy{\mathfrak y}
\newcommand\fz{\mathfrak z}
\newcommand\fH{\mathfrak H}

\newcommand\MU{\vec\mu}
\newcommand\vDelta{\vec\Delta}
\newcommand\vP{\vec P}
\newcommand\vQ{\vec Q}
\newcommand\vX{\vec X}
\newcommand\vZ{\vec Z}
\newcommand\fD{\mathfrak D}
\newcommand\vI{\vec I}
\newcommand\vY{\vec Y}
\newcommand\vT{\vec T}

\newcommand\vJ{\vec J}

\newcommand\vm{\vec m}

\newcommand\vU{\vec U}

\newcommand\PHI{\vec\Phi}

\newcommand\nix{\,\cdot\,}

\newcommand\dd{{\mathrm d}}

\newcommand\G{\vec G}

\numberwithin{equation}{section}

\renewcommand{\vec}[1]{\boldsymbol{#1}}

\newcommand\KL[2]{D_{\mathrm{KL}}\bc{{{#1}\|{#2}}}}

\newcommand\SIGMA{\vec\sigma}
\newcommand\CHI{\vec\chi}
\newcommand\TAU{\vec\tau}

\newtheorem{definition}{Definition}[section]
\newtheorem{claim}[definition]{Claim}

\newtheorem{theorem}[definition]{Theorem}
\newtheorem{lemma}[definition]{Lemma}
\newtheorem{proposition}[definition]{Proposition}
\newtheorem{corollary}[definition]{Corollary}

\newtheorem{fact}[definition]{Fact}

\newcommand\fY{\mathfrak{Y}}
\newcommand\fA{\mathfrak{A}}

\newcommand\cA{\mathcal{A}}
\newcommand\cB{\mathcal{B}}
\newcommand\cC{\mathcal{C}}

\newcommand\cE{\mathcal{E}}
\newcommand\cU{\mathcal{U}}

\newcommand\cH{\mathcal{H}}
\newcommand\cS{\mathcal{S}}
\newcommand\cT{\mathcal{T}}

\newcommand\cL{\mathcal{L}}

\newcommand\cO{\mathcal{O}}
\newcommand\cP{\mathcal{P}}
\newcommand\cX{\mathcal{X}}
\newcommand\cY{\mathcal{Y}}
\newcommand\cV{\mathcal{V}}
\newcommand\cW{\mathcal{W}}

\def\cR{{\mathcal R}}
\def\cC{{\mathcal C}}
\def\cE{{\mathcal E}}

\newcommand\fM{\mathfrak{M}}
\newcommand\fL{\mathfrak{L}}

\newcommand\fp{\mathfrak{p}}
\newcommand\vx{\vec x}
\newcommand\va{\vec a}
\newcommand\vs{\vec s}
\newcommand\vu{\vec u}
\newcommand\vy{\vec y}

\newcommand\eul{\mathrm{e}}
\newcommand\eps{\varepsilon}

\newcommand\Erw{\mathbb{E}}
\newcommand{\vecone}{\mathbb{1}}

\newcommand{\set}[1]{\left\{#1\right\}}
\newcommand{\Po}{{\rm Po}}
\newcommand{\Bin}{{\rm Bin}}

\newcommand{\Be}{{\rm Be}}
\newcommand{\Fra}{\mathfrak{F}}
\newcommand{\q}{\mathfrak{q}}

\newcommand\dTV{d_{\mathrm{TV}}}

\newcommand{\bink}[2] {{\binom{#1}{#2}}}
\newcommand{\stirling}[2] {\cbc{\begin{array}{c}{#1}\\{#2}\end{array}}}

\newcommand\bc[1]{\left({#1}\right)}
\newcommand\cbc[1]{\left\{{#1}\right\}}
\newcommand\bcfr[2]{\bc{\frac{#1}{#2}}}
\newcommand{\bck}[1]{\left\langle{#1}\right\rangle}
\newcommand\brk[1]{\left\lbrack{#1}\right\rbrack}
\newcommand\scal[2]{\bck{{#1},{#2}}}

\newcommand\abs[1]{\left|{#1}\right|}

\newcommand\RR{\mathbb{R}}

\newcommand{\Whp}{A.a.s.}
\newcommand{\whp}{a.a.s.}

\newcommand{\stacksign}[2]{{\stackrel{\mbox{\scriptsize #1}}{#2}}}
\newcommand{\tensor}{\otimes}
\newcommand\pr{\mathbb{P}} 
\renewcommand\Pr{\pr} 

\newcommand\Lem{Lemma}
\newcommand\Prop{Proposition}
\newcommand\Thm{Theorem}

\newcommand\Cor{Corollary}
\newcommand\Sec{Section}
\newcommand\Chap{Chapter}

\newcommand{\Faadi}{Fa\`a di Bruno}

\newcommand\id{\mathrm{id}}

\newcommand{\dsat}{\dk}

\newcommand\supp{\mathrm{supp}}
\renewcommand\ln{\log}

\DeclareRobustCommand{\stirling}{\genfrac\{\}{0pt}{}}

\begin{document}

\title{Belief Propagation on the random $k$-SAT model}

\begin{abstract}
	Corroborating a prediction from statistical physics, we prove that the Belief Propagation message passing algorithm approximates the partition function of the random $k$-SAT model well for all clause/variable densities and all inverse temperatures for which a modest absence of long-range correlations condition is satisfied.
	This condition is known as ``replica symmetry'' in physics language.
	From this result we deduce that a replica symmetry breaking phase transition occurs in the random $k$-SAT model at low temperature for clause/variable densities below but close to the satisfiability threshold.
 \hfill {\em MSc: 	68Q87, 60C05}
\end{abstract}

\thanks{Supported by DFG CO 646/4. M\"uller's research is supported by ERC-Grant 772606-PTRCSP}

\author{Amin Coja-Oghlan, No\"ela M\"uller, Jean B.~Ravelomanana}

\address{Amin Coja-Oghlan, {\tt acoghlan@math.uni-frankfurt.de}, Goethe University, Mathematics Institute, 10 Robert Mayer St, Frankfurt 60325, Germany.}
\address{No\"ela M\"uller, {\tt nmueller@math.lmu.de}, Ludwig-Maximilians-University, Mathematics Institute, 39 Theresienst, Munich 80333, Germany.}
\address{Jean B.~Ravelomanana, {\tt raveloma@math.uni-frankfurt.de}, Goethe University, Mathematics Institute, 10 Robert Mayer St, Frankfurt 60325, Germany.}	

\maketitle

\section{Introduction}\label{Sec_intro}

\subsection{Background and motivation}
According to a prominent physics prediction the Belief Propagation message passing algorithm renders a good approximation to the partition function of locally tree-like graphical models that do not exhibit long-range correlations~\cite{pnas}.
Turning this somewhat vague intuition into a mathematical theorem has been a major open problem at the junction of computer science and probability theory, specifically spin glass theory, for quite some time~\cite{MM}.
The random $k$-SAT model is one of the specific examples to which both communities have directed a large amount of effort~\cite{nae,nature,kSAT,DSS3,GGGY,MPZ,MS,PanchkSAT,PanchenkoTalagrand,Talagrand}.

Corroborating the physics conjecture, we prove that an extremely modest absence of long-range correlations condition known as ``replica symmetry'' precipitates the success of Belief Propagation for the random $k$-SAT model.
The replica symmetry condition is generally deemed to be necessary, too~\cite{pnas}.
Apart from significantly advancing the mathematical understanding of Belief Propagation, this result allows for an intriguing application.
Namely, by way of characterising the fixed point of the Belief Propagation message passing process precisely, we deduce that replica symmetry fails to hold for clause-to-variable densities near the satisfiability threshold.
Thus, we prove that the random $k$-SAT model undergoes a replica symmetry breaking phase transition for clause-to-variable ratios close to but below the satisfiability threshold.

To appraise ourselves of the random $k$-SAT model, let $V_n=\{x_1,\ldots,x_n\}$ be a set of $n$ Boolean variables. 
We represent their possible values `true' and `false' by $\pm1$.
Also let $k\geq3$ be an integer, let $d>0$ be a real and let $\vm$ be a Poisson variable with mean $dn/k$.
The random $k$-SAT formula comprises $\vm$ clauses $a_1,\ldots,a_{\vm}$.
For each clause $a_i$ we independently choose a family $(\vx_{ij})_{1\leq j\leq k}\in V_n^k$ of $k$ variables uniformly without replacement.
Additionally, let $(\vJ_{ij})_{i,j\geq1}$ be a family of independent $\pm1$-variables with mean zero.
Combinatorially $a_i$ represents the Boolean clause comprising the $k$ variables {$\vx_{i1},\ldots,\vx_{ik}$} with signs $\vJ_{i1},\ldots,\vJ_{ik}$. 
Thus, $\vx_{ij}$ appears as a positive literal in $a_i$ if $\vJ_{ij}=1$, and $a_i$ features the negative literal $\neg\vx_{ij}$ otherwise.
Hence, a Boolean assignment $\sigma\in\{\pm1\}^{V_n}$ satisfies clause $a_i$ (``$\sigma\models a_i$'') if {$\max_{j=1,\ldots,k}\vJ_{ij}\sigma_{\vx_{ij}}=1$}.
Finally, $\PHI=\PHI_k(n,\vm)$ is the conjunction of all the $\vm$ clauses, i.e.,
\begin{align*}
	\PHI=\bigwedge_{i=1}^{\vm}a_i=\bigwedge_{i=1}^{\vm}\bc{\vJ_{i1}\vx_{i1}\vee\cdots\vee\vJ_{ik}\vx_{ik}}.
\end{align*}

Further, given an inverse temperature parameter $\beta>0$, the Boltzmann distribution of the model reads
\begin{align}\label{eqBoltz}
	\mu_{\PHI,\beta}(\sigma)&=\frac1{Z(\PHI,\beta)}\prod_{i=1}^{\vm}\exp(-\beta\vecone\cbc{\sigma \not\models a_i})&&(\sigma\in\{\pm1\}^{V_n}),\quad\mbox{ where}&
	Z(\PHI,\beta)&=\sum_{\tau\in\cbc{\pm1}^{V_n}}\exp\bc{-\beta\sum_{i=1}^{\vm}\vecone\cbc{\tau \not\models a_i}}.
\end{align}
Thus, the Boltzmann weight of an assignment $\sigma$ contains an $\exp(-\beta)$ penalty factor for every violated clause.
In effect, as $\beta$ increases, the distribution assigns greater weight to `more satisfying' assignments.
As always, the partition function $Z(\PHI,\beta)$ accounts for the total weight.

The random $k$-SAT model undergoes a satisfiability phase transition at a certain critical value of $d$ called the {\em satisfiability threshold}.
To elaborate, observe that $d$ gauges the average number of clauses in which a given Boolean variable appears.
Clearly, as variables appear in more and more clauses it becomes harder to satisfy all these clauses simultaneously.
Indeed, for large enough $k\geq3$ there exists a threshold $\dk(k)$ such that $\PHI$ admits an assignment that satisfies all clauses asymptotically almost surely if $d<\dk(k)$, while for $d>\dk(k)$ no satisfying assignment exists \whp\
The value of $\dk(k)$ is known precisely {but} the formula is quite complicated~\cite{DSS3}; asymptotically in the limit of large $k$ we have
\begin{align}\label{eqkSAT}
	\dk(k)&=2^kk\ln2-\frac{1+\log2}{2}k+o(1).
\end{align}

The regime $d<\dk(k)$ is of fundamental interest in computer science to assess the power and the limitations of algorithms for finding, counting and sampling solutions to the $k$-SAT problem, the cornerstone of computational complexity theory~\cite{nature}.
Therefore, we will investigate the Boltzmann distribution for $d<\dk(k)$ for varying values of $\beta$.
As we increase $\beta$ we effectively scan the energy landscape that an algorithm has to traverse on its quest for satisfying assignments.
In particular, we investigate the performance of the Belief Propagation message passing algorithm.
With what regimes of $d,\beta$ can the algorithm cope?
Is Belief Propagation fit to approximate the partition function $Z(\PHI,\beta)$?
Does there exist a critical value of $\beta$ where long-range correlations emerge?


\subsection{Belief Propagation}
Belief Propagation associates two `messages' $\mu_{\PHI,\beta,a_i\to\vx_{ij},t}(\pm1),\mu_{\PHI,\beta,\vx_{ij}\to a_i,t}(\pm1)\in(0,1)$ with each interacting clause/variable pair $(a_i,\vx_{ij})$. 
The messages are indexed by time $t\geq0$ and always normalised such that
\begin{align}\label{eqmsgnorm}
	\mu_{\PHI,\beta,a_i\to\vx_{ij},t}(1)+\mu_{\PHI,\beta,a_i\to\vx_{ij},t}(-1)=\mu_{\PHI,\beta,\vx_{ij}\to a_i,t}(1)+\mu_{\PHI,\beta,\vx_{ij}\to a_i,t}(-1)=1.
\end{align}
The first message $\mu_{\PHI,\beta,a_i\to\vx_{ij},t}(\pm1)$ is directed from the clause to the variable.
The other one travels in the reverse direction.
The messages are updated iteratively.
Initially, all messages are set to $1/2$, i.e.,
\begin{align}\label{eqBPinit}
	\mu_{\PHI,\beta,a_i\to\vx_{ij},0}(\pm1)=\mu_{\PHI,\beta,\vx_{ij}\to a_i,0}(\pm1)&=1/2&&\mbox{for all }1\leq i\leq\vm,1\leq j\leq k.
\end{align}
Furthermore, for integers $t\geq0$ and $s=\pm1$ we inductively define
\begin{align}\label{eqBP1}
	\mu_{\PHI,\beta,a_i\to \vx_{ij},t+1}(s)&\propto\sum_{\sigma\in\{\pm1\}^{k}}\vecone\{\sigma_{j}=s\}\exp(-\beta\vecone\cbc{\sigma\not\models a_i})\prod_{\substack{1\leq h\leq k\\h\neq j}}\mu_{\PHI,\beta,\vx_{ih}\to a_i,t}(\sigma_h),\\
	\mu_{\PHI,\beta,\vx_{ij}\to a_i,t+1}(s)&\propto\prod_{\substack{1\leq h\leq\vm\\h\neq i}}\prod_{\substack{1\leq\ell\leq k\\\vx_{h\ell}=\vx_{ij}}} \mu_{\PHI,\beta,a_h\to \vx_{h\ell},t+1}(s).
\label{eqBP2}
\end{align}
Here the $\propto$-symbol hides the normalisation required to bring about $\eqref{eqmsgnorm}$. 
Finally, the estimate of the partition function after $t$ iterations reads
\begin{align}\label{eqBFE}
	\cB_t&=\sum_{i=1}^n\log\Bigg[{\sum_{s=\pm1}\prod_{\substack{1\leq h\leq\vm,1\leq j\leq k\\\vx_{hj}=x_i}}\mu_{\PHI,\beta,a_h\to x_i,t}(s)}\Bigg]
	+\sum_{i=1}^{\vm}\log\brk{\sum_{\sigma\in\cbc{\pm1}^{k}}\eul^{-\beta\vecone\cbc{\sigma\not\models a_i}}\prod_{j=1}^k\mu_{\PHI,\beta,\vx_{ij}\to a_i,t}(\sigma_j)}\nonumber\\
		 &\qquad-\sum_{i=1}^{\vm}\sum_{j=1}^k\log\brk{\sum_{s=\pm1}\mu_{\PHI,\beta,a_i\to \vx_{ij},t}(s)\mu_{\PHI,\beta,\vx_{ij}\to a_i,t}(s)}.
\end{align}
This expression is  called the {\em Bethe free energy} in physics jargon.
An excellent in-depth discussion of Belief Propagation, including a derivation of \eqref{eqBP1}--\eqref{eqBFE}, can be found in~\cite[\Chap~14]{MM}.

The key feature of all the above formulas is that they are governed by the {\em local} structure of the $k$-SAT formula.
For instance, \eqref{eqBP1} involves only the messages sent out by the variables which appear in clause $a_i$.
Similarly, \eqref{eqBP2} comes in terms of the messages sent out by the clauses in which variable $\vx_{ij}$ appears.
Therefore, we can reasonably hope that Belief Propagation represents local dependencies accurately, but hardly that the messages can faithfully capture long-range correlations.
In fact, one of the most important predictions about the random $k$-SAT model holds that a very weak `absence of long-range correlations' condition suffices for the success of Belief Propagation~\cite{pnas}.
Specifically, let $\SIGMA=\SIGMA_{\PHI,\beta}$ denote a sample from the Boltzmann distribution $\mu_{\PHI,\beta}$.
Then following~\cite{pnas} we say that the random $k$-SAT model with parameters $d,\beta$ is {\em replica symmetric} if
\begin{align}\label{eqRS}
\lim_{n\to\infty}\Erw\abs{\mu_{\PHI,\beta}(\{\SIGMA_{x_1}=\SIGMA_{x_2}=1\})-\mu_{\PHI,\beta}(\{\SIGMA_{x_1}=1\})\mu_{\PHI,\beta}(\{\SIGMA_{x_2}=1\})}=0.
\end{align}
In words, the events $\{\SIGMA_{x_1}=1\}$, $\{\SIGMA_{x_2}=1\}$ that the first and the second variable of the formula $\PHI$ are set to `true' are asymptotically independent for large $n$.
Since the typical distance of $x_1,x_2$ is of order $\Omega(\log n)$, \eqref{eqRS} rules out long-range correlations, albeit in a very weak sense.
In particular, \eqref{eqRS} is far more modest a condition than classical spatial mixing properties such as Gibbs uniqueness or non-reconstruction~\cite{Charis,pnas}.

The following theorem vindicates the prediction that \eqref{eqRS} is a sufficient condition for the success of Belief Propagation for all $\beta\geq1$ and for all $d$ up to within a whisker of the satisfiability threshold $\dk(k)$.

\begin{theorem}\label{Thm_main}
	There exists a constant $k_0\geq3$ such that for any $\beta\geq1$ and any 
\begin{align}\label{eqd*}
	d\leq d^*=d^*(k)=k2^k\ln2-10k^2
\end{align}
the following is true: if \eqref{eqRS} is satisfied then
\begin{align*}
\lim_{t\to\infty}\limsup_{n\to\infty}\frac1n\Erw\abs{\cB_t-\log Z(\PHI,\beta)}=
0.
\end{align*}
\end{theorem}

\Thm~\ref{Thm_main} is a conditional result: for all $d\leq d^*$ and all $\beta\geq1$ for large enough $t$ the Bethe free energy formula~\eqref{eqBFE} estimates the logarithm of the partition function up to an additive error of $o(n)$ \emph{provided that} \eqref{eqRS} holds.
At this point \eqref{eqRS} is known to be satisfied only for values of $d$ much smaller than $d^*$; the best current bound yields $d\leq\log k$~\cite{MS}.
However, \eqref{eqRS} is expected to hold for all $\beta>0$ and all $d\leq d^*$ (and in fact for slightly larger $d$)~\cite{pnas}.
Yet conceptually the point that \Thm~\ref{Thm_main} makes is that the modest condition \eqref{eqRS} is the \emph{only} requirement for the success of Belief Propagation.
In other words, Belief Propagation launched from the trivial initial condition \eqref{eqBPinit} does faithfully capture the short-range effects of the random $k$-SAT model.
A further strength of \Thm~\ref{Thm_main} is that the result covers all reasonable values of $\beta$.
Indeed, the assumption $\beta\geq1$ is harmless as the most interesting regime should be that of large $\beta$, where the satisfiability condition really bites, known as the `low temperature' regime in physics terminology.

\subsection{Replica symmetry breaking}\label{Sec_rsb_intro}
The proof of \Thm~\ref{Thm_main} has an unconditional consequence.
Namely, we can turn the tables and prove that \eqref{eqRS} fails to be satisfied for $d$ close to the satisfiability threshold $\dk(k)$.

\begin{theorem}\label{Thm_rsb}
	There exist sequences $\eps_k\to0$ and $\beta_0(k)>0$ such that the following is true.
	Assume that $\beta>\beta_0(k)$ and
	\begin{align}\label{eqThm_rsb}
		2^kk\log2-k(3+\eps_k)\log2/2\leq d\leq\dsat.
	\end{align}
	Then 
	\begin{align}
		\limsup_{n\to\infty}\Erw\abs{\mu_{\PHI,\beta}(\{\SIGMA_{x_1}=\SIGMA_{x_2}=1\})-\mu_{\PHI,\beta}(\{\SIGMA_{x_1}=1\})\mu_{\PHI,\beta}(\{\SIGMA_{x_2}=1\})}>0&&&\mbox{and}\label{eqThm_rsb1}\\
		\liminf_{n\to\infty}\frac{1}{n}\Erw\brk{\cB_t-\log Z(\PHI,\beta)}>0&&&\mbox{uniformly for all }t>0.\label{eqThm_rsb2}
	\end{align}
\end{theorem}

The asymptotic value $2^kk\log2-3k\log2/2$ from \eqref{eqThm_rsb} was predicted via physics methods as the threshold for replica symmetry to break~\cite{pnas}.
Thus, \Thm~\ref{Thm_rsb} confirms this conjecture.
Indeed, \eqref{eqThm_rsb1} shows that very strong long-range correlations start to emerge in the `low temperature' (viz.\ large $\beta$) regime.
Together with known results on the structure of asymptotic Gibbs measures, \eqref{eqThm_rsb1} shows that the Boltzmann distribution decomposes into several `pure states' in a certain precise sense; see~\cite{victor,Will2} for a detailed discussion.
Additionally, \eqref{eqThm_rsb2} implies that beyond~\eqref{eqThm_rsb} Belief Propagation ceases to yield a good approximation to the partition function.

\Thm~\ref{Thm_rsb} also sheds new light on the satisfiability threshold.
Namely, \Thm~\ref{Thm_rsb} establishes for the first time that replica symmetry breaking occurs in the random $k$-SAT problem strictly prior to the $k$-SAT threshold from \eqref{eqkSAT}, which exceeds the bound from \eqref{eqThm_rsb} by an additive $\log(2)-1/2\approx0.19$.
Hence, \Thm~\ref{Thm_rsb} demonstrates that the random $k$-SAT model really is conceptually richer than models like random $k$-XORSAT or random $2$-SAT, whose satisfiability thresholds were found much earlier\cite{CR,Cuckoo,DuboisMandler,Goerdt,PittelSorkin}.

We proceed to outline the proof strategy behind \Thm s~\ref{Thm_main} and~\ref{Thm_rsb}.
Subsequently we discuss how the contributions of this paper compare to prior work.

\section{Overview}\label{Sec_overview}

\noindent
The proof of \Thm~\ref{Thm_main} has three basic ingredients.
First we need a rough estimate of $Z(\PHI,\beta)$, which we derive via a subtle second moment calculation.
From this estimate we will deduce that `most' variable marginals under the Boltzmann distribution $\mu_{\PHI,\beta}$ are close to $1/2$ \whp{}
Second, we investigate the Belief Propagation message passing scheme on a random Galton-Watson tree that mimics the local geometry of the random $k$-SAT formula $\PHI$.
Specifically, we will use contraction arguments to show that if Belief Propagation launches from messages that are mostly close to $1/2$, the message passing scheme will rapidly approach a fixed point.
Third, we will combine these two facts with probabilistic invariance properties of the random formula $\PHI$ to complete the proof of \Thm~\ref{Thm_main}.

The proof of \Thm~\ref{Thm_rsb} is an afterthought to the proof of the first theorem.
Indeed, the proof of \Thm~\ref{Thm_main} renders an implicit formula for the value that $Z(\PHI,\beta)$ must take if \eqref{eqRS} is satisfied.
To obtain \Thm~\ref{Thm_rsb} we calculate this value explicitly for large $\beta$.
To refute \eqref{eqRS} we then compare this result with the upper bound that the interpolation method from mathematical physics yields.

\subsection{The second moment bound}\label{Sec_smm_overview}
A natural first stab at estimating $Z(\PHI,\beta)$ is to calculate its first two moments.
The first moment is easy.
Indeed, because any specific assignment $\sigma\in\{\pm1\}^{V_n}$ satisfies a random clause with probability $1-2^{-k}$ and because the clauses are independent, the linearity of expectation gives
\begin{align}\label{eqFirstMmt}
\log\Erw[Z(\PHI,\beta)\mid\vm]&=n\ln2+\vm\log\bc{1-2^{-k}(1-\eul^{-\beta})}.
\end{align}
Hence, Markov's inequality immediately implies that $\log Z(\PHI,\beta)\leq n\log2+\frac{dn}{k}\log(1-(1-\eul^{-\beta})2^{-k})+o(n)$.

Moving on to the second moment and using the linearity of expectation and independence once more, we find
\begin{align}
\Erw[Z(\PHI,\beta)^2\mid\vm]&=\sum_{\sigma,\tau\in\{\pm1\}^{V_n}}\Erw\brk{\eul^{-\beta\sum_{i=1}^{\vm}\vecone\cbc{\sigma \not\models a_i}+\vecone\cbc{\tau\not\models a_i}}\mid\vm}=\sum_{\sigma,\tau\in\{\pm1\}^{V_n}}\Erw\brk{\eul^{-\beta\bc{\vecone\cbc{\sigma\not\models a_1}+\vecone\cbc{\tau \not\models a_1}}}}^{\vm}.\label{eqsmm_deriv1}
\end{align}
To evaluate the r.h.s.\ we define the {\em overlap} of two assignments $\sigma,\tau\in\{\pm1\}^{V_n}$ as
\begin{align}\label{eqoverlap}
\alpha(\sigma,\tau)&=\frac1n\sum_{i=1}^n\frac{1+\sigma_{x_i}\tau_{x_i}}{2}=\frac{1}{n}\sum_{i=1}^n\vecone\cbc{\sigma_{x_i}=\tau_{x_i}}.
\end{align}
A straightforward application of inclusion/exclusion then reveals that
\begin{align}\label{eqsmm_deriv2}
\Erw\brk{\eul^{-\beta\bc{\vecone\cbc{\sigma\not\models a_1}+\vecone\cbc{\tau \not\models a_1}}}}&=
1-2^{1-k}(1-\eul^{-\beta})+2^{-k}\alpha(\sigma,\tau)^k(1-\eul^{-\beta})^2+O(1/n). 
\end{align}
Hence, it seems like a good idea to reorder the sum \eqref{eqsmm_deriv1} according to the overlap.
We thus sum on $\alpha\in[0,1]$ such that $\alpha n$ is an integer. 
Since there are $2^n\binom n{\alpha n}$ pairs $\sigma,\tau$ with overlap $\alpha$, \eqref{eqsmm_deriv2} yields 
\begin{align}
\Erw[Z(\PHI,\beta)^2\mid\vm]
&=\exp(O(\vm/n))\cdot 2^n\sum_{\alpha}\bink{n}{\alpha n}\brk{1-2^{1-k}(1-\eul^{-\beta})+2^{-k}\alpha(\sigma,\tau)^k(1-\eul^{-\beta})^2}^{\vm}.\label{eqsmm_deriv3}
\end{align}
Taking logarithms in \eqref{eqsmm_deriv3}, assuming $\vm=dn/k + o(n)$ and replacing the sum by a max, we obtain
\begin{align}\label{eqSecondMmt}
	\frac1n\log\Erw[Z(\PHI,\beta)^2\mid\vm]&=\max_{\alpha\in(0,1)}f(\alpha)+o(1),\qquad\mbox{\whp, where}\\
	f(\alpha)&=f_{d,k,\beta}(\alpha)=\ln2-\alpha\ln\alpha-(1-\alpha)\ln(1-\alpha)+\frac{d}{k}\log\bc{1-2^{1-k}(1-\eul^{-\beta})+2^{-k}\alpha^k(1-\eul^{-\beta})^2}.\nonumber
\end{align}
Thus, the maximiser $\alpha$ in \eqref{eqSecondMmt} represents the overlap value that renders the dominant contribution to the second moment.
Since the entropy function $-\alpha\ln\alpha-(1-\alpha)\ln(1-\alpha)$ attains its maximum at $\alpha=1/2$ while the second term $\log\bc{1-2^{1-k}(1-\eul^{-\beta})+2^{-k}\alpha^k(1-\eul^{-\beta})^2}$ is strictly increasing in $\alpha$, the dominant overlap value inevitably exceeds $1/2$.
In effect, since  $f(1/2)$ equals twice the r.h.s.\ of \eqref{eqFirstMmt} and $\max_\alpha f(\alpha)>f(1/2)$, the second moment $\Erw[Z(\PHI,\beta)^2]$ exceeds the square $\Erw[Z(\PHI,\beta)]^2$ of the first moment by an exponential factor for all $d,\beta>0$.
While it is still possible to salvage {\em some} estimate of $\log Z(\PHI,\beta)$ from \eqref{eqFirstMmt}--\eqref{eqSecondMmt}, its quality deteriorates rapidly as $\beta$ increases.
In the extreme case $\beta=\infty$ of `hard' constraints this issue was already highlighted in the seminal work~\cite{nae} where Achlioptas and Moore pioneered the second moment method for random $k$-SAT.

To remedy this problem we take a leaf out of earlier work on `hard' random $k$-SAT~\cite{yuval,kSAT}.
Instead of applying the second moment method directly to $Z(\PHI,\beta)$, we consider a suitably truncated random variable.
Its second moment is asymptotically bounded by the square of the first moment and we obtain the following explicit lower bound.

\begin{proposition}\label{Prop_bal}
	Let $\beta\geq1$, $k\geq k_0$ and $d<d^*$ and let $p\in(0,1)$ be the unique root of
\begin{align} \label{maximal_p}
	1-2p-(1-\eul^{-\beta})(1-p)^k=0;&&\mbox{then}\\
\label{eqProp_bal}
\liminf_{n\to\infty}\frac1n\Erw[\log Z(\PHI,\beta)]\geq
\bc{1- \frac{(k-1)d}{k}}\log 2 -\frac{d}{2}\log p - \frac{d}{2}\log(1-p) +  \frac{d}{k} \log p.
\end{align}
\end{proposition}


We apply \Prop~\ref{Prop_bal} to estimate the Boltzmann marginals.
More precisely, recalling that $\SIGMA$ signifies a sample from $\mu_{\PHI,\beta}$, we care to learn the marginal probabilities $\mu_{\PHI,\beta}(\{\SIGMA_{x_i}=1\})$ that specific variables $x_i$ take the value `true'.
Hence, with $\delta_z$ denoting the probability measure on $\RR$ that places mass one on the number $z$, let
\begin{align*}
\pi_{\PHI,\beta}&=\frac1n\sum_{i=1}^n\delta_{\mu_{\PHI,\beta}(\{\SIGMA_{x_i}=1\})}\in\cP(0,1) 
\end{align*}
be the empirical distribution of these marginals.
We say that a probability measure $\pi$ on $[0,1]$ has {\em slim tails} if
\begin{align} \label{xinitial_cond1}
\pi\bc{\brk{0,\frac12-2^{-k/10}}\cup \brk{\frac12+2^{-k/10},1}}\leq 2^{-k/10}.
\end{align}
Additionally, $\pi$ has {\em very slim tails} if \eqref{xinitial_cond1} holds with the r.h.s.\ replaced by $2^{-k/9}$.

\begin{corollary}\label{Cor_bal}
Suppose that $\beta\geq1$, $k\geq k_0$ and $d<d^*$ and that \eqref{eqRS} is satisfied.
Then $\pi_{\PHI,\beta}$ has very slim tails \whp 
\end{corollary}

\noindent
The proofs of \Prop~\ref{Prop_bal} and \Cor~\ref{Cor_bal} can be found in Section~\ref{Sec_Prop_bal}.

\subsection{Belief Propagation on trees}
As a next step we analyse Belief Propagation on a Galton-Watson tree that mimics the local structure of the random formula $\PHI$.
To elaborate, we can represent $\PHI$ by a bipartite graph $G(\PHI)$ known as the {\em factor graph}.
One class of vertices comprises the variables $x_1,\ldots,x_n$.
The second class of vertices consists of the clauses $a_1,\ldots,a_{\vm}$.
A clause $a_i$ and a variable $x_j$ are connected by an edge if $x_j$ appears in $a_i$.
For a variable $x_j$ we denote by $\partial x_j$ the set of adjacent clauses.
Moreover, to keep track of the order as variables appear in clauses we write $\partial_ha_i$ for the $h$-th variable in clause $a_i$ and $\partial a_i$ for the set of all variables that occur in $a_i$.
Finally, for an adjacent clause/variable pair $(a,x)$ we let $\vJ_{ax}=\sign(a,x)\in\{\pm1\}$ signify the sign with which $x$ appears in $a$.

The graph $G(\PHI)$ induces a metric on the set of variables and clauses.
Moreover, it is well known that $G(\PHI)$ contains only a small number of, say, $o(\log n)$ cycles of bounded length.
Hence, for any specific variable $x_i$ and for any fixed radius $t>0$ the depth-$t$ neighbourhood of $x_i$ in $G(\PHI)$ is a tree \whp\
The distribution of this tree can be characterised precisely by a two-type Galton-Watson process.
The two types are variables and clauses, of course.
The process starts from a single root variable $x_0$.
Moreover, the offspring of a variable is a $\Po(d)$ number of clauses.
Furthermore, a clause begets $k-1$ variables.
Let $\vT$ signify the resulting (quite possibly infinite) tree.
Also let $V(\vT),C(\vT)$ be the sets of variables and clauses of $\vT$, respectively.
As in the case of the random formula $\PHI$ we use the $\partial$-symbol to denote adjacencies.
Finally, to turn $\vT$ into a $k$-SAT formula, we choose for each adjacent clause/variable pair $(a,x)\in C(\vT)\times V(\vT)$ a sign $\vJ_{ax}\in\{\pm1\}$ uniformly and independently.

It is well known that the graph $G(\PHI)$ converges locally to the random tree $\vT$ in the sense that for any specific variable node $x_i$, $1\leq i\leq n$, and for any fixed radius $t$ the depth-$t$ neighbourhood of $x_i$ and the depth-$t$ neighbourhood of the root $x_0$ of $\vT$ can be coupled such that both coincide \whp\
Therefore, in order to investigate the first $t$ rounds of Belief Propagation $\PHI$ as per \eqref{eqBPinit}--\eqref{eqBP2}, we just need to study Belief Propagation on $\vT$.

Hence, we proceed to define Belief Propagation messages on $\vT$.
Generalising \eqref{eqBPinit}, we allow for an arbitrary probability distribution $\pi$ on $[0,1]$ from which we draw the initial messages.
Thus, for any adjacent $a,x$ we draw $\mu_{\vT,\beta,\pi,x\to a,0}(1),\mu_{\vT,\beta,\pi,a\to x,0}(1)$ independently from $\pi$ and set
\begin{align*}
	\mu_{\vT,\beta,\pi,x\to a,0}(-1)&=1-\mu_{\vT,\beta,\pi,x\to a,0}(1),&\mu_{\vT,\beta,\pi,a\to x,0}(-1)&=1-\mu_{\vT,\beta,\pi,a\to x,0}(1).
\end{align*}
Further, for $t\geq0$, $s=\pm1$ and adjacent $a,x$ we inductively define
\begin{align}
\mu_{\vT,\beta,\pi,a\to x,t+1}(s)&\propto\sum_{\sigma\in\{\pm1\}^{\partial a}}\vecone\cbc{\sigma_{x}=s}\eul^{-\beta\vecone\cbc{\sigma \not\models a}}\prod_{y\in\partial a\setminus\cbc x}\mu_{\vT,\beta,\pi,y\to a,t}(\sigma_y),\label{eqTreeBP1}\\
\mu_{\vT,\beta,\pi,x\to a,t+1}(s)&\propto\prod_{b\in\partial x\setminus\cbc a}\mu_{\vT,\beta,\pi,b\to x,t+1}(s).\label{eqTreeBP2}
\end{align}
Finally, the Belief Propagation estimate of the marginal of $x_0$ after $t+1$ rounds reads
\begin{align}\label{eqTreeBP3}
	\mu_{\vT,\beta,\pi,x_0,t+1}(s)&\propto\prod_{b\in\partial x_0}\mu_{\vT,\beta,\pi,b\to x_0,t+1}(s)&&(s=\pm1).
\end{align}
Let $\pi_0=\delta_{1/2}$ be the probability distribution on $(0,1)$ that places all mass on $1/2$.
The following proposition shows that in the limit of large $t$, any distribution with slim tails yields the same Belief Propagation marginal as $\pi_0$.

\begin{proposition}\label{Prop_Noela}
	Assume that $d\leq\dk(k)$ and $\beta\geq1$.
	Then uniformly for all $\pi$ with slim tails we have
$$\lim_{t\to\infty}\Erw\abs{\mu_{\vT,\beta,\pi,x_0,t}(1)-\mu_{\vT,\beta,\pi_0,x_0,t}(1)}=0.$$
Furthermore, the sequence $(\mu_{\vT,\beta,\pi_0,x_0,t}(1))_{t\geq1}$ converges weakly to a probability measure $\pi_{d,\beta}^\star$ with slim tails.
\end{proposition}


\noindent
The proof of \Prop~\ref{Prop_Noela} can be found in \Sec~\ref{Sec_Prop_Noela}.

\subsection{The Bethe free energy}
It is known that under the replica symmetry assumption \eqref{eqRS} the partition function $Z(\PHI,\beta)$ can be approximated well in terms of certain `pseudo-messages'.
To be precise, consider a clause $a_i$ and a variable $\vx_{ij}$ that appears in it.
Then we define the pseudo-message $\mu_{\PHI,\beta,\vx_{ij}\to a_i}$ as the Boltzmann marginal of $\vx_{ij}$ in the formula $\PHI-a_i$ obtained by deleting clause $a_i$.
Thus, $\mu_{\PHI,\beta,\vx_{ij}\to a_i}(\pm1)=\mu_{\PHI-a_i,\beta}(\{\SIGMA_{\vx_{ij}}=\pm1\})$.
Similarly, we define the reverse message $\mu_{\PHI,\beta,a_i\to\vx_{ij}}$ as the marginal of $\vx_{ij}$ in the formula obtained from $\PHI$ by deleting all clauses in which the variable $\vx_{ij}$ appears apart from $a_i$.
In symbols,
\begin{align*}
	\mu_{\PHI,\beta,a_i\to\vx_{ij}}(s)&=\mu_{\PHI-(\partial\vx_{ij}\setminus\{a_i\}),\beta}(\{\SIGMA_{\vx_{ij}}=s\})&&(s=\pm1).
\end{align*}
A result about general random factor graph models from \cite{Will} implies that the pseudo-messages yield the following approximation to the partition function if \eqref{eqRS} is satisfied.

\begin{lemma}[{\cite[\Cor~1.2]{Will}}]\label{Cor_Amin}
Let
\begin{align}\nonumber
\cB(\PHI,\beta)&=
\sum_{i=1}^n\log\brk{\sum_{\sigma=\pm1}\prod_{a\in\partial x_i}\mu_{\PHI,\beta,a\to x_i}(\sigma)}
+\sum_{i=1}^{\vm}\log\brk{\sum_{\sigma\in\cbc{\pm1}^{\partial a_i}}\exp(-\beta\vecone\cbc{\sigma\not\models a_i})\prod_{x\in\partial a_i}\mu_{\PHI,\beta,x\to a_i}(\sigma_x)}\\
&\qquad-\sum_{i=1}^n\sum_{a\in\partial x_i}\log\brk{\sum_{\sigma=\pm1}\mu_{\PHI,\beta,x_i\to a}(\sigma)\mu_{\PHI,\beta,a\to x_i}(\sigma)}.\label{eqrealBFE}
\end{align}
If \eqref{eqRS} holds and $\lim_{n\to\infty}\cB(\PHI,\beta)/n=b\in\RR$ in probability, then $\lim_{n\to\infty}\frac1n\log Z(\PHI,\beta)=b$ in probability.
\end{lemma}


Apart from the formula \eqref{eqrealBFE} it is known that the pseudo-messages form an approximate fixed point of the Belief Propagation recurrence \eqref{eqBP1}--\eqref{eqBP2} if \eqref{eqRS} is satisfied~\cite[\Thm~1.1]{Will}.
In light of the contraction property of the Belief Propagation iteration that \Prop~\ref{Prop_Noela} provides it is therefore tempting to think that the messages obtained after a large enough number $t$ of iterations of \eqref{eqBPinit}--\eqref{eqBP2} should be about the same as the pseudo-messages.
Yet this conclusion is anything but immediate.
First, \Prop~\ref{Prop_Noela} establishes contraction only under the assumption that Belief Propagation launches from a set of {\em independent} initial messages.
Second, the proposition requires that the distribution of these initial messages has slim tails.
Thus, for all we know the Belief Propagation equations on $\PHI$ could have many fixed points that do not fall into the basin of attraction of the all--$\frac{1}{2}$ initialisation \eqref{eqBPinit}.
Nonetheless, combining a subtle coupling argument with the tail bound for $\pi_{\PHI,\beta}$ from \Cor~\ref{Cor_bal} we can link the pseudo-message with the Belief Propagation fixed point that we approach from the naive initialisation \eqref{eqBPinit}.
We can thus relate the pseudo-messages and the real ones as follows.

\begin{proposition}\label{Prop_Amin}
	If \eqref{eqRS} is satisfied, then 
\begin{align*}
	\lim_{t\to\infty}\limsup_{n\to\infty}\frac1n\Erw\brk{\sum_{i=1}^n\sum_{a\in\partial x_i}\abs{\mu_{\PHI,\beta,x_i\to a}(1)-\mu_{\PHI,\beta,x_i\to a,t}(1)}+\abs{\mu_{\PHI,\beta,a\to x_i}(1)-\mu_{\PHI,\beta,a\to x_i,t}(1)}}=0.
\end{align*}
\end{proposition}

Equipped with \Lem~\ref{Cor_Amin} and \Prop~\ref{Prop_Amin} we have only very little work left to complete the proof of \Thm~\ref{Thm_main}.
Indeed, \Lem~\ref{Cor_Amin} shows how $\cB(\PHI,\beta)$ approximates $\log Z(\PHI,\beta)$.
Moreover, the only difference between $\cB(\PHI,\beta)$ and the expression $\cB_t(\PHI,\beta)$ that appears in \Thm~\ref{Thm_main} is that the latter comes in terms of the Belief Propagation messages $\mu_{\PHI,\beta,x\to a,t},\mu_{\PHI,\beta,a\to x,t}$ rather than the actual standard messages $\mu_{\PHI,\beta,x\to a},\mu_{\PHI,\beta,a\to x}$.
But \Prop~\ref{Prop_Amin} shows that the Belief Propagation messages approximate the standard messages well.
Hence,  we are left to show that the approximation is good enough and that the Bethe free energy is sufficiently continuous.
We will carry these steps out in \Sec~\ref{Sec_Amin}, thereby completing the proof of \Thm~\ref{Thm_main}.

\subsection{Replica symmetry breaking}
The proof of \Thm~\ref{Thm_rsb} consists of two parts: an unconditional upper bound on $\ex[\log Z(\PHI,\beta)]$ and a lower bound that is conditional on the assumption \eqref{eqRS} of replica symmetry.
To state these bounds we define a functional $\fB_{d,\beta}$ on the space of probability measures on the unit interval.
This functional can be viewed as the $n\to\infty$ limit of the Bethe free energy functional from \eqref{eqBFE}.

Hence, let $\pi$ be a probability measure on $[0,1]$.
Let $(\vec\rho_{\pi,i,j})_{i,j\geq1}$ be an array of independent random variables with distribution $\pi$.
Furthermore, let $(\vJ_{i,j})_{i,j\geq1}$ be an array of Rademacher variables with mean zero, mutually independent and independent of the $\vec\rho_{\pi,i,j}$.
Additionally, let
\begin{align} \label{eqmupij}
\vec\mu_{\pi,i,j}&=\frac{1+\vJ_{i,j}(2\vec\rho_{\pi,i,j}-1)}2=\begin{cases} \vec\rho_{\pi,i,j}&\mbox{ if }\vJ_{i,j}=1\\ 1-\vec\rho_{\pi,i,j}&\mbox{ if }\vJ_{i,j}=-1 \end{cases}\qquad.
\end{align}
Further, let $\vec\gamma^+,\vec\gamma^-$ be two $\Po(d/2)$ variables, mutually independent and independent of everything else.
We define
\begin{align}\label{eqBetheFunctional}
	\fB_{d,\beta}(\pi)&=\ex\brk{\log\bc{\prod_{i=1}^{\vec\gamma^+}1-(1-\eul^{-\beta})\prod_{j=1}^{k-1}\vec\mu_{\pi,i,j}+\prod_{i=1}^{\vec\gamma^-}1-(1-\eul^{-\beta})\prod_{j=1}^{k-1}\vec\mu_{\pi,i+\vec\gamma^-,j}}-\frac{d(k-1)}{k}\log\bc{1-(1-\eul^{-\beta})\prod_{j=1}^k\vec\mu_{\pi,1,j}}}.
\end{align}
Using the so-called 1-step replica symmetry breaking interpolation method from mathematical physics, we obtain the following upper bound.
Recall $\pi_{d,\beta}^\star$ from \Prop~\ref{Prop_Noela}.

\begin{proposition}\label{Prop_rsbint}
	Assume that $d$ satisfies \eqref{eqThm_rsb} and that $\beta>\beta_0(k)$ for a large enough $\beta_0(k)$.
	Then $$\lim_{n\to\infty}\frac{1}{n}\Erw\brk{\log Z(\PHI,\beta)}<\fB_{d,\beta}(\pi_{d,\beta}^\star).$$
\end{proposition}

\noindent
We remark that the limit on the l.h.s.\ is known to exist~\cite{Panchenko}.

On the other hand, a careful study facilitated by the ideas from the proof of \Prop~\ref{Prop_bal} and by \Prop~\ref{Prop_Noela} shows that under the assumption \eqref{eqRS} even in the regime $d^*(k)<d<\dk(k)$ the only conceivable scenario is that the actual pseudo-messages nearly coincide with the messages that result from the fixed point iteration \eqref{eqBPinit}--\eqref{eqBP2}.
Combining this fact with \Lem~\ref{Cor_Amin}, we obtain the following lower bound.

\begin{proposition}\label{Prop_rsint}
	Assume that $d$ satisfies \eqref{eqThm_rsb} and that $\beta\geq\beta_0(k)$ for a large enough $\beta_0(k)$.
	If \eqref{eqRS} holds, then $$\lim_{n\to\infty}\frac{1}{n}\Erw\brk{\log Z(\PHI,\beta)}\geq \fB_{d,\beta}(\pi_{d,\beta}^\star).$$ 
\end{proposition}

Naturally, \Prop s~\ref{Prop_rsbint} and~\ref{Prop_rsint} imply that \eqref{eqRS} cannot be satisfied under the assumptions of \Thm~\ref{Thm_rsb}.
The detailed proofs of the propositions as well as of \Thm~\ref{Thm_rsb} can be found in \Sec~\ref{Sec_rsb}.

\section{Discussion}\label{Sec_related}

\noindent
Early experimental work~\cite{Cheeseman} inspired the hunt for the satisfiability threshold, a pursuit conducted by many authors over many years (e.g.~\cite{nature,yuval,AchSorkin,Broder,Chao,kSAT,Ehud,FS,FriezeWormald}) and which culminated in the aforementioned work of Ding, Sly and Sun~\cite{DSS3}.
Prior to the seminal work of Achlioptas and Moore~\cite{nae}, lower bounds on the $k$-SAT threshold were based on the analysis of simple satisfiability algorithms~\cite{AchSorkin,Broder,Chao,FS}.
However, all known algorithms fail to find satisfying assignments efficiently for densities well below the satisfiability threshold.
The best algorithmic result today reaches up to $d\sim 2^k\log k$ \cite{BetterAlg}, undershooting $\dk(k)$ by almost a factor of $k$. 
Even satisfiability algorithms that employ message passing techniques such as Belief Propagation Guided Decimation fail to beat this bound~\cite{BPDec,Samuel}.
Furthermore, the current algorithmic threshold of $2^k\log k$ marks the onset of combinatorial effects that conceivably stymie various algorithmic techniques~\cite{Barriers}.

Matters get worse when it comes to algorithms for counting or sampling satisfying assignments.
From a worst-case viewpoint  this task is the epitome of the complexity class $\#\mathrm P$, the counterpart of the notorious complexity class NP in the realm of counting and sampling~\cite{Valiant}.
Hence, we expect that counting or sampling satisfying assignments is conceptually far more challenging than `merely' finding one. 
In the context of random $k$-SAT a first important contribution is due to Montanari and Shah~\cite{MS}, who showed that Belief Propagation does the trick up to $d\sim\log k$.
Their analysis is based on Gibbs uniqueness, an extremely strong spatial mixing property that fails to hold for (much) larger $d$.
In particular, Gibbs uniqueness implies condition~\eqref{eqRS}.
Moreover, in the case $k=2$ Gibbs uniqueness holds up to the satisfiability threshold~\cite{2SAT}.

Recently Galanis, Goldberg, Guo and Yang \cite{GGGY} proposed a fully polynomial-time approximation scheme for computing $Z(\PHI, \infty)$ for large enough $k$ and $d\leq 2^{k/301}$.
While avoiding an explicit connection to the replica symmetry assumption (\ref{eqRS}), they seize upon the technique of Moitra \cite{Moitra} for approximately counting satisfying assignments of formulas with bounded variable degrees.
It would be interesting to see if this approach extends to finite $\beta$ and if a proof of \eqref{eqRS} can be salvaged from the techniques from~\cite{GGGY,Moitra}.

A key feature of the $k$-SAT problem that goes a long way to explaining the technical difficulty of the model is that the marginals of the Boltzmann distribution vary from one variable to another.
This manifests itself in the fact that the limiting distribution $\pi_{d,\beta}^\star$ from \Prop~\ref{Prop_Noela} is non-trivial~\cite{MZ}.
The distribution is generally a mixture of a discrete and a continuous probability measure.
As a result, we do not expect that there is a simple analytic expression for the limiting value of the Bethe free energy in \Thm~\ref{Thm_main}.
On a technical level one of the main contributions of this article is that we manage to deal with the inherent asymmetry of the problem.
By comparison, the replica symmetric regime of symmetric problems where the relevant Belief Propagation fixed point is trivial (e.g., the uniform distribution) is well understood.
In this case the counterpart of \Thm~\ref{Thm_main} is essentially trivial and the existence and location of the replica symmetry breaking phase have been established precisely~\cite{victor,Charting,ACOWorm}.
But of course quite a few prominent problems do not possess symmetry.
For instance, apart from the $k$-SAT model the hard-core model on random graphs springs to mind.

Beyond probabilistic combinatorics and computer science, the random $k$-SAT model has been studied as a diluted spin glass model.
Using a combination of the interpolation method and spatial mixing arguments, Panchenko and Talagrand~\cite{PanchkSAT,PanchenkoTalagrand,Talagrand} studied the model in the case of very low $d$ and/or $\beta$, known as the ``high-temperature'' version of the model.
Additionally, Panchenko~\cite{Panchenko} obtained a variational formula for $\lim_{n\to\infty}n^{-1}\ex[\log Z(\PHI,\beta)]$.
However, this formula does not easily reveal the connection with Belief Propagation, nor the existence or location of the replica symmetry breaking phase transition.
Generally speaking, the analysis of diluted models with sparse interactions appears to lead to less explicit solutions than in the case of models with full interactions such as the Sherrington-Kirkpatrick model~\cite{PanchenkoBook}.

In a few models that bear similarity with random $k$-SAT it has been possible to move beyond the replica symmetric phase.
For example, the 1-step replica symmetry breaking phase has been investigated in detail in the random regular $k$-NAESAT model, where even the existence of a Gardener (or full replica symmetry breaking) transition has been established rigorously~\cite{Bartha,SSZ}.
The work of Sly, Sun and Zhang~\cite{SSZ} on the 1-step RSB formula for the free energy combines the interpolation method with the second moment method.
Moreover, the proof design from~\cite{Bartha} is somewhat reminiscent of the strategy that we pursue here to establish \Thm~\ref{Thm_rsb}, but the details are very different.
More specifically, conceptually \cite{Bartha} deals with a more complex question, namely 1-step versus 2-step replica symmetry breaking, whereas here we are concerned with replica symmetry versus 1-step replica symmetry breaking.
That said, the model that we study here is more intricate than the regular $k$-NAESAT model, which enjoys relatively strong symmetry properties.

What the present strategy and that pursued in~\cite{Bartha} have in common is that we use contraction techniques to pin down the conceivable solutions to the simpler recurrence (replica symmetry in our case and 1-step RSB in~\cite{Bartha}).
In addition, both the present work and~\cite{Bartha} use the interpolation method to obtain a contradiction.
Yet a difference is that here the reference point of the analysis is the replica symmetry condition~\eqref{eqRS}, whereas the starting point in~\cite{Bartha} is the 1RSB formula for the ground state energy.
In particular, the proof of \Thm~\ref{Thm_rsb} requires a lower bound as well as an upper bound on the partition function. 
Moreover, in order to refute the replica symmetric scenario we need to investigate two different conceivable solutions to replica symmetric ansatz, respectively two fixed points of Belief Propagation.
The first of these corresponds to the scenario that typical pairs of satisfying assignments are essentially orthogonal; this case we tackle via the contraction method.
The second scenario is that typical Boltzmann samples have a very high overlap.
This case requires delicate combinatorial expansion arguments.

Finally, Panchenko~\cite{PanchenkoSATRSB} studied the random $k$-SAT model in the limit of large $d$. The main result is that $n^{-1}$ $\ex[\log Z(\PHI,\beta)]$ approaches the solution to a certain fully connected $k$-spin model in the limit of large $d$; that is, a model of Sherrington-Kirkpatrick type that lives on a complete hypergraph.
These models are currently better understood than models on sparse random graphs and, in particular, there are formulas for the approximate value of the partition function that match the so-called full replica symmetry breaking predictions from physics~\cite{PanchenkoBook}.
However, the regime of $d$ where Panchenko's results bite is well beyond the $k$-SAT threshold and, indeed, a full replica symmetry breaking regime is not expected to occur in the satisfiable phase of the random $k$-SAT model~\cite{MPZ}.

\section{Preliminaries and notation}\label{Sec_prelims}

\noindent
A $k$-SAT formula $\Phi$ with a set $V(\Phi)$ of variables and a set $C(\Phi)$ of clauses can be represented by a bipartite graph $G(\Phi)$ known as the {\em factor graph}.
In this graph there is an edge between $x\in V(\Phi)$ and $a\in C(\Phi)$ iff $x$ occurs in $a$.
For a vertex $v$ of the graph we let $\partial v$ denote the set of neighbours. 
Furthermore, for an integer $\ell\geq1$ we let $\partial^\ell v$ be the set of vertices at distance precisely $\ell$ from $v$.
Where necessary we annotate $\Phi$ to clarify the reference to the formula.
In addition, for a formula $\Phi$ and an assignment $\sigma\in\{\pm1\}^{V(\Phi)}$ we let
\begin{align*}
	\cH_\Phi(\sigma)&=\sum_{a\in C(\Phi)}\vecone\cbc{\sigma\not\models a}
\end{align*}
be the number of clauses that $\sigma$ violates; the function $\cH_\Phi$ is known as the {\em Hamiltonian} in physics jargon.

We always write $V_n=\{x_1,\ldots,x_n\}$ for the variable set of the random formula $\PHI$.
For each of the corresponding $2n$ literals $x_i,\neg x_i$ we denote by $\vd_i^{\pm}$ the degree of that literal, i.e., the number of clauses where the literal appears.
For the entire literal degree sequence we introduce the symbol $\vd=(\vd_i^{\pm})_{i\in[n]}$.
Additionally, let $\fD$ be the $\sigma$-algebra generated by $\vd$; observe that the total number $\vm$ of clauses is $\fD$-measurable.

We will use the following important theorem about the random $k$-SAT model.

\begin{theorem}[{\cite{DSS3}}]\label{Thm_DSS}
	There exist a number $k_0>3$ and a sequence $\dk(k)=k2^k\log 2-k(1+\log 2)/2+o(1)$ such that for all $k\geq k_0$ \whp\ $\PHI$ has a satisfying assignment if $d<\dk(k)$ and fails to possess one if $d>\dk(k)$.
\end{theorem}

\noindent
\Thm~\ref{Thm_DSS} implies that $Z(\PHI,\beta)\geq Z(\PHI,\infty)\geq1$ for all $\beta>0$ and all $d<\dk(k)$ \whp\

Throughout the paper we will be dealing a fair bit with probability distributions on discrete cubes.
For finite sets $\Omega,V\neq\emptyset$ we let $\cP(\Omega^V)$ be the set of all probability measures on $\Omega^V$.
For a set $U\subset V$ and $\mu\in\cP(\Omega^V)$ we let $\mu_U\in\cP(\Omega^U)$ be the joint distribution of the coordinates $u\in U$ under $\mu$.
If $U=\{u_1,\ldots,u_\ell\}$ is given explicitly, we use the shorthand $\mu_U=\mu_{u_1,\ldots,u_\ell}$.
Moreover, a distribution $\mu\in\cP(\Omega^V)$ is {\em $(\eps,\ell)$-extremal} if
\begin{align*}
	\sum_{U\subset V:|U|=\ell}\dTV(\mu_U,\bigotimes_{u\in U}\mu_u)<\eps\binom{|V|}\ell.
\end{align*}
Thus, for most $\ell$-sets $U\subset V$ the joint distribution $\mu_U$ is close in total variation to the product distribution with the same marginals $\mu_u$, $u\in U$.
If $\ell=2$ we just call $\mu$ $\eps$-extremal.
Hence, because $\PHI$ is invariant under permuatations of the variables, the condition \eqref{eqRS} posits that the Boltzmann distribution $\mu_{\PHI,\beta}$ is $o(1)$-extremal \whp\
We will apply the following result repeatedly.

\begin{lemma}[{\cite{harnessing}}]\label{Lemma_Victor}
	For any $\Omega\neq\emptyset,\eps>0,\ell\geq3$ there exists $\delta>0$ such that for all sets $V$ with $|V|>1/\delta$ any $\delta$-extremal $\mu\in\cP(\Omega^V)$ is $(\eps,\ell)$-extremal.
\end{lemma}

Let $\mu,\nu\in\cP(\Omega^V)$ and $c>0$.
We say that $\mu$ is {\em $c$-contiguous} w.r.t.\ $\nu$ if $\mu(\cE)\leq c\nu(\cE)$ for any $\cE\subset\Omega^V$.

\begin{lemma}[{\cite{Will2}}]\label{Lemma_contig}
	For any $\Omega\neq\emptyset,c>0,\eps>0$ there exists $\delta>0$ such that for all sets $V$ with $|V|>1/\delta$, any $\delta$-extremal $\mu\in\cP(\Omega^V)$ and any $\nu\in\cP(\Omega^V)$ that is $c$-contiguous w.r.t.\ $\mu$ the following statements are true. 
	\begin{enumerate}[(i)]
	\item $\nu$ is $\eps$-extremal.
	\item $\sum_{v\in V}\dTV(\mu_v,\nu_v)<\eps |V|$.
	\end{enumerate}
\end{lemma}

For a probability measure $\mu$ on a discrete space $\Omega$ and function $X:\Omega\to\RR$ we introduce the bracket notation
\begin{align*}
	\scal X\mu&=\sum_{\omega\in\Omega}X(\omega)\mu(\omega).
\end{align*}
Thus, $\scal X\mu$ is the mean of $X$ w.r.t.\ $\mu$.
More generally, if $Y:\Omega^\ell\to\RR$ for some $\ell\geq1$, then
\begin{align*}
	\scal Y\mu&=\sum_{\omega_1,\ldots,\omega_\ell\in\Omega}Y(\omega_1,\ldots,\omega_\ell)\prod_{i=1}^\ell\mu(\omega_i)
\end{align*}
is the expectation of $Y$ w.r.t.\ $\mu^{\otimes\ell}$.

Apart from discrete distributions we will also be working with spaces of continuous probability measures.
For a Polish space $\fE$ let $\cP(\fE)$ be the space of all probability measures on $\fE$.
In addition, for a subspace $\fE\subset\RR$ we introduce the $L^r$-Wasserstein space $\cW_r(\fE)$ as the space of all probability distributions $\mu\in\cP(\fE)$ with $\int_\fE|x|^r\dd\mu(x)<\infty$.
We endow this space with the Wasserstein metric $W_r$, thereby turning $\cW_r(\fE)$ into a complete metric space.
We recall that the Wasserstein metric is defined as
\begin{align*}
	W_r(\mu,\nu)&=\inf\cbc{\bc{\int_{\fE\times\fE}|x-y|^r\dd\gamma(x,y)}^{1/r}:\gamma\in\cP(\fE\times\fE)\mbox{ is a coupling of }\mu,\nu}.
\end{align*}

Throughout the paper we use the $O$-notation to refer to asymptotics as either $n$, $k$ or $\beta$ get large.
By default $O$-symbols refer to the limit as $n\to\infty$.
However, where the expression inside the $O(\,\cdot\,)$-symbol depends on $k$ or $\beta$ but not on $n$, it is understood that we mean to take the respective parameter to infinity instead of $n$.
Where there is a risk of ambiguity we make the reference explicit by adding a subscript.
Thus, $O_k(1)$ stands for an expression that remains bounded in the limit of large $k$.
Naturally, the same conventions also apply to $o(\nix)$, $\Omega(\nix)$, $\Theta(\nix)$.
In addition, we use symbols such as $\tilde O(\nix)$ to suppress logarithmic factors.
For example, $\tilde O(n)$ can be written out as $O(n\log^{O(1)}n)$ while $\tilde O(2^k)$ stands for a term of order $k^{O(1)}2^k$.
Moreover, in the entire paper we tacitly assume that $k,n$ exceed large enough absolute constants $k_0,n_0$, respectively.

We also need a few basic large deviations inequalities.
Recall that the Kullback-Leibler divergence of two probability measures $\mu,\nu\in\cP(\Omega)$ is defined as
\begin{align*}
	\KL{\mu}\nu&=\sum_{\omega\in\Omega}\mu(\omega)\log\frac{\mu(\omega)}{\nu(\omega)}\in[0,\infty],
\end{align*}
with the conventions $0\log0=0\log\frac{0}{0}=0$.

\begin{lemma}[``Chernoff bound'']\label{Lemma_Chernoff}
	Suppose that $X$ has a binomial distribution $\Bin(n,p)$.
	Then
	\begin{align*}
	\pr\brk{X\geq qn}&\leq\exp\bc{-n\KL{\Be(q)}{\Be(p)}}&&\mbox{ if }q>p,\\
\pr\brk{X\leq qn}&\leq\exp\bc{-n\KL{\Be(q)}{\Be(p)}}&&\mbox{ if }q<p.
	\end{align*}
\end{lemma}

\begin{lemma}[``Bennett's inequality'']\label{Lemma_Bennett}
Suppose that $X$ is a $\Po(\lambda)$ variable. 
Then 
\begin{align}\label{Poi_Ben}
	\Pr\bc{X\geq \lambda + x} &\leq \exp\bc{x-(\lambda+x)\log\bc{1+\frac{x}{\lambda}}} \leq \mathrm{exp}\bc{-\frac{x^2}{2\lambda+2x/3}}&&\mbox{for any $x \geq 0$,}\\
	\Pr\bc{X\leq \lambda - x} &\leq\exp\bc{-x-(\lambda-x)\log\bc{1-\frac{x}{\lambda}}}\leq\exp\bc{-\frac{x^2}{2\lambda}} &&\mbox{for any }0\leq x<\lambda.
\end{align}
\end{lemma}

Finally, we remind ourselves of the well-known fact that Belief Propagation ``is exact on trees''.
To be precise, let $T$ be a $k$-SAT formula whose factor graph $G(T)$ is a tree.
Then we can introduce Belief Propagation on $T$ via \eqref{eqTreeBP1}--\eqref{eqTreeBP3}.
The following statement provides that for large enough $t$ these recurrences render the Boltzmann marginals of $T$.

\begin{theorem}[{\cite[\Thm~14.4]{MM}}]\label{Thm_treeBP}
	Assume $T$ is a $k$-SAT instance whose factor graph $G(T)$ is a tree and that $t$ exceeds the diameter of $G(T)$.
	Then for all variables $x$ of $T$ and any $\beta>0$ we have $ \mu_{T,\beta}(\{\SIGMA_x=1\})=\mu_{T,\beta,x,t}(1).  $
\end{theorem}

\section{Moment calculations}\label{Sec_Prop_bal}

\noindent
In this section we prove \Prop~\ref{Prop_bal} and \Cor~\ref{Cor_bal}.
Unless specified otherwise we tacitly assume that $d\leq d^*$.
Moreover, recalling the definition \eqref{maximal_p} of $p=p(k,\beta)$, we introduce
\begin{align}\label{equ}
u&=u(k,\beta)=\frac{1-2p}{2p(\eul^\beta-1)}\in(0,1).
\end{align}

\subsection{Overview}
As we saw in \Sec~\ref{Sec_smm_overview}, we cannot hope to prove \Prop~\ref{Prop_bal} by simply calculating the second moment of the partition function $Z(\PHI,\beta)$. 
This is because the expression \eqref{eqSecondMmt} for the second moment attains its maximum at a value of $\alpha$ strictly greater than $1/2$.
To solve this problem we will replace $Z(\PHI,\beta)$ by a modified random variable for which the overlap value $\alpha=1/2$ dominates by design.
The precise construction of this random variable borrows an idea from the work of Achlioptas and Peres~\cite{yuval} on $k$-SAT with hard constraints (i.e., $\beta=\infty$).
Namely, we call an assignment $\sigma\in\{\pm1\}^{V_n}$ {\em balanced} if 
	\begin{align}\label{eqBAL}
	\sum_{x\in V_n}\sigma_x(\vd_x^{+}-\vd_x^{-})=\begin{cases}
	0&\mbox{ if }k\vm\mbox{ is even}, \\
	1&\mbox{ otherwise.}
	\end{cases}
	\end{align}
Hence, if we inspect the truth values of the $k\vm$  literals as they appear in the $\vm$ clauses, we observe as many `true' as `false' literals, up to an additive error of one.
Further, we call a balanced assignment $\sigma$ {\em strongly balanced} if
\begin{align}\label{eqstronglybal}
\abs{\sum_{x\in V_n}\sigma_x\vecone\{\vd_x^{+}=d^+,\vd_x^{-}=d^-\}}&\leq\sqrt n&&
\mbox{ for all integers $d^+,d^-\geq0$.}
\end{align}
Thus, under a strongly balanced assignment about half the variables with each possible degree constellation $(d^+,d^-)$ are set to `true', up to an error of $O(\sqrt n)$.
Let $\BAL$ denote the set of all strongly balanced assignments.
Now our modified version of the partition function reads
\begin{align*}
\Zbal&=\exp(-\beta u\vm)\sum_{\sigma\in\BAL}\vecone\cbc{\sum_{i=1}^{\vm}\vecone\cbc{\sigma\not\models a_i}=\lceil u\vm\rceil}.
\end{align*}
Thus, we confine ourselves to strongly balanced assignments that leave precisely $\lceil u\vm\rceil$ clauses unsatisfied.
Naturally, it will emerge in due course that the choice \eqref{equ} of $u$ maximises the mean of $\Zbal$.
The following two propositions, which we prove in \Sec s~\ref{Sec_Thm:FirstMoment} and~\ref{Sec_Thm:SecondMoment}, render the first and the second moment of $\Zbal$.

\begin{proposition} \label{Thm:FirstMoment}
\Whp\ we have 
\begin{align*}
\frac{1}{n}\log\Erw\brk{\Zbal\Big \vert \fD} 
&= \bc{1- \frac{(k-1)d}{k}}\log 2 -\frac{d}{2}\log(p(1-p))+\frac{d}{k} \log p+o(1).
\end{align*}
\end{proposition}

\begin{proposition} \label{Thm:SecondMoment}
\Whp\ we have
\begin{align*}
\frac{1}{2n}\log\Erw\brk{\Zbal^2\Big \vert \fD} 
&= \bc{1- \frac{(k-1)d}{k}}\log 2 - \frac{d}{2}\log(p(1-p)) +  \frac{d}{k} \log p+o(1).
\end{align*}
\end{proposition}

\noindent
The proofs of \Prop s~\ref{Thm:FirstMoment}--\ref{Thm:SecondMoment} are generalisations of the moment calculations from~\cite{ACOWorm}, where assignments that satisfy {\em all} clauses were counted.
A significant complication here is that a certain number of clauses are left unsatisfied.
This introduces a further dimension to the second moment analysis, namely the number of clauses that are left unsatisfied under both assignments, leaving us with a technically far more challenging task.
\Prop~\ref{Prop_bal} is an easy consequence of \Prop s~\ref{Thm:FirstMoment} and~\ref{Thm:SecondMoment} and the Paley-Zygmund and Azuma--Hoeffding inequalities.

\begin{proof}[Proof of \Prop~\ref{Prop_bal}]
The Paley-Zygmund inequality implies that
\begin{align*}
\pr\brk{\Zbal\geq\Erw[\Zbal\mid\fD]/4\mid\fD}&\geq\frac{\Erw[\Zbal\mid\fD]^2}{4\Erw[\Zbal^2\mid\fD]}.
\end{align*}
Hence, \Prop~\ref{Thm:SecondMoment} shows that \whp{}
\begin{align}\label{eqProp_bal_1}
\pr\brk{\Zbal\geq\Erw[\Zbal\mid\fD]/4\mid\fD}&\geq\exp(o(n)).
\end{align}
Further, combining \eqref{eqProp_bal_1} with \Prop~\ref{Thm:FirstMoment} and using the trivial inequality $Z(\PHI,\beta)\geq\Zbal$, we obtain
\begin{align}\label{eqProp_bal_2}
\pr\brk{n^{-1}\log Z(\PHI,\beta)\geq\bc{1- \frac{(k-1)d}{k}}\log 2 -\frac{d}{2}\log(p(1-p))+\frac{d}{k} \log p+o(1)}&\geq\exp(o(n)).
\end{align}
Moreover, because adding or removing a single clause can alter the value of the partition function by no more than a factor of $\exp(\pm\beta)$, the Azuma--Hoeffding inequality shows that for any $t>0$,
\begin{align}\label{eqProp_bal_3}
\pr\brk{\abs{\log Z(\PHI,\beta)-\Erw[\log Z(\PHI,\beta)\vert \vm]}\geq t \vert \vm}\leq2\exp(-t^2/(2\beta^2\vm)). 
\end{align}
Thus, combining \eqref{eqProp_bal_2} and \eqref{eqProp_bal_3}, we conclude that 
$$\liminf_{n \to \infty} n^{-1}\Erw[\log Z(\PHI,\beta)] \geq\bc{1- \frac{(k-1)d}{k}}\log 2 -\frac{d}{2}\log(p(1-p))+\frac{d}{k} \log p+o(1),$$ as desired.
\end{proof}

Let us move on to the proof of \Cor~\ref{Cor_bal} concerning the tails of the distribution of Boltzmann marginals.
Combining \Prop~\ref{Prop_bal} with the Azuma--Hoeffding inequality as in \eqref{eqProp_bal_3}, we conclude that
\begin{align}\label{eqBALest}
Z(\PHI,\beta)\geq\exp\bc{n\brk{\bc{1- \frac{(k-1)d}{k}}\log 2 -\frac{d}{2}\log(p(1-p))+\frac{d}{k} \log p+o(1)}}&&\mbox{\whp}
\end{align}
Of course, this bound directly yields a lower bound on the corresponding sum over pairs of assignments, namely
\begin{align}
Z(\PHI,\beta)^2&=\sum_{\sigma,\tau}\eul^{-\beta\sum_{i=1}^{\vm}\vecone\cbc{\sigma\not\models a_i}+\vecone\cbc{\tau\not\models a_i}}\geq\exp\bc{2n\brk{\bc{1- \frac{(k-1)d}{k}}\log 2 -\frac{d}{2}\log(p(1-p))+\frac{d}{k} \log p+o(1)}}&&\mbox{\whp}
\label{eqBALest2}
\end{align}
Let us compare this bound with the expansion~\eqref{eqsmm_deriv3} of the second moment.
The contribution to \eqref{eqsmm_deriv3} of a specific overlap value $\alpha$ is bounded by $\exp(n(f(\alpha)+o(1)))$.
Comparing these estimates carefully, we will discover that the total contribution of all overlap values $\alpha$ that differ significantly from $1/2$ is tiny by comparison to \eqref{eqBALest2}.
As a consequence, \whp\ the overlap of two independent random samples $\SIGMA,\SIGMA'$ drawn from the Boltzmann distribution must be close to $1/2$. 
The following corollary provides a precise statement of this observation.

\begin{corollary}\label{Cor_SecondMoment}
We have $ \Erw\brk{\mu_{\PHI,\beta}(\{|\alpha(\SIGMA,\SIGMA')-1/2|>k^92^{-k/2}\})}=o(1).$
\end{corollary}

Finally, \Cor~\ref{Cor_bal} is an easy consequence of \Cor~\ref{Cor_SecondMoment} and general facts about Boltzmann distributions.
The details can be found in \Sec~\ref{Sec_Cor_bal}.

\subsection{Proof of \Prop~\ref{Thm:FirstMoment}}\label{Sec_Thm:FirstMoment}
As a first step we estimate the number of balanced assignments.

\begin{lemma}\label{Lem_1stmmt_entropy}
\Whp\ we have $|\BAL|=2^{n+o(n)}$. 
\end{lemma}
\begin{proof}
The Chernoff bound shows that $\vd_x^+,\vd_x^-\leq\log n$ \whp\ for all $x\in V_n$.
Moreover, Chebyshev's inequality easily shows that for a uniformly random $\vzeta\in\{\pm1\}^{V_n}$, for any  $d^+, d^- \leq \log n$ we have
\begin{align}\label{eqLem_1stmmt_entropy_1}
\pr\brk{\abs{\sum_{x\in V_n}\vzeta_x\vecone\{\vd_x^{+}=d^+,\vd_x^{-}=d^-\}}\leq\sqrt n\mid\fD}=\Omega(1).
\end{align}
Consequently, $\vzeta$ satisfies \eqref{eqstronglybal} with probability $\exp(O(\log^2 n))$.
Further, the central limit theorem shows that \whp
\begin{align}\label{eqLem_1stmmt_entropy_2}
\pr\brk{\abs{\sum_{x\in V_n}\bc{1-\vecone\cbc{\vd_x^+=1,\vd_x^-=0}}\vzeta_x(\vd_x^+-\vd_x^-)}\leq \sqrt n/2
\mid\vzeta\mbox{ satisfies \eqref{eqstronglybal}},\fD}&=\Omega(1).
\end{align}
Finally, there are $\Theta(n)$ variables $x\in V_n$ such that $\vd_x^+=1$, $\vd_x^-=0$ \whp\ and hence Stirling's formula shows that for any integer $h$ with $|h|\leq \sqrt n/2$ we have
\begin{align}\label{eqLem_1stmmt_entropy_3}
\pr\brk{\sum_{x\in V_n}\vecone\cbc{\vd_x^+=1,\vd_x^-=0}\vzeta_x=h\mid\vzeta\mbox{ satisfies \eqref{eqstronglybal}},\fD}&=\Omega(n^{-1/2}).
\end{align}
Since $\sum_{x\in V_n}\bc{1-\vecone\cbc{\vd_x^+=1,\vd_x^-=0}}\vzeta_x (\vd_x^+-\vd_x^-)$ and $\sum_{x\in V_n}\vecone\cbc{\vd_x^+=1,\vd_x^-=0}\vzeta_x$ are conditionally independent given $\fD$, \eqref{eqLem_1stmmt_entropy_1}--\eqref{eqLem_1stmmt_entropy_3} imply that \whp\  $\pr\brk{\vzeta\in\BAL\mid\fD}=\exp(o(n))$, whence the assertion is immediate.
\end{proof}

Let us now fix any strictly balanced assignment $\sigma$.
Given $\fD$ and $\sigma$, the only randomness left is the way in which the positive and negative occurrences of the individual variables are matched to the clauses.
To be precise, since we only need to know the number of clauses that will be left unsatisfied,  we do not care about the identity of the underlying variable of a literal in a given clause, but only about its truth value.
Therefore, we can think of the positive and negative variable occurrences as tokens that are labelled either `true' or `false'.
Hence, a variable $x$ gives rise to $\vd_x^+$ `true' and $\vd_x^-$ `false' tokens if $\sigma_x=1$, and to  $\vd_x^+$ `false' and $\vd_x^-$ `true' tokens if $\sigma_x=-1$.
In effect, we just need to study the number of clauses that receive $k$ `false' tokens if we put the $k\vm$ tokens down randomly upon the $\vm$ clauses.
In fact, since $\sigma$ is strictly balanced, we know that the precise number of `true' tokens equals 
\begin{align}\label{eqtruetokens}
\sum_{x \in V_n}\vd_x^+\vecone\cbc{\sigma_x=1}+\vd_x^-\vecone\cbc{\sigma_x=-1}=\lceil k\vm/2\rceil.
\end{align}
Of course, the remaining $\lfloor k\vm/2\rfloor$ tokens must be `false'.

To study this token shuffling experiment we introduce an auxiliary probability space.
Let $(\CHI_{i,j})_{i,j\geq1}$ be a family of Rademacher variables such that $\pr[\CHI_{i,j}=1]=p$ for all $i,j$.
The idea is that $\CHI_{i,1},\ldots,\CHI_{i,k}$ represent the $k$ tokens that clause $i$ receives.
Of course, in order to faithfully represent the token experiment we need to ensure that \eqref{eqtruetokens} is satisfied, i.e., that the total number of $+1$-tokens comes to $\lceil k\vm/2\rceil$.
Thus, we need to condition on the event
\begin{align*}
\cB&=\cbc{\sum_{i=1}^{\vm}\sum_{j=1}^k\vecone\{\CHI_{i,j}=1\}=\lceil k\vm/2\rceil}.
\end{align*}
We need to compute the conditional probability that given $\cB$ the total number of clauses that only receive $-1$-tokens equals $\lceil u\vm\rceil$.
Hence, introducing
\begin{align}\label{eqZbalFirstMmtNiceFormula}
	\cS_{\vm}&=\cbc{\sum_{i=1}^{\vm}\vecone\cbc{\max_{j\in[k]}\CHI_{i,j}=-1}=\lceil u\vm\rceil},&\mbox{we obtain}\quad
\Erw[\Zbal\mid\fD]&= \exp\bc{-\beta u \vm} \abs\BAL\cdot \frac{\pr\brk{\cS \cap \cB \vert \vm} }{\pr\brk{\cB \vert \vm}}.
\end{align}
Since \Lem~\ref{Lem_1stmmt_entropy} already shows that $\abs\BAL=2^{n+o(n)}$, the remaining challenge is to calculate $\pr\brk{\cS_{\vm}\mid\cB_{\vm}}$.
To this end we calculate $\pr\brk{\cB_{\vm}}$, $\pr\brk{\cS_{\vm}}$ and $\pr\brk{\cB_{\vm}\mid\cS_{\vm}}$ and use Bayes' formula.

\begin{lemma}\label{Lemma_1stmmt_B}
We have $\pr\brk{\cB_{\vm}}=\binom{k\vm}{\lceil k\vm/2\rceil}p^{\lceil k\vm/2\rceil}(1-p)^{k\vm-\lceil k\vm/2\rceil}$. 
\end{lemma}
\begin{proof}
This is because the random variables $\CHI_{i,j}$ are mutually independent.
\end{proof}

\begin{lemma}\label{Lemma_1stmmt_S}
\Whp\  we have $\pr\brk{\cS\mid\fD}=\binom{\vm}{u \vm}(1-p)^{k\lceil u\vm\rceil}(1-(1-p)^k)^{\vm-\lceil u\vm\rceil}$.
\end{lemma}
\begin{proof}
Because the $\CHI_{i,j}$ are independent, for any given index $i\in[\vm]$ we have $\pr[\max_{j\in[k]}\CHI_{i,j}=-1]=(1-p)^k$, independently of all others.
Thus, the number $i\in[\vm]$ with $\max_{j\in[k]}\CHI_{i,j}=-1$ has distribution $\Bin(\vm,(1-p)^k)$.
\end{proof}

\begin{lemma}\label{Lemma_1stmmt_B|S}
\Whp\ we have $\pr\brk{\cB\mid\cS, \fD}$.
\end{lemma}
\begin{proof}
Due to (\ref{maximal_p}) and the choice of $u$ \whp\ we have 
\begin{align*}
\Erw\brk{\sum_{i=1}^{\vm}\sum_{j=1}^k\vecone\cbc{\CHI_{i,j}=1}\mid\cS,\fD}&=\frac{(\vm-\lceil u\vm\rceil)kp}{1-(1-p)^k}
=dn\bc{1-\frac{(1-p)^k}{2p\exp(\beta)}}\cdot\frac p{1-(1-p)^k}+O(\sqrt n)=\frac{k\vm}2+O(\sqrt n).
\end{align*}
Therefore, the assertion follows from the local limit theorem for sums of independent random variables.
\end{proof}

\begin{proof}[Proof of \Prop~\ref{Thm:FirstMoment}]
Combining \Lem s~\ref{Lem_1stmmt_entropy}, \ref{Lemma_1stmmt_B}, \ref{Lemma_1stmmt_S} and~\ref{Lemma_1stmmt_B|S}, we conclude that \whp{}
\begin{align*}
\log&\Erw[\Zbal\mid\fD]=(n- k\vm)\log 2-\frac{k\vm}{2}\log(p(1-p))\\
&+ku\vm\log(1-p)+(1-u)\vm\log(1-(1-p)^k)-\beta u\vm+ \log \binom{m}{um}+o(n)\\
&=n\left[(1-d)\log 2- \frac d2\log p + \frac d2\bc{\frac{(1-p)^k}{2p\exp(\beta)} -1}\log(1-p) \right. \\
	& \left. + \frac dk\bc{1-\frac{(1-p)^k}{2p\exp(\beta)}}\log(1-(1-p)^k)-\frac{\beta d u}{k} -\frac{d}{k} \bc{u \log (u) + (1-u) \log(1-u)}+o(1) \right].
\end{align*}
Simplifying the above using the definition of $u$ and \eqref{maximal_p} yields the desired expression.
\end{proof}

\subsection{Proof of \Prop~\ref{Thm:SecondMoment}}\label{Sec_Thm:SecondMoment}
The {\em weighted overlap} of two truth assignments $\sigma,\tau\in\{\pm1\}^{V_n}$ is defined as
\begin{align*}
\omega(\sigma,\tau)&=\frac1{k\vm}\sum_{x\in V_n}{\vecone\cbc{\sigma_x=\tau_x=1}\vd_x^++\vecone\cbc{\sigma_x=\tau_x=-1}\vd_x^-}.
\end{align*}
Thus, the weighted overlap equals the fraction of literal occurrences that evaluate to `true' under both $\sigma,\tau$.
Let $\cO=\cO(\vd)=\{\omega(\sigma,\tau):\sigma,\tau\in\{\pm1\}^{V_n}\}$ be the set of all conceivable weighted overlaps.
Introducing
\begin{align}
\vE(\omega)&=\sum_{\sigma,\tau\in\BAL}\vecone\cbc{\omega(\sigma,\tau)=\omega}\exp(-2\beta u\vm)\pr\brk{\sum_{i=1}^{\vm}\vecone\cbc{\sigma\not\models a_i}=\sum_{i=1}^{\vm}\vecone\cbc{\tau\not\models a_i}=\lceil u\vm\rceil\mid\fD},\label{eqE}
\end{align}
we can then write the second moment as $\Erw[\Zbal^2\mid\fD]=\sum_{\omega\in\cO}\vE(\omega)$.

We will use two separate arguments to estimate $\vE(\omega)$ for different regimes of $\omega$.
The first  regime that we consider is $\omega$ close to $1/4$.
This will turn out to be the dominant case.

\begin{proposition}\label{Prop_smm_1/4}
\Whp\ $\max\{\vE(\omega):\omega\in\cO,\,|\omega-1/4|\leq k^{100}2^{-k/2}\}\leq \exp(o(n))\Erw[\Zbal\mid\fD]^2$.
\end{proposition}

\noindent
The proof of \Prop~\ref{Prop_smm_1/4} can be found in \Sec~\ref{Sec_Prop_smm_1/4}.
Moving on to weighted overlaps far from $1/4$,  we will derive the following bound on $\vE(\omega)$ in terms of the function $f(\alpha)$ from \eqref{eqSecondMmt}.

\begin{proposition}\label{Prop_smm_not_1/4}
\Whp\ $\max\{\vE(\omega):\omega\in\cO,|\omega-1/4|> k^{100}2^{-k/2}\}\leq\exp\bc{n\max\{f(\alpha):\alpha\in[1/2+k^{90}2^{-k/2},1]\}}$
.
\end{proposition}
\noindent
We prove \Prop~\ref{Prop_smm_not_1/4} in \Sec~\ref{Sec_Prop_smm_not_1/4}.
Finally, in \Sec~\ref{Sec_Prop_f} we will bound $f(\alpha)$ as follows.

\begin{proposition}\label{Prop_f}
	We have $\max\{f(\alpha):\alpha\in[1/2+k^{90}2^{-k/2},1]\}<2\bc{1- (k-1)d/k}\log 2 -d\log(p(1-p))+2d\log(p)/k.$
\end{proposition}

\noindent
\Prop~\ref{Thm:SecondMoment} is an easy consequence of \Prop s~\ref{Prop_smm_not_1/4}--\ref{Prop_f}.

\begin{proof}[Proof of \Prop~\ref{Thm:SecondMoment}]
Combining \Prop s~\ref{Thm:FirstMoment}, \ref{Prop_smm_not_1/4} and \ref{Prop_f}, we conclude that \whp{}
\begin{align}\label{eqThm:SecondMoment1}
\max\cbc{\vE(\omega):\omega\in\cO,|\omega-1/4|> k^{100}2^{-k/2}}&\leq\exp(-\Omega(n))\Erw[\Zbal\mid\fD]^2.
\end{align}
Since $|\cO|=O(n)$ \whp, \Prop~\ref{Prop_smm_1/4}, \eqref{eqE} and~\eqref{eqThm:SecondMoment1} imply that $\Erw[\Zbal^2\mid\fD]\leq\exp(o(n))\Erw[\Zbal\mid\fD]^2$ \whp,  as desired.
\end{proof}

\subsubsection{Proof of \Prop~\ref{Prop_smm_1/4}}\label{Sec_Prop_smm_1/4}
As in the proof of  \Prop~\ref{Thm:FirstMoment} we begin by estimating the number of pairs of balanced assignments with a given weighted overlap.
Subsequently we will switch to an auxiliary probability space to calculate the probability that both such assignments happen to violate precisely $\lceil u\vm\rceil$ clauses.
Hence, draw $\TAU,\TAU'\in\BAL$ uniformly and independently.

\begin{lemma}\label{Lemma_2ndmmt_entropy}
\Whp\ we have $ \pr\brk{\abs{\omega(\TAU,\TAU')-1/4}>\eps\mid\fD}\leq2\exp\bc{-\frac{\eps^2\vm^2}{4d^2n}+o(n)}\mbox {for all }\eps>0.  $
\end{lemma}
\begin{proof}
Let $\vzeta,\vzeta'\in\{\pm1\}^{V_n}$ be drawn uniformly and independently.
Then \Lem~\ref{Lem_1stmmt_entropy} implies that \whp{}
\begin{align}\label{eqLemma_2ndmmt_entropy1}
\pr\brk{\abs{\omega(\TAU,\TAU')-1/4}>\eps\mid\fD}&\leq\exp(o(n))\pr\brk{\abs{\omega(\vzeta,\vzeta')-1/4}>\eps\mid\fD}.
\end{align}
Furthermore, since the pairs $(\vzeta_x,\vzeta_x')\in\{\pm1\}^2$ are mutually independent and changing $(\vzeta_x,\vzeta_x')$ can alter $\omega(\vzeta,\vzeta')$ by at most $\vd_x/(k\vm)$, the Azuma--Hoeffding inequality yields
\begin{align}\label{eqLemma_2ndmmt_entropy2}
\pr\brk{\abs{\omega(\vzeta,\vzeta')-1/4}>\eps\mid\fD}&\leq2\exp\bc{-\frac{\eps^2(k\vm)^2}{2\sum_{x\in V_n}\vd_x^2}}.
\end{align}
Finally, since $\sum_{x\in V_n}\vd_x^2\leq 2d^2n$ \whp,     \eqref{eqLemma_2ndmmt_entropy2} and  \eqref{eqLemma_2ndmmt_entropy1} imply the assertion.
\end{proof}

Like in \Sec~\ref{Sec_Thm:FirstMoment} we now fix any two assignments $\TAU,\TAU'$ with a given weighted overlap $\omega$ such that
\begin{align}\label{eqwoverlap}
|\omega-1/4|\leq k^{100}2^{-k/2}.
\end{align}
Given $\TAU,\TAU',\fD$, the experiment of actually constructing the random formula $\PHI$ boils down to matching the $k\vm$ literal slots in the $\vm$ clauses with the positive/negative occurrences of the variables $x_1,\ldots,x_n$.
But once again we do not actually care to know the identities of the literals in the various clauses, but only their truth value combinations under $\TAU,\TAU'$.
Hence, instead of actually matching literals to clauses, we might as well think of merely tossing tokens that indicate the truth value combinations $(1,1),(1,-1),(-1,1),(-1,-1)$ of the literals onto the clauses.
To be precise, because $\TAU,\TAU'$ are balanced, the fractions of tokens of each of the four types work out to be
\begin{align}\label{eqHash}
\omega_{11}&=\omega,&\omega_{1-1}&=\omega_{-11}=\frac{1}{2}-\omega_{11}+\frac{\vecone\cbc{k\vm\mbox{ is odd}}}{k\vm},&\omega_{-1-1}&=1-\omega_{11}-2\omega_{1-1}.
\end{align}
Hence, we just need to work out the probability that if we randomly put down $k\vm\omega_{11},k\vm\omega_{1-1},k\vm\omega_{-11},k\vm\omega_{-1-1}$ tokens of these four types onto the $\vm$ clauses, precisely $\lceil u\vm\rceil$ clauses will receive $k$ tokens of type either $(-1,1)$ or $(-1,-1)$ and, symmetrically, precisely $\lceil u\vm\rceil$ clauses will receive tokens of type $(1,-1)$ or $(-1,-1)$ only.

As in the first moment calculation, in order to calculate this probability it is convenient to move to an auxiliary probability space.
Specifically, let $(p_{11},p_{1 -1},p_{-11},p_{-1-1})\in(0,1)^4$ be a probability distribution on $\{\pm1\}^2$, i.e.,  $p_{11}+p_{1 -1}+p_{-11}+p_{-1-1}=1$, such that $p_{1-1}=p_{-11}$; we will choose expedient values of $p_{11},\ldots,p_{-1-1}$ in due course.
Moreover, let $(\vchi_{ij},\vchi'_{ij})_{i,j\geq1}$ be a sequence of i.i.d.\ random pairs $(\vchi_{ij},\vchi'_{ij})\in\{\pm1\}^2$ such that
\begin{align}\label{eqvchi}
\pr\brk{\vchi_{ij}=s,\ \vchi'_{ij}=t}&=p_{st}&&(s,t=\pm1).
\end{align}
Further, let 
\begin{align*}
\cB^\tensor&=\cbc{\sum_{i=1}^{\vm}\sum_{j=1}^k\vchi_{ij}=\sum_{i=1}^{\vm}\sum_{j=1}^k\vchi'_{ij}=\vecone\cbc{k\vm\mbox{ is odd}}},&
\cR^\tensor(\omega)&=\cbc{\sum_{i=1}^{\vm}\sum_{j=1}^k\vecone\cbc{\vchi_{ij}=\vchi'_{ij}=1}=\omega k\vm}\cap\cB^\tensor_{\vm}.
\end{align*}
Then the sequence $(\CHI_{ij},\CHI_{ij}')_{i\in[\vm],j\in[k]}$ given $\cR^\tensor(\omega)$ is distributed precisely as a random sequence of $\{\pm1\}^2$ tokens comprising precisely $k\vm\omega_{11},\ldots,k\vm\omega_{-1-1}$ tokens of each type.
Hence, letting
\begin{align*}
\cS^\tensor&=\cbc{\sum_{i=1}^{\vm}\vecone\cbc{\max_{j\in[k]}\vchi_{ij}=-1}=\sum_{i=1}^{\vm}\vecone\cbc{\max_{j\in[k]}\vchi'_{ij}=-1}=\lceil u\vm\rceil},
\end{align*}
we obtain
\begin{align}\label{eqsmmexact}
\Erw[\vE(\omega)\mid\fD]&=|\BAL|^2\pr\brk{\omega(\TAU,\TAU')=\omega\mid\fD}\cdot\pr\brk{\cS_{\vm}^\tensor\mid\cR_{\vm}^\tensor(\omega)}\exp\bc{-2\beta u \vm}.
\end{align}
Since \Lem s~\ref{Lem_1stmmt_entropy} and~\ref{Lemma_2ndmmt_entropy} already yield the first two factors on the r.h.s., we are left to compute $\pr\brk{\cS_{\vm}^\tensor\mid\cR_{\vm}^\tensor(\omega)}$.

As in the first moment calculations we will calculate this conditional probability via Bayes' formula.
Hence, we begin by computing the unconditional probabilities of $\cR_{\vm}^\tensor(\omega)$ and $\cS_{\vm}^\tensor$.

\begin{lemma}\label{Lemma_smmt_R}
We have $\pr\brk{\cR_{\vm}^\tensor(\omega)\mid\fD}=\exp(o(n))\binom{k\vm}{\omega k\vm,(1/2-\omega)k\vm,(1/2-\omega)k\vm,\omega k\vm}(p_{11}p_{-1-1})^{\omega k\vm}(p_{1-1}p_{-11})^{(1/2-\omega)k\vm}.  $
\end{lemma}
\begin{proof}
This is an immediate consequence of the fact that the random pairs $(\CHI_{ij},\CHI_{ij}')_{i,j}$ are mutually independent and that the individual pairs $(\CHI_{ij},\CHI_{ij}')$ are drawn from the distribution \eqref{eqvchi}.
\end{proof}

To calculate the probability of $\cS_{\vm}^\tensor$ we need to introduce one extra parameter.
Namely, for $s\in[0,u]$ we let
\begin{align*}
\cS_{\vm}^\tensor(s)&=\cS_{\vm}^\tensor\cap\cbc{\sum_{i=1}^{\vm}\vecone\cbc{\max_{j\in[k]}\vchi_{ij}=\max_{j\in[k]}\vchi'_{ij}=-1}=\lceil s\vm\rceil}.
\end{align*}
Thus, $s$ specifies the fraction of clauses that receive $(-1,-1)$ tokens only.
Of course, we have the bound
\begin{align}\label{eqLemma_smmt_S_bd}
\pr\brk{\cS_{\vm}^\tensor}\leq \vm \max_{s\in[0,u]}\pr\brk{\cS_{\vm}^\tensor(s)}.
\end{align}

\begin{lemma}\label{Lemma_smmt_S}
For any $s\in[0,u]$ we have
\begin{align*}
\pr\brk{\cS_{\vm}^\tensor(s)\mid\fD}
&=\exp(o(n))\binom{\vm}{s\vm,(u-s)\vm,(u-s)\vm,(1-2u+s)\vm}\\
&\qquad\cdot
p_{-1-1}^{ks\vm}\bc{(p_{-1-1}+p_{1-1})^k-p_{-1-1}^k}^{2(u-s)\vm}\bc{1-(p_{1-1}+p_{-1-1})^k-(p_{-11}+p_{-1-1})^k+p_{-1-1}^k}^{(1-2u+s)\vm}.
\end{align*}
\end{lemma}
\begin{proof}
The vectors $(\CHI_{i1},\CHI_{i1}',\ldots,\CHI_{ik},\CHI_{ik}')_{i\in[\vm]}$ are mutually independent.
Moreover, \eqref{eqvchi} provides that
\begin{align}
\pr\brk{\CHI_{i1}=\CHI_{i1}'=\cdots=\CHI_{ik}=\CHI_{ik}'=-1}&=p_{-1-1}^k,\label{eqCHI1}\\
\pr\brk{\CHI_{i1}=\cdots=\CHI_{ik}=-1\wedge\exists j\in[k]:\CHI_{ij}'=1}&= (p_{-11}+p_{-1-1})^k-p_{-1-1}^k,\label{eqCHI2}\\
\pr\brk{\CHI_{i1}'=\cdots=\CHI_{ik}'=-1\wedge\exists j\in[k]:\CHI_{ij}=1}&=(p_{-11}+p_{-1-1})^k-p_{-1-1}^k, \label{eqCHI3}\\
\pr\brk{\exists j,l\in[k]:\CHI_{ij}= \CHI_{il}'=1}&=1-(p_{1-1}+p_{-1-1})^k-(p_{-11}+p_{-1-1})^k+p_{-1-1}^k.\label{eqCHI3}
\end{align}
Since $\cS_{\vm}^\tensor(s)$ asks that \eqref{eqCHI1} occur for $\lceil s\vm\rceil$ indices $i\in[\vm]$, that \eqref{eqCHI2} and \eqref{eqCHI3} occur for $\lceil u\vm\rceil-\lceil s\vm\rceil$ indices and that, naturally, \eqref{eqCHI3} occur for the remaining $\vm-2\lceil u\vm\rceil+\lceil s\vm\rceil$ indices, we obtain the assertion.
\end{proof}

As in the first moment calculation we are going to apply Bayes' rule
\begin{align}\label{eqsmmBayes}
\pr\brk{\cS_{\vm}^\tensor\mid\cR_{\vm}^\tensor(\omega)}&=\frac{\pr\brk{\cS_{\vm}^\tensor(\omega)}}{\pr\brk{\cR_{\vm}^\tensor(\omega)}}\cdot\pr\brk{\cR_{\vm}^\tensor(\omega)\mid\cS_{\vm}^\tensor}
\end{align}
to calculate the probability on the l.h.s.\ 
But while we easily obtained succinct expressions for the unconditional probabilities $\pr\brk{\cS_{\vm}^\tensor(\omega)}$ and $\pr\brk{\cR_{\vm}^\tensor(\omega)}$ for any choice of $p_{11},\ldots,p_{-1-1}$, calculating $\pr\brk{\cR_{\vm}^\tensor(\omega)\mid\cS_{\vm}^\tensor}$ for general choices of these parameters appears to be tricky.
Yet it turns out that for a diligent choice of the $p_{\pm1\,\pm1}$ we will obtain $\pr\brk{\cR_{\vm}^\tensor(\omega)\mid\cS_{\vm}^\tensor}=\exp(o(n))$.
To work out the these $p_{\pm1\,\pm1}$, we need to calculate the conditional expectation of the number of pairs $(\CHI_{ij},\CHI_{ij}')$ for which specific $(\pm1,\pm1)$-values materialise given $\cS_{\vm}^\tensor$.
Thus,  for $v,w=\pm1$ let
\begin{align}\label{eqXvw}
\vX_{vw}&=\frac1{k\vm}\sum_{i=1}^{\vm}\sum_{j=1}^{k}\vecone\cbc{\vchi_{ij}=v,\vchi'_{ij}=w}.
\end{align}
For brevity we introduce  $p_{1}=p_{11}+p_{1-1}=p_{11}+p_{-11}$ and $p_{-1}=p_{-11}+p_{-1-1}=p_{1-1}+p_{-1-1}$.

\begin{lemma}\label{Lemma_2ndmmt_cond_ex}
\Whp\ for any $s\in[0,u]$ we have
\begin{align}\label{eqLemma_2ndmmt_cond_ex1}
\Erw[\vX_{11}\mid\fD,\cS_{\vm}^\tensor(s)]&=\frac{(1-2u+s)p_{11}}{1-2p_{-1}^k+p_{-1-1}^k}+O(1/n),\\
\Erw[\vX_{1-1}\mid\fD,\cS_{\vm}^\tensor(s)]&=\frac{(u-s)p_{1-1}p_{-1}^{k-1}}{p_{-1}^k-p_{-1-1}^k}+\frac{(1-2u+s)p_{1-1}(1-p_{-1}^{k-1})}{1-2p_{-1}^k+p_{-1-1}^k}+O(1/n).
\label{eqLemma_2ndmmt_cond_ex2}
\end{align}
\end{lemma}
\begin{proof}
By linearity of expectation we just need to contemplate the $\{\pm1\}^2$-sequence $(\CHI_{11},\CHI_{11}'),\ldots,(\CHI_{1k},\CHI_{1k}')$ that represents the first clause.
Given $\cS_{\vm}^\tensor(s)$ the event $\max_{j\in[k]}\CHI_{1j}=\max_{j\in[k]}\CHI_{1j}'=1$ has probability $1-2u+s+O(1/\vm)=1-2u+s+O(1/n)$.
Furthermore, the conditional probability that $\CHI_{11}=\CHI_{11}'=1$ given $\max_{j\in[k]}\CHI_{1j}=\max_{j\in[k]}\CHI_{1j}'=1$ equals 
$p_{11}/(1-2p_{-1}^k+p_{-1-1}^k)$ because the pairs $(\CHI_{1j},\CHI_{1j}')_j$ are mutually independent, whence we obtain\eqref{eqLemma_2ndmmt_cond_ex1}.

Similar steps yield \eqref{eqLemma_2ndmmt_cond_ex2}.
For with probability $u-s+O(1/\vm)$ we have  $\max_{j\in[k]}\CHI_{1j}=-\max_{j\in[k]}\CHI_{1j}'=1$ and
given this event the probability that $(\CHI_{11},\CHI_{11}')=(1,-1)$ equals $p_{1-1}p_{-1}^{k-1}/(p_{-1}^k-p_{-1-1}^k)$; hence the first summand in~\eqref{eqLemma_2ndmmt_cond_ex2}.
Further, as in the previous paragraph, the event $\max_{j\in[k]}\CHI_{1j}=\max_{j\in[k]}\CHI_{1j}'=1$ has probability $1-2u+s+O(1/\vm)$ and then the conditional probability of $(\CHI_{11},\CHI_{11}')=(1,-1)$  works out to be  $p_{1-1}(1-p_{-1}^{k-1})/(1-2p_{-1}^k+p_{-1-1}^k)$.
\end{proof}

\begin{lemma}\label{Lemma_2ndmmt_p}
For any $r\in[1/4-2^{-k/3},1/4+2^{-k/3}],s\in[0,u]$ the system of equations
\begin{align}\label{eqLemma_2ndmmt_p}
\frac{(1-2u+s)p_{11}}{1-2p_{-1}^k+p_{-1-1}^k}&=r,&
\frac{(u-s)p_{1-1}p_{-1}^{k-1}}{p_{-1}^k-p_{-1-1}^k}+\frac{(1-2u+s)p_{1-1}(1-p_{-1}^{k-1})}{1-2p_{-1}^k+p_{-1-1}^k}&=1/2,\\
p_{1-1}&=p_{-11},&
p_{11}+p_{1-1}+p_{-11}+p_{-1-1}&=1\label{eqLemma_2ndmmt_p2}
\end{align}
possesses a unique solution $p_{\pm1\pm1}\in[\frac14-2^{-k/3-1},\frac14+2^{-k/3-1}]$.
\end{lemma}
\begin{proof}
The proof is based on the inverse function theorem.
Implementing the last two constraints \eqref{eqLemma_2ndmmt_p2}, we substitute $p_{-11}=p_{1-1}$ and $p_{-1-1}=1-2p_{1-1}-p_{11}$.
Hence, the remaining free variables are $s,p_{11},p_{1-1}$ and we need to work out the Jacobi matrix of the map
\begin{align*}
g&:\RR^3\to\RR^3,\qquad(p_{11},p_{1-1},s)\mapsto
(g_1,g_2,g_3)\qquad\mbox{where}\\
g_1&=\frac{(1-2u+s)p_{11}}{1-2p_{-1}^k+p_{-1-1}^k},\quad
g_2=\frac{(u-s)p_{1-1}p_{-1}^{k-1}}{p_{-1}^k-p_{-1-1}^k}+\frac{(1-2u+s)p_{1-1}(1-p_{-1}^{k-1})}{1-2p_{-1}^k+p_{-1-1}^k},\quad
g_3=s.
\end{align*}
The partial derivatives of $g_1$ come to
\begin{align}\label{eqImpossible1}
\frac{\partial}{\partial p_{11}}\frac{(1-2u+s)p_{11}}{1-2p_{-1}^k+p_{-1-1}^k}&=\frac{1-2u+s}{1-2p_{-1}^k+p_{-1-1}^k}-\frac{k(1-2u+s)p_{11}(2p_{-1}^{k-1}-p_{-1-1}^{k-1})}{\bc{1-2p_{-1}^k+p_{-1-1}^k}^2},\\
\frac{\partial}{\partial p_{1-1}}\frac{(1-2u+s)p_{11}}{1-2p_{-1}^k+p_{-1-1}^k}&=-\frac{2(1-2u+s)k(p_{-1}^{k-1}-p_{-1-1}^{k-1})}{\bc{1-2p_{-1}^k+p_{-1-1}^k}^2},\\
\frac{\partial}{\partial s}\frac{(1-2u+s)p_{11}}{1-2p_{-1}^k+p_{-1-1}^k}&=\frac{p_{11}}{1-2p_{-1}^k+p_{-1-1}^k}.\label{eqImpossible2}
\end{align}
Moreover, the first summand of $g_2$ has derivatives
\begin{align}
\frac{\partial}{\partial p_{11}}\frac{(u-s)p_{1-1}p_{-1}^{k-1}}{p_{-1}^k-p_{-1-1}^k}
&=-\frac{(k-1)(u-s)p_{1-1}p_{-1}^{k-2}}{p_{-1}^k-p_{-1-1}^k}+\frac{k(u-s)p_{1-1}p_{-1}^{k-1}(p_{-1}^{k-1}-p_{-1-1}^{k-1})}{(p_{-1}^k-p_{-1-1}^k)^2},\label{eqImpossible3}\\
\frac{\partial}{\partial p_{1-1}}\frac{(u-s)p_{1-1}p_{-1}^{k-1}}{p_{-1}^k-p_{-1-1}^k}&= \frac{(u-s)(p_{-1}^{k-1}-(k-1)p_{1-1}p_{-1}^{k-2})}{p_{-1}^k-p_{-1-1}^k}+\frac{k(u-s)p_{1-1}p_{-1}^{k-1}(p_{-1}^{k-1}-2p_{-1-1}^{k-1})}{(p_{-1}^k-p_{-1-1}^k)^2},\label{eqImpossible4}\\
\frac{\partial}{\partial s}\frac{(u-s)p_{1-1}p_{-1}^{k-1}}{p_{-1}^k-p_{-1-1}^k}&=-\frac{p_{1-1}p_{-1}^{k-1}}{p_{-1}^k-p_{-1-1}^k} .
\label{eqImpossible5}
\end{align}
Finally, for the second summand we obtain
\begin{align}
\frac{\partial}{\partial p_{11}}\frac{(1-2u+s)p_{1-1}(1-p_{-1}^{k-1})}{1-2p_{-1}^k+p_{-1-1}^k}&=\frac{(1-2u+s)(k-1)p_{1-1}p_{-1}^{k-2}}{1-2p_{-1}^k+p_{-1-1}^k}-\frac{(1-2u+s)kp_{1-1}(1-p_{-1}^{k-1})(2p_{-1}^{k-1}-p_{-1-1}^{k-1})}{\bc{1-2p_{-1}^k+p_{-1-1}^k}^2},\label{eqImpossible6}\\
\frac{\partial}{\partial p_{1-1}}\frac{(1-2u+s)p_{1-1}(1-p_{-1}^{k-1})}{1-2p_{-1}^k+p_{-1-1}^k}&=\frac{(1-2u+s)\bc{1-p_{-1}^{k-1}+2(k-1)p_{1-1}p_{-1}^{k-2}}}{1-2p_{-1}^k+p_{-1-1}^k}\nonumber\\
&\qquad-\frac{2(1-2u+s)kp_{1-1}(1-p_{-1}^{k-1}))(p_{-1}^{k-1}-p_{-1-1}^{k-1})}{\bc{1-2p_{-1}^k+p_{-1-1}^k}^2},\label{eqImpossible7}\\
\frac{\partial}{\partial s}\frac{(1-2u+s)p_{1-1}(1-p_{-1}^{k-1})}{1-2p_{-1}^k+p_{-1-1}^k}&=\frac{p_{1-1}(1-p_{-1}^{k-1})}{1-2p_{-1}^k+p_{-1-1}^k}.
\label{eqImpossible8}
\end{align}
Hence, for $p_{11}=\frac14+O(k^{-4})$, $p_{1-1}=\frac14+O(k^{-4})$ we obtain
\begin{align}\label{eqDg}
Dg&=\begin{pmatrix}
1+\tilde O(2^{-k})&\tilde O(2^{-k})&p_{11}+\tilde O(2^{-k})\\
\tilde O(2^{-k})&1+\tilde O(2^{-k})&p_{1-1}+\tilde O(2^{-k})\\
0&0&1
\end{pmatrix}
\end{align}
Consequently, \eqref{eqLinAlg} and the inverse function theorem yield
\begin{align}\label{eqDg-1}
Dg^{-1}&=\begin{pmatrix}
1+\tilde O(2^{-k})&\tilde O(2^{-k})&-p_{11}+\tilde O(2^{-k})\\
\tilde O(2^{-k})&1+\tilde O(2^{-k})&-p_{1-1}+\tilde O(2^{-k})\\
0&0&1
\end{pmatrix},
\end{align}
whence the assertion follows.
\end{proof}

\noindent

Let $\fp=\fp(r,s)=(\fp_{11},\fp_{1-1},\fp_{-11},\fp_{-1-1})$ denote the solution to \eqref{eqLemma_2ndmmt_p}--\eqref{eqLemma_2ndmmt_p2} provided by \Lem~\ref{Lemma_2ndmmt_p}.
By construction, $\fp_{-11}=\fp_{1-1}$ and $\fp_{-1-1}=1-2\fp_{1-1}-\fp_{11}$.
Let us make a note of the first and second derivatives of $\fp_{11},\fp_{1-1}$.

\begin{corollary}\label{Cor_2ndmmt_p}
For all $0\leq s\leq u$ and $\omega=\frac12+O(2^{-k/3})$ we have
\begin{align}\label{eqCor_2ndmmt_p_1}
\frac{\partial\fp_{11}}{\partial\omega}&=1+\tilde O(2^{-k}),\frac{\partial\fp_{1-1}}{\partial\omega}=-1+\tilde O(2^{-k}),\qquad\frac{\partial\fp_{11}}{\partial s}=-\fp_{11}+\tilde O(2^{-k}),\ \frac{\partial\fp_{1-1}}{\partial s}=-\fp_{1-1}+\tilde O(2^{-k}),\\
\frac{\partial^2\fp_{11}}{\partial \omega^2}&,\frac{\partial^2\fp_{11}}{\partial\omega \partial s},
\frac{\partial^2\fp_{1-1}}{\partial\omega^2},\frac{\partial^2\fp_{1-1}}{\partial\omega \partial s}
=\tilde O(2^{-k}),\qquad
\frac{\partial^2\fp_{11}}{\partial s^2},\frac{\partial^2\fp_{1-1}}{\partial s^2}=\tilde O(1).\label{eqCor_2ndmmt_p_2}
\end{align}
\end{corollary}
\begin{proof}
The assertions regarding the first derivative are immediate from the Jacobi matrix \eqref{eqDg-1} of $g$.
Furthermore, to obtain the bounds on the second derivatives we need to estimate the derivatives of the entries of $Dg^{-1}$ with respect to $\omega,s$.
Let
\begin{align*}
\fa&=\frac{\partial g_1}{\partial p_{11}}=1+\tilde O(2^{-k}),&\fb&=\frac{\partial g_1}{\partial p_{1-1}}=\tilde O(2^{-k}),&\fc&=\frac{\partial g_1}{\partial s}=-p_{11}+\tilde O(2^{-k}),\\
\fd&=\frac{\partial g_2}{\partial p_{11}}=\tilde O(2^{-k}),&\fe&=\frac{\partial g_2}{\partial p_{1-1}}=1+\tilde O(2^{-k}),&\ff&=\frac{\partial g_2}{\partial s}=-p_{1-1}+\tilde O(2^{-k})
\end{align*}
denote the entries in the first two rows of $Dg$.
Revisiting the explicit expressions \eqref{eqImpossible1}--\eqref{eqImpossible8} for these partial derivatives, we verify that the second partial derivatives satisfy
\begin{align*}
&\frac{\partial\fa}{\partial p_{11}},\frac{\partial\fa}{\partial p_{1-1}},
\frac{\partial\fb}{\partial p_{11}},\frac{\partial\fb}{\partial p_{1-1}},
\frac{\partial\fc}{\partial p_{11}},\frac{\partial\fc}{\partial p_{1-1}},
\frac{\partial\fd}{\partial p_{11}},\frac{\partial\fd}{\partial p_{1-1}},
\frac{\partial\fe}{\partial p_{11}},\frac{\partial\fe}{\partial p_{1-1}},
\frac{\partial\ff}{\partial p_{11}},\frac{\partial\ff}{\partial p_{1-1}}=\tilde O(2^{-k}),\\
&\frac{\partial\fa}{\partial s},\frac{\partial\fd}{\partial s},\frac{\partial\fe}{\partial s}=\tilde O(1),\quad
\frac{\partial\fb}{\partial s}=\tilde O(2^{-k}),\quad
\frac{\partial\fc}{\partial s},\frac{\partial\ff}{\partial s}=0.
\end{align*}
Consequently, using \eqref{eqCor_2ndmmt_p_1} and the chain rule, we obtain
\begin{align}\nonumber
\frac{\partial}{\partial\omega}\brk{\bcfr\fe{\fa\fe-\fb\fd}\big|_{p_{11}=\fp_{11},p_{1-1}=\fp_{1-1}}}
&=\frac{\frac{\partial\fe}{\partial p_{11}}|_{p_{11}=\fp_{11},p_{1-1}=\fp_{1-1}}\frac{\partial\fp_{11}}{\partial\omega}
	+\frac{\partial\fe}{\partial p_{1-1}}|_{p_{11}=\fp_{11},p_{1-1}=\fp_{1-1}}\frac{\partial\fp_{1-1}}{\partial\omega}}{\fa\fe-\fb\fd}\\
	&+\frac{\fe\bc{
		\frac{\partial(\fa\fe-\fb\fd)}{\partial p_{11}}|_{p_{11}=\fp_{11},p_{1-1}=\fp_{1-1}}\frac{\partial\fp_{11}}{\partial\omega}
		+\frac{\partial(\fa\fe-\fb\fd)}{\partial p_{1-1}}|_{p_{11}=\fp_{11},p_{1-1}=\fp_{1-1}}\frac{\partial\fp_{1-1} }{\partial\omega}}
	}{(\fa\fe-\fb\fd)^2}\nonumber\\
&=\tilde O(2^{-k}).\label{eqNasty1}
\end{align}
Similarly,
\begin{align}\label{eqNasty2}
\frac{\partial}{\partial\omega}\brk{\bcfr\fb{\fa\fe-\fb\fd}\big|_{p_{11}=\fp_{11},p_{1-1}=\fp_{1-1}}},
\frac{\partial}{\partial\omega}\brk{\bcfr\fd{\fa\fe-\fb\fd}\big|_{p_{11}=\fp_{11},p_{1-1}=\fp_{1-1}}},
\frac{\partial}{\partial\omega}\brk{\bcfr\fa{\fa\fe-\fb\fd}\big|_{p_{11}=\fp_{11},p_{1-1}=\fp_{1-1}}}&=\tilde O(2^{-k}),\\
\frac{\partial}{\partial\omega}\brk{\bcfr{\fb\ff-\fc\fe}{\fa\fe-\fb\fd}\big|_{p_{11}=\fp_{11},p_{1-1}=\fp_{1-1}}},
\frac{\partial}{\partial\omega}\brk{\bcfr{\fa\ff-\fc\fd}{\fa\fe-\fb\fd}\big|_{p_{11}=\fp_{11},p_{1-1}=\fp_{1-1}}}&=\tilde O(2^{-k}),\label{eqNasty3}\\
\frac{\partial}{\partial s}\brk{\bcfr{\fb\ff-\fc\fe}{\fa\fe-\fb\fd}\big|_{p_{11}=\fp_{11},p_{1-1}=\fp_{1-1}}},
\frac{\partial}{\partial s}\brk{\bcfr{\fa\ff-\fc\fd}{\fa\fe-\fb\fd}\big|_{p_{11}=\fp_{11},p_{1-1}=\fp_{1-1}}}&=\tilde O(1),\label{eqNasty4}
\end{align}

\renewcommand\fa{\mathfrak s}
\renewcommand\fb{\mathfrak t}
\renewcommand\fc{\mathfrak u}
\renewcommand\fd{\mathfrak v}
\renewcommand\fe{\mathfrak w}
\renewcommand\ff{\mathfrak x}

\noindent
Finally, we reminder ourselves of the following elementary formula: for any $\fa,\fb,\fc,\fd,\fe,\ff$ such that $\fa\fe-\fb\fd\neq0$, 
\begin{align}\label{eqLinAlg}
\begin{pmatrix}\fa&\fb&\fc\\\fd&\fe&\ff\\0&0&1\end{pmatrix}^{-1}
&=\begin{pmatrix}
\frac\fe{\fa\fe-\fb\fd}&-\frac{\fb}{\fa\fe-\fb\fd}&\frac{\fb\ff-\fc\fe}{\fa\fe-\fb\fd}\\
-\frac{\fd}{\fa\fe-\fb\fd}&\frac{\fa}{\fa\fe-\fb\fd}&-\frac{\fa\ff-\fc\fd}{\fa\fe-\fb\fd}\\
0&0&1
\end{pmatrix}\enspace .
\end{align}
Combining \eqref{eqNasty1}--\eqref{eqNasty4} with the formula \eqref{eqLinAlg} for the Jacobian of $g^{-1}$, we obtain \eqref{eqCor_2ndmmt_p_2}.
\renewcommand\fa{\mathfrak a}
\renewcommand\fb{\mathfrak b}
\renewcommand\fc{\mathfrak c}
\renewcommand\fd{\mathfrak d}
\renewcommand\fe{\mathfrak e}
\renewcommand\ff{\mathfrak f}
\end{proof}

We next verify that with the choice $p_{\pm1\pm1}=\fp_{\pm1\pm1}$ we may neglect the conditional probability $\pr\brk{\cR_{\vm}(\omega)\mid\fD,\cS(s)}$.

\begin{lemma}\label{Lemma_2ndmmt_R}
\Whp\ at the point $p_{\pm1\pm1}=\fp_{\pm1\pm1}$ we have  $\pr\brk{\cR_{\vm}^\tensor(\omega)\mid\fD,\cS_{\vm}^\tensor(s)}=\exp(o(n)).$
\end{lemma}
\begin{proof}
Given $\cS_{\vm}^\tensor(s)$ the random variables $\vX_{\pm1\pm1}$ from \eqref{eqXvw} can be written as sums of $\vm$ independent random variables.
Indeed, because the pairs $(\CHI_{ij},\CHI_{ij}')$ are identically distributed, instead of conditioning on $\cS_{\vm}^{\tensor}(s)$ we may condition on the event $\cS_{0,\vm}^\tensor(s)$ that
\begin{itemize}
\item $\CHI_{ij}=\CHI_{ij}'=-1$ for $i=1,\ldots,\lceil s\vm\rceil$, $j\in[k]$,
\item $\max_{j\in[k]}\CHI_{ij}=1$ and $\max_{j\in[k]}\CHI_{ij}'=-1$ for $i=\lceil s\vm\rceil+1,\ldots,\lceil u\vm\rceil$,
\item $\max_{j\in[k]}\CHI_{ij}=-1$ and $\max_{j\in[k]}\CHI_{ij}'=1$ for $i=\lceil u\vm\rceil+1,\ldots,2\lceil u\vm\rceil-\lceil s\vm\rceil$,
\item $\max_{j\in[k]}\CHI_{ij}=\max_{j\in[k]}\CHI_{ij}'=1$ for $i=2\lceil u\vm\rceil-\lceil s\vm\rceil,\ldots,\vm$.
\end{itemize}
Evidently, given $\cS_{0,\vm}^\tensor(s)$ the random variables $\vX_{vw}(i)=\frac1{k\vm}\sum_{j=1}^k\vecone\{\CHI_{ij}=v,\CHI_{ij}'=w\}$ with $v,w=\pm1$ are independent for all $i\in[\vm]$.
Moreover, $\vX_{vw}=\sum_{i=1}^{\vm}\vX_{vw}(i)$ and due to \eqref{eqLemma_2ndmmt_cond_ex1}--\eqref{eqLemma_2ndmmt_cond_ex2} the choice $p_{\pm1\pm1}=\fp_{\pm1\pm1}$ ensures that
\begin{align}\label{eqLemma_2ndmmt_R_1}
\Erw[\vX_{vw}\mid\cS_{0,\vm}^\tensor(s)]&={\omega_{vw}}+O(1/n).
\end{align}
Hence, assuming that $\vm\sim dn/k$ is about as large as its expectation, we can apply the {local limit theorem} for sums of independent random variables to the conditional random variables $\vX_{vw}$ given $\cS_{0,\vm}$ to conclude that
\begin{align*}
\pr\brk{\cR_{\vm}^\tensor(\omega)\mid\fD,\cS_{\vm}^\tensor(s)}&=
\pr\brk{\forall v,w\in\{\pm1\}:\vX_{vw}=\omega_{vw}\mid\fD,\cS_{\vm}^\tensor(s)}=\Omega(n^{-3/2})=\exp(o(n)),
\end{align*}
thereby completing the proof.
\end{proof}

Combining \Lem s~\ref{Lemma_smmt_R}, \ref{Lemma_smmt_S} and~\ref{Lemma_2ndmmt_R}, we finally obtain a handy bound on $\pr\brk{\vE(\omega)\mid\fD}$.
Indeed, keeping in mind our convention that $p_{-11}=p_{1-1}$ and $p_{-1-1}=1-p_{11}-2p_{1-1}$, we introduce
\begin{align}\nonumber
\mathfrak{F}(\omega,s,p_{11},p_{1-1})&=-\KL{s,u-s,u-s,1-2u+s}{p_{-1-1}^k,p_{-1}^k-p_{-1-1}^k,p_{-1}^k-p_{-1-1}^k,1-2p_{-1}^k+p_{-1-1}^k}\\
&\qquad+k\KL{\omega,1/2-\omega,1/2-\omega,\omega}{p_{11},p_{1-1},p_{1-1},p_{-1-1}}.
\label{eqfomegas}
\end{align}
Moreover, with $\fp=\fp(\omega,s)$ the solution to \eqref{eqLemma_2ndmmt_p}--\eqref{eqLemma_2ndmmt_p2} from \Lem~\ref{Lemma_2ndmmt_p}, we let
\begin{align}\label{eqFomegas}
F(\omega,s)&=\Fra(\omega,s,\fp_{11},\fp_{1-1}).
\end{align}

\begin{corollary}\label{Cor_handy}
\Whp\ we have $\vm^{-1}\log
\pr\brk{\cS_{\vm}^\tensor\mid\cR_{\vm}^\tensor(\omega),\fD}\leq \max_{s\in[0,u]}F(\omega,s)+o(1).$
\end{corollary}

Thus, the next item on the agenda is to find the stationary points of $F(\omega,s)$ for $\omega$ close to $1/2$.
We begin by exhibiting an explicit stationary point.

\begin{lemma}\label{Lemma_2ndmmt_max}
We have $DF(1/4,u^2)=0$ and $\frac{d}{k} F(1/4,u^2)=-\frac{2(k-1)d}{k}\log 2 -d\log(p(1-p))+\frac{2d}{k} \log p +   \frac{2d}{k} \beta u.  $
\end{lemma}
\begin{proof}
With $p$ the solution to \eqref{maximal_p}, we verify directly that at the point $\omega=1/4,s=u^2$ the solution to \eqref{eqLemma_2ndmmt_p} reads
\begin{align}\label{eqLemma_2ndmmt_max0}
\fp_{11}&=p^2,&\fp_{1-1}&=\fp_{-11}=p(1-p),&\fp_{-1-1}&=(1-p)^2.
\end{align}
Hence, we obtain the formula for $F(1/4,u^2)$ by simply plugging \eqref{eqLemma_2ndmmt_max0} into \eqref{eqFomegas}.
Moreover, using the formulas $\frac{\partial}{\partial y}y\log\frac yz=1+\log\frac yz$, $\frac{\partial}{\partial z}y\log\frac yz=-\frac yz$ we compute
\begin{align}\label{eqLemma_2ndmmt_max1}
\frac{\partial \Fra}{\partial \omega}&=
k\brk{2\log2+\log\frac \omega{p_{11}}+\log\frac\omega{1-p_{11}-2p_{1-1}}-2\log\frac{1-2\omega}{p_{1-1}}}.
\end{align}
Substituting \eqref{eqLemma_2ndmmt_max0} into \eqref{eqLemma_2ndmmt_max1}, we find
\begin{align}\label{eqLemma_2ndmmt_max2}
	\frac{\partial \Fra}{\partial\omega}&=\bigg|_{\omega=1/4,p_{11}=p^2,p_{1-1}=p(1-p)}=0,&&\mbox{and similarly}\\
\frac{\partial \Fra}{\partial s}&=
-\log\frac{s}{(1-2p_{1-1}-p_{11})^k}+2\log\frac{u-s}{(1-p_{1-1}-p_{11})^k-(1-2p_{1-1}-p_{11})^k}\nonumber\\
&\qquad-\log\frac{1-2u+s}{1-2(1-p_{11}-p_{1-1})^k+(1-2p_{1-1}-p_{11})^k}.
\label{eqLemma_2ndmmt_max3}
\end{align}
As substituting $s=u^2$, $p_{11}=p^2$ and $p_{1-1}=p(1-p)$ into the last expression yields zero, we conclude that
\begin{align}\label{eqLemma_2ndmmt_max4}
\frac{\partial \Fra}{\partial s}\bigg|_{s=u^2,p_{11}=p^2,p_{1-1}=p(1-p)}=0.
\end{align}
Moreover,
\begin{align}\nonumber
\frac{\partial \Fra}{\partial p_{11}}&=-\frac{ks}{1-p_{11}-2p_{1-1}}-\frac{2(u-s)k\bc{(1-2p_{1-1}-p_{11})^{k-1}-(1-p_{11}-p_{1-1})^{k-1}}}{(1-2p_{1-1}-p_{11})^k-(1-p_{11}-p_{1-1})^k}\\
&\quad-\frac{(1-2u+s)k\bc{(1-2p_{1-1}-p_{11})^{k-1}-2(1-p_{11}-p_{1-1})^{k-1}}}{1-2(1-p_{11}-p_{1-1})^k+(1-2p_{1-1}- p_{11})^k}
-\frac{(1-2p_{11}-2p_{1-1})k\omega}{p_{11}(1-2p_{1-1}-p_{11})}.
\label{eqLemma_2ndmmt_max5}
\end{align}
Hence, 
\begin{align*}
\frac{\partial \Fra}{\partial p_{11}}\bigg|_{\substack{\omega=1/4,s=u^2\\p_{11}=p^2\\p_{1-1}=p(1-p)}}&
=-\frac{u^2k}{(1-p)^2}-\frac{2u(1-u)k\bc{1-(1-p)^{k-1}}}{(1-p)(1-(1-p)^{k})}
+\frac{k(1-u)^2(1-p)^{k-1}\bc{2-(1-p)^{k-1}}}{(1-(1-p)^k)^2}
-\frac{k(1-2p)}{4p^2(1-p)^2}.
\end{align*}
Further, recalling that $u=(1-2p)/(2p(\exp(\beta)-1))$ and that $p$ is the solution to \eqref{maximal_p} and hence $ \eul^\beta = \frac{(1-p)^k}{2p -1 + (1-p)^k } $, we obtain
\begin{align}\label{eqLemma_2ndmmt_max6}
\frac{\partial \Fra }{\partial p_{11}}\bigg|_{\substack{\omega=1/4,s=u^2\\p_{11}=p^2\\p_{1-1}=p(1-p)}}&=0.
\end{align}
In addition,
\begin{align}
\nonumber
\frac{\partial \Fra }{\partial p_{1-1}}&=
-\frac{2ks}{1-2p_{1-1}-p_{11}}-\frac{2k(u-s)\bc{2(1-2p_{1-1}-p_{11})^{k-1}-(1-p_{11}-p_{1-1})^{k-1}}}{(1-2p_{1-1}-p_{11})^k-(1-p_{11}-p_{1-1})^k}\\
&\quad-\frac{2k(1-2u+s)\bc{(1-2p_{1-1}-p_{11})^{k-1}-(1-p_{11}-p_{1-1})^{k-1}}}{1-2(1-p_{11}-p_{1-1})^k+(1-2p_{1-1}-p_{11})^k}
+
k\bc{\frac{2\omega}{1-2p_{1-1}-p_{11}}-\frac{1-2\omega}{p_{1-1}}}.
\label{eqLemma_2ndmmt_max7}
\end{align}
Hence, using $u=(1-2p)/(2p(\exp(\beta)-1))$ and \eqref{maximal_p}, we obtain
\begin{align}\label{eqLemma_2ndmmt_max8}
\frac{\partial \Fra }{\partial p_{1-1}}\bigg|_{\substack{\omega=1/4,s=u^2\\p_{11}=p^2\\p_{1-1}=p(1-p)}}&=0.
\end{align}
Combining  \eqref{eqLemma_2ndmmt_max2}, \eqref{eqLemma_2ndmmt_max4}, \eqref{eqLemma_2ndmmt_max6} and~\eqref{eqLemma_2ndmmt_max8} with the chain rule, we conclude that $DF(1/4,u^2)=0$.
\end{proof}

We are now going to compare $\max_{s\in[0,u]}F(\omega,s)$ for $\omega$ close to $1/4$ with $F(1/4,u^2)$.
To this end we need to get a handle on the value of the maximiser $s$ of $F(\omega,s)$ for a given $\omega$.
Hence, we investigate the second partial derivatives of the function $\Fra$ from \eqref{eqfomegas}.
Let
\begin{align*}
g(\omega,p_{11},p_{1-1})&=\KL{\omega,1/2-\omega,1/2-\omega,\omega}{p_{11},p_{1-1},p_{1-1},p_{-1-1}},\\
h(s,p_{11},p_{1-1})&=-\KL{s,u-s,u-s,1-2u+s}{p_{-1-1}^k,p_{-1}^k-p_{-1-1}^k,p_{-1}^k-p_{-1-1}^k,1-2p_{-1}^k+p_{-1-1}^k}
\end{align*}
denote the two constituent terms of \eqref{eqfomegas}.

\begin{lemma}\label{Claim_2ndmmt_2ndderiv_4}
For $\omega=\frac14+\tilde O(2^{-k/2})$ and $0\leq s\leq u$ we have
\begin{align*}
\frac{\partial^2g}{\partial \omega^2}\bigg|_{\substack{p_{11}=\fp_{11}\\p_{1-1}=\fp_{1-1}}}
+\frac{\partial^2g}{\partial p_{11}^2}\bigg|_{\substack{p_{11}=\fp_{11}\\p_{1-1}=\fp_{1-1}}}\bcfr{\partial\fp_{11}}{\partial \omega}^2
+\frac{\partial^2g}{\partial p_{1-1}^2}\bigg|_{\substack{p_{11}=\fp_{11}\\p_{1-1}=\fp_{1-1}}}
	\bcfr{\partial\fp_{1-1}}{\partial\omega}^2
+2\frac{\partial^2g}{\partial\omega\partial p_{11}}\bigg|_{\substack{p_{11}=\fp_{11}\\p_{1-1}=\fp_{1-1}}}\frac{\partial\fp_{11}}{\partial\omega}\\
\qquad+2\frac{\partial^2g}{\partial \omega\partial p_{1-1}}\bigg|_{\substack{p_{11}=\fp_{11}\\p_{1-1}=\fp_{1-1}}}\frac{\partial\fp_{1-1}}{\partial \omega}
+2\frac{\partial^2g}{\partial p_{11}\partial p_{1-1}}\bigg|_{\substack{p_{11}=\fp_{11}\\p_{1-1}=\fp_{1-1}}}\frac{\partial\fp_{11}}{\partial \omega}\frac{\partial\fp_{1-1}}{\partial\omega}&=\tilde O(4^{-k}).
\end{align*}
\end{lemma}
\begin{proof}
Direct calculations reveal
\begin{align}
\frac{\partial^2g}{\partial\omega^2}&=\frac{2}{\omega(1-2\omega)},\quad
\frac{\partial^2g}{\partial\omega\partial p_{11}}=-\frac{1-2p_{1-1}-2p_{11}}{p_{11}(1-2p_{1-1}-p_{11})},\quad
\frac{\partial^2g}{\partial\omega\partial p_{1-1}}=\frac{2(1-p_{1-1}-p_{11})}{p_{1-1}(1-2p_{1-1}-p_{11})},\label{eqClaim_2ndmmt_2ndderiv_4_1}\\
\frac{\partial^2g}{\partial p_{11}^2}&=\frac{\omega(1-2p_{11}-2p_{1-1}+2p_{11}^2+4p_{11}p_{1-1})}{p_{11}^2(1-2p_{1-1}-p_{11})^2},\label{eqClaim_2ndmmt_2ndderiv_4_2}\\
\frac{\partial^2g}{\partial p_{1-1}^2}&=\frac{4(1-\omega)p_{1-1}^2-4(1-2\omega)(1-p_{11})p_{1-1}+(1-2\omega)(p_{11}-1)^2}{p_{1-1}^2(1-2p_{1-1}-p_{11})^2},\quad
\frac{\partial^2g}{\partial p_{11}\partial p_{1-1}}=\frac{2\omega}{(1-2p_{1-1}-p_{11})^2}.
\label{eqClaim_2ndmmt_2ndderiv_4_3}
\end{align}
Moreover, since \Cor~\ref{eqCor_2ndmmt_p_1} shows that
$\frac{\partial p_{11}}{\partial\omega}=1+\tilde O(2^{-k})$, $\frac{\partial p_{1-1}}{\partial\omega}=-1+\tilde O(2^{-k})$, we obtain
\begin{align}\label{eqClaim_2ndmmt_2ndderiv_4_4}
p_{11}&=\omega+\tilde O(2^{-k}),&
p_{1-1}&=\frac12-\omega+\tilde O(2^{-k}).
\end{align}
Substituting \eqref{eqClaim_2ndmmt_2ndderiv_4_4} into \eqref{eqClaim_2ndmmt_2ndderiv_4_1}--\eqref{eqClaim_2ndmmt_2ndderiv_4_3} yields the assertion.
\end{proof}

\begin{lemma}\label{Lemma_2ndmmt_s}
For $0\leq s\leq u$ and $\omega=\frac14+\tilde O(2^{-k/2})$ we have $\frac{\partial^2}{\partial s^2}F(\omega,s)=-\Omega(1/s)$.
\end{lemma}
\begin{proof}
Combining \eqref{eqClaim_2ndmmt_2ndderiv_4_2} and \eqref{eqClaim_2ndmmt_2ndderiv_4_3} with \Cor~\ref{Cor_2ndmmt_p}, we find
\begin{align}
\frac{\partial^2g}{\partial p_{11}^2}\bigg|_{\substack{p_{11}=\fp_{11}\\p_{1-1}=\fp_{1-1}}}\bcfr{\partial\fp_{11}}{\partial s}^2
+\frac{\partial^2g}{\partial p_{1-1}^2}\bigg|_{\substack{p_{11}=\fp_{11}\\p_{1-1}=\fp_{1-1}}}
	\bcfr{\partial\fp_{1-1}}{\partial s}^2
+2\frac{\partial^2g}{\partial p_{11}\partial p_{1-1}}\bigg|_{\substack{p_{11}=\fp_{11}\\p_{1-1}=\fp_{1-1}}}\frac{\partial\fp_{11}}{\partial s}\frac{\partial\fp_{1-1}}{\partial s}&=\tilde O(1).
\label{eqLemma_2ndmmt_s0}
\end{align}
Moreover, we compute
\begin{align}\label{eqLemma_2ndmmt_s1}
\frac{\partial^2 h}{\partial s^2}&=-\frac 1s-\frac2{u-s}-\frac1{1-2u+s},&
\frac{\partial^2 h}{\partial p_{11}\partial s},
\frac{\partial^2 h}{\partial p_{1-1}\partial s},
\frac{\partial^2 h}{\partial p_{11}^2},
\frac{\partial^2 h}{\partial p_{1-1}^2},
\frac{\partial^2 h}{\partial p_{11}\partial p_{1-1}}&=\tilde O(1).
\end{align}
Since $s\leq u=\tilde O(2^{-k})$, the first term in \eqref{eqLemma_2ndmmt_s1} is of order $-\Omega(1/s)=-\tilde\Omega(2^k)$.
Therefore, \eqref{eqLemma_2ndmmt_s1} and \Cor~\ref{Cor_2ndmmt_p} show 
\begin{align}\nonumber
\frac{\partial^2h}{\partial s^2}
+\frac{\partial^2h}{\partial p_{11}^2}\bigg|_{\substack{p_{11}=\fp_{11}\\p_{1-1}=\fp_{1-1}}}\bcfr{\partial\fp_{11}}{\partial s}^2
+\frac{\partial^2h}{\partial p_{1-1}^2}\bigg|_{\substack{p_{11}=\fp_{11}\\p_{1-1}=\fp_{1-1}}}
	\bcfr{\partial\fp_{1-1}}{\partial s}^2
+2\frac{\partial^2h}{\partial s\partial p_{11}}\bigg|_{\substack{p_{11}=\fp_{11}\\p_{1-1}=\fp_{1-1}}}\frac{\partial\fp_{11}}{\partial s}\\
\qquad+2\frac{\partial^2h}{\partial s\partial p_{1-1}}\bigg|_{\substack{p_{11}=\fp_{11}\\p_{1-1}=\fp_{1-1}}}\frac{\partial\fp_{1-1}}{\partial s}
+2\frac{\partial^2h}{\partial p_{11}\partial p_{1-1}}\bigg|_{\substack{p_{11}=\fp_{11}\\p_{1-1}=\fp_{1-1}}}\frac{\partial\fp_{11}}{\partial s}\frac{\partial\fp_{1-1}}{\partial s}&=-\tilde\Omega(1/s).
\label{eqLemma_2ndmmt_s2}
\end{align}
Combining \eqref{eqLemma_2ndmmt_s0} and \eqref{eqLemma_2ndmmt_s2}, we see that
\begin{align}\nonumber
\frac{\partial^2 \Fra }{\partial s^2}
+\frac{\partial^2 \Fra}{\partial p_{11}^2}\bigg|_{\substack{p_{11}=\fp_{11}\\p_{1-1}=\fp_{1-1}}}\bcfr{\partial\fp_{11}}{\partial s}^2
+\frac{\partial^2 \Fra }{\partial p_{1-1}^2}\bigg|_{\substack{p_{11}=\fp_{11}\\p_{1-1}=\fp_{1-1}}}
	\bcfr{\partial\fp_{1-1}}{\partial s}^2
+2\frac{\partial^2 \Fra}{\partial\omega\partial p_{11}}\bigg|_{\substack{p_{11}=\fp_{11}\\p_{1-1}=\fp_{1-1}}}\frac{\partial\fp_{11}}{\partial s}\\
\qquad+2\frac{\partial^2 \Fra}{\partial s\partial p_{1-1}}\bigg|_{\substack{p_{11}=\fp_{11}\\p_{1-1}=\fp_{1-1}}}\frac{\partial\fp_{1-1}}{\partial s}
+2\frac{\partial^2 \Fra }{\partial p_{11}\partial p_{1-1}}\bigg|_{\substack{p_{11}=\fp_{11}\\p_{1-1}=\fp_{1-1}}}\frac{\partial\fp_{11}}{\partial s}\frac{\partial\fp_{1-1}}{\partial s}&=-\tilde\Omega(1/s).
\label{eqLemma_2ndmmt_s3}
\end{align}
Furthermore, the expressions \eqref{eqLemma_2ndmmt_max5} and~\eqref{eqLemma_2ndmmt_max7} for the partial derivatives of $\Fra$ and \Cor~\ref{Cor_2ndmmt_p} yield
\begin{align}\label{eqLemma_2ndmmt_s4}
\frac{\partial \Fra}{\partial p_{11}}\bigg|_{\substack{p_{11}=\fp_{11}\\p_{1-1}=\fp_{1-1}}}
\frac{\partial^2\fp_{11}}{\partial s^2},
\frac{\partial \Fra }{\partial p_{1-1}}\bigg|_{\substack{p_{11}=\fp_{11}\\p_{1-1}=\fp_{1-1}}}
\frac{\partial^2\fp_{1-1}}{\partial s^2}&=\tilde O(2^{-k}).
\end{align}
Hence, combining \eqref{eqLemma_2ndmmt_s3} and \eqref{eqLemma_2ndmmt_s4} with \Faadi's rule, we conclude that $\frac{\partial^2}{\partial s^2}F(\omega,s)=-\Omega(1/s)$.
\end{proof}

\begin{corollary}\label{Cor_2ndmmt_s}
For any $\omega=\frac14+\tilde O(2^{-k/2})$ the function $s\in[0,u]\mapsto F(\omega,s)$ attains its unique maximum at $s=\tilde O(4^{-k})$.
\end{corollary}
\begin{proof}
We consider the first derivative  $\frac{\partial}{\partial s}F(\omega,s)$ for $\omega=\frac14+\tilde O(2^{-k/2})$.
The partial derivatives of $\Fra$, which we computed in \eqref{eqLemma_2ndmmt_max1}, \eqref{eqLemma_2ndmmt_max3}, \eqref{eqLemma_2ndmmt_max5} and~\eqref{eqLemma_2ndmmt_max7}, 
satisfy
\begin{align}\nonumber
\frac{\partial \Fra}{\partial s}\bigg|_{\substack{p_{11}=\fp_{11}\\p_{1-1}=\fp_{1-1}}}
&=
-\log\frac{s}{4^{-k}+\tilde O(8^{-k})}+2\log\frac{u-s}{2^{-k}+\tilde O(4^{-k})}-\log\frac{1-2u+s}{1+\tilde O(2^{-k})},\\
\frac{\partial \Fra}{\partial p_{11}}\bigg|_{\substack{p_{11}=\fp_{11}\\p_{1-1}=\fp_{1-1}}},
\frac{\partial \Fra }{\partial p_{1-1}}\bigg|_{\substack{p_{11}=\fp_{11}\\p_{1-1}=\fp_{1-1}}}&=\tilde O(2^{-k}).\label{eqCor_2ndmmt_s_10}
\end{align}
Hence, \Cor~\ref{Cor_2ndmmt_p} and the chain rule yield
\begin{align}\label{eqFderivs}
\frac{\partial F(\omega,s)}{\partial s}
&=\log\frac{(u-s)^2}{s(1-2u+s)}+\tilde O(2^{-k}).
\end{align}
Since $u=\tilde O(2^{-k})$, \eqref{eqFderivs} shows that the equation $\frac{\partial F(\omega,s)}{\partial s}=0$ is satisfied only for $s=\tilde O(4^{-k})$.
Thus, the assertion follows from \Lem~\ref{Lemma_2ndmmt_s}.
\end{proof}

Having estimated the maximiser $s$ of $F(\omega,s)$, we now bound the Hessian $D^2F(\omega,s)$ of $F(\omega,s)$.

\begin{lemma}\label{Lemma_2ndmmt_2ndderiv}
We have $D^2F(\omega,s)\preceq\tilde O(4^{-k})\id$ for all $\omega=\frac14+\tilde O(2^{-k/2})$ and $s=\tilde O(4^{-k})$.
\end{lemma}



\noindent
The proof requires two intermediate steps.

\begin{claim}\label{Claim_2ndmmt_2ndderiv_3}
For $\omega=\frac14+\tilde O(2^{-k/2})$ and $s\leq u$ we have
\begin{align*}
\frac{\partial \Fra}{\partial p_{11}}\bigg|_{\substack{p_{11}=\fp_{11}\\p_{1-1}=\fp_{1-1}}}\frac{\partial^2 \fp_{11}}{\partial\omega^2},
\frac{\partial \Fra}{\partial p_{11}}\bigg|_{\substack{p_{11}=\fp_{11}\\p_{1-1}=\fp_{1-1}}}\frac{\partial^2 \fp_{11}}{\partial\omega  \partial s},
\frac{\partial \Fra}{\partial p_{1-1}}\bigg|_{\substack{p_{11}=\fp_{11}\\p_{1-1}=\fp_{1-1}}}\frac{\partial^2 \fp_{1-1}}{\partial\omega^2},
\frac{\partial \Fra}{\partial p_{1-1}}\bigg|_{\substack{p_{11}=\fp_{11}\\p_{1-1}=\fp_{1-1}}}\frac{\partial^2 \fp_{1-1}}{\partial \omega \partial s}
=\tilde O(4^{-k}).
\end{align*}
\end{claim}
\begin{proof}
The claim follows from \eqref{eqCor_2ndmmt_s_10} and \eqref{eqCor_2ndmmt_p_2}.
\end{proof}

\begin{claim}\label{Lemma_2ndmmt_rs}
For $s=\tilde O(4^{-k})$ and $\omega=\frac14+\tilde O(2^{-k/2})$ we have $\frac{\partial^2}{\partial \omega\partial s}F(\omega,s)=\tilde O(1)$.
\end{claim}
\begin{proof}
The second derivative of $g$ with respect to $s$ and $\omega$ is bounded because of Corollary \ref{Cor_2ndmmt_p} and as 
	\begin{align} \label{eqseconddervgws}
	\frac{\partial^2 g}{ \partial s \partial w} = - \frac{1}{p_{11}} \frac{\partial p_{11}}{\partial s} + \frac{2}{p_{1-1}} \frac{\partial p_{1-1}}{\partial s}  + \frac{1}{1-2p_{-1-1}-p_{11}} \frac{ \partial p_{11}}{ \partial s}+  \frac{2}{1-2p_{-1-1}-p_{11}} \frac{ \partial p_{1-1}}{ \partial s}.
	\end{align}	
So 	is the contribution of the first derivative.
Thus, we just need to investigate the second derivative of $h$.
The contribution of $ \frac{\partial^2h}{\partial p_{11}^2},\frac{\partial^2h}{\partial p_{1-1}^2}\frac{\partial^2h}{\partial p_{11}\partial p_{1-1}} $ is bounded by \eqref{eqLemma_2ndmmt_s1} and $\frac{\partial^2h}{\partial \omega\partial s}=0$.
Therefore, the assertion follows from \Faadi's rule.
\end{proof}
\begin{proof}[Proof of \Lem~\ref{Lemma_2ndmmt_2ndderiv}]
Because the Kullback-Leibler divergence is convex and because $\frac{\partial^2F}{\partial s^2}<0$ by \Lem~\ref{Lemma_2ndmmt_s}, 
\Lem~\ref{Claim_2ndmmt_2ndderiv_4} and Claims~\ref{Claim_2ndmmt_2ndderiv_3}--\ref{Lemma_2ndmmt_rs} imply that
\begin{align*}
D^2F(\omega,s)&\preceq\fH=\begin{pmatrix}\fx&\fy\\\fy&\fz\end{pmatrix},
\qquad\mbox{where $\fx=\tilde O(4^{-k}),\fy=\tilde O(1),\fz=\tilde\Theta(4^k)$.}
\end{align*}
The eigenvalues of $\fH$ work out to be $\frac12(\fx+\fz\pm\sqrt{(\fx-\fz)^2+4\fy^2})$.
Therefore, the smaller eigenvalue of $D^2F(\omega,s)$ has size $-\tilde\Omega(4^k)$, while the large one is upper bounded by $\tilde O(4^{-k})$.
Consequently, $D^2F\preceq \tilde O(4^{-k})\id$.
\end{proof}
\begin{proof}[Proof of \Prop~\ref{Prop_smm_1/4}]
\Cor~\ref{Cor_handy} shows that $\vm^{-1}\log\pr\brk{\vE(\omega)\mid\fD}\leq \max_{s\in[0,u]}F(\omega,s)+o(1)$ \whp\
Hence, with $s^*(\omega)=\tilde O(4^{-k})$ the unique maximiser of $s\mapsto F(\omega,s)$, we obtain
\begin{align}\label{eqProp_smm_1/4_1}
\vm^{-1}\log\pr\brk{\cS_{\vm}^\tensor\mid\cR^\tensor_{\vm}(\omega),\fD}&\leq F(\omega,s^*(\omega))+o(1)
	\leq F(1/4,u^2)+(F(\omega,s^*(\omega))-F(1/4,u^2))&&\mbox{\whp{}}
\end{align}
Furthermore, \Lem~\ref{Lemma_2ndmmt_2ndderiv} and Taylor's formula imply that
\begin{align}\label{eqProp_smm_1/4_2}
F(\omega,s^*(\omega))-F(1/4,u^2)&\leq \tilde O(4^{-k})\cdot (\omega-1/4)^2.
\end{align}
Therefore, combining  \eqref{eqsmmexact} and \eqref{eqProp_smm_1/4_1}--\eqref{eqProp_smm_1/4_2} with \Lem~\ref{Lemma_2ndmmt_entropy}, we conclude that \whp
\begin{align}\nonumber
\frac{\Erw[\vE(\omega)\mid\fD]}{\Erw[\vE(1/4)\mid\fD]}&\leq\exp\brk{-(\omega-1/4)^2\bc{\frac{\vm^2}{d^2n}+\vm\tilde O(4^{-k}) }+o(n)}\\&=\exp\brk{-(\omega-1/4)^2\bc{\Omega(k^{-2}n)+\tilde O(2^{-k}) }+o(n)}\leq\exp(o(n)).
\label{eqProp_smm_1/4_3}
\end{align}
Finally, the assertion follows from \Lem~\ref{Lem_1stmmt_entropy}, \Lem~\ref{Lemma_2ndmmt_max} and \eqref{eqProp_smm_1/4_3}.
\end{proof}

\subsubsection{Proof of \Prop~\ref{Prop_smm_not_1/4}}\label{Sec_Prop_smm_not_1/4}
We need to connect the weighted with the unweighted overlap.
To this end we approximate for a given weighted overlap the dominant contributing unweighted overlap.
Let $P(\ell)=\pr\brk{\Po(d/2)=\ell}$.

\begin{lemma}\label{Lemma_degconc}
\Whp\ we have
\begin{align}\label{eqLemma_degconc}
\sum_{\ell,\ell'\geq0}(\ell+\ell'+1)\abs{P(\ell)P(\ell')-\frac1n\sum_{x\in V_n}\vecone\cbc{\vd_x^+=\ell,\vd_x^-=\ell'}}&\leq\sqrt n\log^4 n,
&\max\{\vd_x^+,\vd_x^-:x\in V_n\}\leq\log n.
\end{align}
\end{lemma}
\begin{proof}
This follows from routine concentration bounds.
Indeed, for $\ell,\ell'\leq\log n$ a straightforward application of Chebyshev's inequality shows that
$|\sum_{x\in V_n}\vecone\cbc{\vd_x^+=\ell,\vd_x^-=\ell'}-nP(\ell)P(\ell')|\leq \sqrt n\log n$ \whp{}
Moreover, Bennett's inequality shows that $\vd_x^+,\vd_x^-\leq\log n$ for all $x\in V_n$ \whp{}
\end{proof}

If the condition \eqref{eqLemma_degconc} is satisfied, then we can express the most likely overlap $\alpha$ that gives rise to a given weighted overlap $\omega$ in terms of a neat optimisation problem:
\begin{align*}
\fM(\omega)=&\max\quad-2\sum_{d^+,d^-\geq0}P(d^+)P(d^-)\bc{\alpha_{11}(d^+,d^-)\log \alpha_{11}(d^+,d^-)+(1/2-\alpha_{11}(d^+,d^-))\log(1/2-\alpha_{11}(d^+,d^-))}\\
&\mbox{s.t.\ }\qquad\sum_{d^+,d^-\geq0}P(d^+)P(d^-)(d^++d^-)\alpha_{11}(d^+,d^-)=d\omega,\qquad 
\forall d^+,d^-\geq0:0\leq\alpha_{11}(d^+,d^-)\leq1/2.
\end{align*}
Here the variable $\alpha_{11}(d^+,d^-)$ represents the fraction of variables with degrees $d^+,d^-$ that get both set to `true'.
Let $N(d^+,d^-)$ the number of variables $x$ with $\vd_x^+=d^+,\vd_x^-=d^-$.
Because we only count assignments that satisfy the strongly balanced condition \eqref{eqstronglybal}, there remain
$N(d^+,d^-)(1/2-\alpha_{11}(d^+,d^-))$ variables of degree $(d^+,d^-)$ set to $(1,-1)$, another $N(d^+,d^-)(1/2-\alpha_{11}(d^+,d^-))$ set to $(-1,1)$, while the remaining $N(d^+,d^-)(\alpha_{11}(d^+,d^-))$ ones are set to $(-1,-1)$.
Hence, assuming \eqref{eqLemma_degconc} the total number of such assignments comes to
\begin{align}\nonumber
&\binom{N(d^+,d^-)}{\alpha_{11}N(d^+,d^-),(1/2-\alpha_{11})N(d^+,d^-),
(1/2-\alpha_{-11})(d^+,d^-)N(d^+,d^-),\alpha_{11}(d^+,d^-)N(d^+,d^-)}\\
&\qquad=\exp\bc{-2n P(d^+) P(d^-) \bc{\alpha_{11}\log\alpha_{11}+(1/2-\alpha_{11})\log(1/2-\alpha_{11})}+o(n)}
\label{eqLemma_degconc_1}
\end{align}
In other words, $\fM(\omega)$ asks to choose $\alpha_{11}(d^+,d^-)$ so as to maximise the total number of possible assignments with weighted overlap $\omega$.
The following lemma shows that for $\omega$ far from $1/4$, the optimal solution to $\fM(\omega)$ renders an unweighted overlap 
$2\sum_{d^+,d^-}\alpha_{11}(d^+,d^-)$ far from $1/2$.

\begin{lemma}\label{Lemma_Lag}
	For $|\omega-1/4|>k^{100}2^{-k/2}$ the optimal solution to $\fM(\omega)$ satisfies $\abs{\frac14-\sum_{d^+,d^-} P(d^+) P(d^-)\alpha_{11}(d^+,d^-)}\geq k^{95}2^{-k/2}.$
\end{lemma}
\begin{proof}
We set up the Lagrangian
\begin{align*}
\fL&=-2\sum_{d^+,d^-\geq0}P(d^+)P(d^-)\bc{\alpha_{11}(d^+,d^-)\log \alpha_{11}(d^+,d^-)+(1/2-\alpha_{11}(d^+,d^-))\log(1/2-\alpha_{11}(d^+,d^-))}\\
	&\qquad-\lambda\brk{\sum_{d^+,d^-\geq0}P(d^+)P(d^-)(d^++d^-)\alpha_{11}(d^+,d^-)-d\omega}
\end{align*}
The derivatives
\begin{align*}
\frac{\partial\fL}{\partial \alpha_{11}(d^+,d^-)}&=-2P(d^+)P(d^-)\brk{\log\frac{2\alpha_{11}(d^+,d^-)}{1-2\alpha_{11}(d^+,d^-)}+\lambda (d^++d^-)},\\
\frac{\partial\fL}{\partial \lambda}&= d\omega-\sum_{d^+,d^-\geq0}P(d^+)P(d^-)(d^++d^-)\alpha_{11}(d^+,d^-).
\end{align*}
vanish iff
\begin{align}\label{eqLemma_Lag_0}
\sum_{d^+,d^-\geq0}P(d^+)P(d^-)(d^++d^-)\alpha_{11}(d^+,d^-)&=d\omega,&
\alpha_{11}(d^+,d^-)&=\frac{1-\tanh(\lambda (d^++d^-)/2)}{4}.
\end{align}
Substituting the expression for $\alpha_{11}(d^+,d^-)$ into the left equation, we obtain
\begin{align}
\omega
&=\frac14-\frac{1}{4d\eul^d}\sum_{d^+,d^-\geq0}\frac{d^++d^-}{d^+!d^-!}\bcfr{d}{2}^{d^++d^-}\tanh\frac{\lambda (d^++d^-)}2.
\label{eqLemma_Lag_1}
\end{align}
Since the r.h.s.\ is strictly increasing in $\lambda$, this equation has a unique solution for any $\alpha\in[0,1/2]$.
In fact, for the derivative of the r.h.s.\ we obtain
\begin{align*}
\frac{\partial}{\partial\lambda}\sum_{d^+,d^-\geq0}\bcfr{d}{2}^{d^++d^-}\frac{d^++d^-}{d^+!d^-!}\tanh\frac{\lambda (d^++d^-)}2&=
\sum_{d^+,d^-\geq0}\bcfr{d}{2}^{d^++d^-}\frac{\bc{d^++d^-}^2}{2(d^+!d^-!)}\bc{1-\tanh^2\frac{\lambda (d^++d^-)}2}.
\end{align*}
In particular, for $\omega=0$ the choice $\lambda=0$ solves \eqref{eqLemma_Lag_1}.
Moreover, the derivatives of $\omega$ obtained through \eqref{eqLemma_Lag_1} can be estimated as 
\begin{align*}
\frac{\partial \omega }{\partial \lambda} &= - \frac{1}{8 d} \sum_{d^+, d^- \geq 0} P(d^+) P(d^-)  (d^+ + d^-)^2 \bc{1-\tanh^2 \frac{\lambda (d^++d^-)}{2}} \\
&=  - \frac{d+1}{8} + \frac{1}{32 d} \lambda^2 \Erw[(d^++d^-)^4] + O(d^6 \lambda^3) = - \frac{d+1}{8} + \frac{\lambda^2}{32} (d^3+6d^2+ 7d+1) + O(d^5 \lambda^3),\\
\frac{\partial^2 \omega }{\partial \lambda^2} &= \frac{\lambda}{16} (d^3+6d^2+ 7d+1) + O(d^5 \lambda^2) = \Theta(\lambda d^3\textbf{}).
\end{align*}
Because $\omega = \tilde{O}( 2^{-k/2})$  the inverse function theorem yields $\frac{\partial \lambda }{\partial w}=\frac8{d+1}= \Theta(\frac1d)$.
Thus, $ \frac{\partial^2 \omega }{\partial \lambda^2} = \Theta(d2^{-k})$. Substituting the approximation for $\lambda$ that we get from $\frac{\partial \lambda }{\partial w}=\frac8{d+1}$ back into \eqref{eqLemma_Lag_0} and summing on $d^+,d^-$ completes the proof.
\end{proof}

\begin{proof}[Proof of \Prop~\ref{Prop_smm_not_1/4}]
This is now an immediate consequence of \Lem~\ref{Lemma_degconc}, \Lem~\ref{Lemma_Lag} and the elementary bound \eqref{eqsmm_deriv3}.
\end{proof}

\subsubsection{Proof of \Prop~\ref{Prop_f}}\label{Sec_Prop_f}
Toward the proof of the proposition we first derive an explicit approximation of the term on the r.h.s.
We begin by estimating $p$ from \eqref{maximal_p}.

\begin{lemma}\label{Lem_Jeanp}
We have $p= \frac{1}{2} -\bc{1-\eul^{-\beta}} 2^{-k-1}+  \bc{ 1- \eul^{-\beta}}^2 k 2^{-2k-2}+ O(k^2 2^{-3k})$.
\end{lemma}
\begin{proof}
	The choice \eqref{maximal_p} ensures that
	\begin{align*} 
		\eul^{\beta} &= \frac{(1-p)^k}{2p-1+(1-p)^k} 
		= 1- \frac{2p-1}{2p-1+(1-p)^k} 
		= 1- \frac{2 \bc{p- \frac{1}{2}}}{2\bc{p- \frac{1}{2}}+\bc{\frac{1}{2}- \bc{p- \frac{1}{2}}}^k}.
	\end{align*}
	Hence, $p \in \bc{\frac{1}{2}-2^{-k}\bc{1-\eul^{-\beta}},\frac{1}{2}}$ and thus $q=p-1/2 \in \bc{-2^{-k}\bc{1-\eul^{-\beta}},0}$ is the solution to
	$\bc{2 q+\bc{\frac{1}{2}- q}^k}( \eul^{\beta} - 1)+ {2 q}=0$.
	Using the binomial expansion $ \bc{\frac{1}{2}- q}^k = 2^{-k} - k   2^{-k-1} q + O( k^2 2^{-k} q^2) $ we obtain
	\begin{align*}
		\eul^{\beta} - 1 &=  - \frac{2 q}{2 q + 2^{-k} - k   2^{-k-1} q + O( k^2 2^{-k} q^2)} 	
	\end{align*}
	Hence,
	\begin{align*}
		q &= \frac{\bc{\eul^{\beta} - 1} 2^{-k} }{k 2^{-k -1} \bc{\eul^{\beta} - 1}- 2 \eul^{\beta} + O(k^2 2^{-2k})} = -\frac{\bc{\eul^{\beta} - 1} 2^{-k}}{2 \eul^\beta} + k 2^{-2k-2} \bc{ \frac{  \eul^\beta -1  }{\eul^\beta}}^2  + O( k^2 2^{-3k}),
	\end{align*}
	which implies the assertion.
\end{proof}

\begin{corollary}\label{Lemma_Zbalexpansion}
We have	
	\begin{align*}
		\bc{1- \frac{(k-1)d}{k}}&\log 2 -\frac{d}{2}\log(p(1-p))+\frac{d}{k} \log p
		\\&=\log 2 -  \frac{d}{k}  \bc{1-\eul^{-\beta}}2^{-k} + \frac{d(2k-1)}{2k}2^{-2k}  \bc{   1-\eul^{-\beta}}^2   + O(d k^2 2^{-3k}).
	\end{align*}
\end{corollary}
\begin{proof}
	Using the approximation for $p$ from \Lem~\ref{Lem_Jeanp}, we obtain
	\begin{align*}
		-\frac d2\log(p(1-p))=d\log 2 &- \frac{d}{2} \log \bc{1  -  \bc{1-\eul^{-\beta}} 2^{-k} +  \bc{1-\eul^{-\beta}}^2 k 2^{-2k-1}  + O(k^2 2^{-3k}) }  \\
									  &- \frac{d}{2} \log \bc{1  +  \bc{1-\eul^{-\beta}} 2^{-k} - \bc{1-\eul^{-\beta}}^2 k 2^{-2k-1}  + O(k^2 2^{-3k}) } 
	\end{align*}
	Moreover,
	\begin{align*}
		\log&\bc{1\mp \bc{1-\eul^{-\beta}} 2^{-k} \pm \bc{1-\eul^{-\beta}}^2 k 2^{-2k-1} + O(k^2 2^{-3k}) } \\
			&= \mp  \bc{1-\eul^{-\beta}} 2^{-k} \pm \bc{1-\eul^{-\beta}}^2 k 2^{-2k-1} -  2^{-2k-1} \bc{ 1- \eul^{-\beta}}^2  + O(k^2 2^{-3k}).
	\end{align*}
	Hence,
	\begin{align}\label{eqJeanApx1}
		-\frac d2\log(p(1-p))&= d \log 2 + d 2^{-2k-1} \bc{ 1- \eul^{-\beta}}^2 + O( d \ k^2 2^{-3k}).
	\end{align}
	Further, using \eqref{maximal_p}, we obtain $2p=1-(1-\eul^{-\beta})(1-p)^k$ and thus
	\begin{align}\label{eqJeanApx2}
		\log(2p) &=  - \bc{1-\eul^{-\beta}} 2^{-k} + (k-1) 2^{-2k-1} \bc{  1- \eul^{-\beta}}^2 + O(k^2 2^{-3k}).
	\end{align}
	Combining \eqref{eqJeanApx1} and \eqref{eqJeanApx2} completes the proof.
\end{proof}

Having estimated the expression on the r.h.s.\ of \Prop~\ref{Prop_f}, we proceed to investigate the function $f(\alpha)$.
Its derivatives read
\begin{align}
f'(\alpha)&=\log\frac{1-\alpha}{\alpha}+\frac{d\alpha^{k-1}(1-\eul^{-\beta})^2}{2^k(1-2^{1-k}(1-\eul^{-\beta})+\alpha^k2^{-k}(1-\eul^{-\beta})^2)},\label{eqfd1}\\
f''(\alpha)&=-\frac1{\alpha}-\frac{1}{1-\alpha}+\frac{(k-1)d\alpha^{k-2}(1-\eul^{-\beta})^2}{2^k(1-2^{1-k}(1-\eul^{-\beta})+2^{-k}\alpha^k(1-\eul^{-\beta})^2)}-\frac{kd\alpha^{2(k-1)}(1-\eul^{-\beta})^4}{2^{2k}(1-2^{1-k}(1-\eul^{-\beta})+2^{-k}\alpha^k(1-\eul^{-\beta})^2)^2}.
\label{eqfd2}
\end{align}

\begin{claim}\label{Claim_f1}
	We have $f(\alpha)\leq f(1-\alpha)$ for all $\alpha<1/2$.
	Moreover, $f$ is concave on the interval $[1/2,1/2+o(k^{-3})]$, where it attains a local maximum at $\alpha^*=\frac{1}{2}+O(d4^{-k})$ with
	\begin{align}\label{eqProp_f_X1}
		f(\alpha^*)&=2\log2-\frac{2d(1-\eul^{-\beta})}{2^k}\bc{1+2^{-k-1}(1-\eul^{-\beta})}+O(k^84^{-k}).  
	\end{align}
\end{claim}
\begin{proof}
	The first assertion follows immediately from the symmetry of the entropy function and the fact that $\alpha\mapsto 1-2^{1-k}(1-\exp(-\beta))+2^{-k}\alpha^k(1-\exp(-\beta))^2$ is increasing.
	We also read off of \eqref{eqfd1} that $f'(1/2)=O(d4^{-k})$, while \eqref{eqfd2} shows that $f''(\alpha)=-4+o(1)$ if $\alpha=1/2+o(k^{-3})$.
	Hence, $f$ attains a local maximum at $\alpha^*=1/2+O(d4^{-k})$.
	Finally, Taylor's formula shows that $f(\alpha^*)=f(1/2)+O(k^4d^216^{-k})$, whence we obtain \eqref{eqProp_f_X1}.
\end{proof}

\begin{claim}\label{Claim_f2}
	The function $f(\alpha)$ is monotonically decreasing on $(1/2+k^{-4},1-2\log(k)/k)$.	
\end{claim}
\begin{proof}
	For $\alpha\in[1/2+k^{-4},0.99]$ we see that $\log((1-\alpha)/\alpha)=-\Omega(k^{8})$ while
	\begin{align}\label{eqClaim_f2}
	\frac{d\alpha^{k-1}(1-\eul^{-\beta})^2}{2^k(1-2^{1-k}(1-\eul^{-\beta})+\alpha^k2^{-k}(1-\eul^{-\beta})^2)}=\exp(-\Omega(k)).
	\end{align}
	Hence, \eqref{eqfd1} shows that $f$ is decreasing on this interval.
	Similarly, for $\alpha\in(0.99,1-2\log(k)/k)$ we obtain $\log((1-\alpha)/\alpha)=-\Omega(1)$, while the l.h.s.\ of \eqref{eqClaim_f2} is of order $o(1)$.
	Hence, $f'(\alpha)<0$.
\end{proof}

We are going to prove that for $d\leq d^*$ the maximum of $f(\alpha)$ is approximately equal to $f(1/2)$ by comparing the function value $f(1/2)$ with the function values $f(\alpha)$ for $\alpha>1/2$.
Since the function $\alpha\mapsto 1-2^{1-k}(1-\exp(-\beta))+2^{-k}\alpha^k(1-\exp(-\beta))^2$ is monotonically increasing, we may assume that $d=d^*$.
Actually, in order to facilitate the proof of \Thm~\ref{Thm_rsb}, in some of the estimates below we will allow for $d$ to take values up to $\dk(k)$.

\begin{claim}\label{Claim_f3}
	Assume that $d^*\leq d\leq\dk(k)  \  $.
Then the function $f(\alpha)$ has only one stationary point in the Interval $[1-2\log(k)/k,1-k^{-3/2}]$, which is a local minimum.
\end{claim}
\begin{proof}
Substituting $\alpha=1-\eps$ into \eqref{eqfd1} we can write
\begin{align}\label{eqClaim_f3}
	f'(\alpha)&=\log\eps-\log(1-\eps)+\frac{d(1-\eul^{-\beta})^2\exp(\eps-k\eps+O(k\eps^2))}{2^k(1+O(2^{-k}))}.
\end{align}
Therefore, for $d^*\leq d\leq\dk(k)$ the only solution to $f'(\alpha)=0$ is such that $\eps=(\log(k)+O(\log\log k))/k$.
Furthermore, for the root of $f'(\alpha)$ we read off \eqref{eqfd2} that $f''(\alpha)=\Omega(k\log k)>0$.
\end{proof}

\begin{claim}\label{Claim_f4}
	Assume that $d^*\leq d\leq\dk(k)$.
Then the function $f(\alpha)$ has only one stationary point in the interval $[1-k^{-3/2},1]$, namely a local maximum at $\alpha_*=1-2^{-k(1-\eul^{-\beta})^2+o(1)}$.
\end{claim}
\begin{proof}
	Contemplating \eqref{eqClaim_f3}, we see that for $\alpha=1-\eps$ with $\eps<k^{-3/2}$ the last summand has the form $(1+o(1))(1-\eul^{-\beta})^2k\log 2$.
	Hence, the only root of $f'(\alpha)$ in this interval occurs at $\eps=\eps_*=2^{-k(1-\eul^{-\beta})^2+o(1)}$.
	A glimpse at \eqref{eqfd2} reveals that $f''(1- \eps_*)<0$.
	Hence, $\alpha_*=1-\eps_*$ is a local maximum.
\end{proof}

\begin{claim}\label{Claim_f5}
	Assume that $d\leq d^*$.
	Then the global maximum of $f(\alpha)$ is attained at $\alpha^*$ from Claim~\ref{Claim_f1}.
\end{claim}
\begin{proof}
	In light of Claims~\ref{Claim_f2}--\ref{Claim_f4} we just need to compare $f(\alpha_*)$ and $f(\alpha^*)$ for $d=d^*$.
	Let us estimate them:
	\begin{align}
		f(\alpha_*)&\leq\log2+\eps_*(1-\ln\eps_*)+\frac{d^*}{k}\ln(1-2^{1-k}(1-\eul^{ -\beta })+2^{-k}(1-\eps_*)^k(1-\eul^{ -\beta })^2)\nonumber\\
				   &\leq\eul^{-2\beta}\log2+10k2^{-k}(1-\eul^{-2\beta})+o(k2^{-k}),\label{eqClaim_f5_1}\\
		f(\alpha^*)&=2\log2+\frac{2d}{k}\log\bc{1-(1-\eul^{-\beta})2^{-k}}=2\eul^{-\beta}\log2+20k2^{-k}(1-\eul^{-\beta})+o(k2^{-k}).\label{eqClaim_f5_2}
	\end{align}
	Combining \eqref{eqClaim_f5_1} and \eqref{eqClaim_f5_2}, we see that $f(\alpha^*)>f(\alpha_*)$ for $\beta\geq1$.
\end{proof}

\begin{proof}[Proof of \Prop~\ref{Prop_f}]
	Combining Claims~\ref{Claim_f2}--\ref{Claim_f5}, we are left to merely compare $f(\alpha^*)$ and the r.h.s.\ expression from \Prop~\ref{Prop_f}.
	Comparing the approximation of the latter supplied by \Cor~\ref{Lemma_Zbalexpansion} with \eqref{eqProp_f_X1} completes the proof.
\end{proof}

\subsection{Proof of \Cor~\ref{Cor_SecondMoment}}\label{Sec_Cor_bal}
As an immediate consequence of \Prop~\ref{Prop_bal}, 
\eqref{eqThm:SecondMoment1} and Markov's inequality we obtain that \whp\ the random formula $\PHI$ satisfies
\begin{align*}
\sum_{\sigma,\tau\in\cbc{\pm1}^n}
\vecone\cbc{|\sigma\cdot\tau|>k^{100}2^{-k/2}n}
\prod_{i=1}^{\vm}\exp\bc{-\beta\bc{\vecone\cbc{\sigma\not\models a_i}+\vecone\cbc{\tau\not\models a_i}}}=o(Z(\PHI,\beta)^2).
\end{align*}
Dividing by $Z(\PHI,\beta)^2$, we obtain
\begin{align}\label{eqCor_bal1}
\mu_{\PHI,\beta}\bc{\cbc{|\SIGMA\cdot\SIGMA'|>nk^{100}2^{-k/2}}}&=o(1),
\end{align}
whence \Cor~\ref{Cor_SecondMoment} is immediate.

To prove \Cor~\ref{Cor_bal} we combine \eqref{eqCor_bal1} with the following general lemma.

\begin{claim}\label{Lemma_general}
	Suppose that $\nu\in\cP(\{\pm1\}^n)$ satisfies $\nu\bc{\cbc{|\SIGMA\cdot\SIGMA'|\geq\eps n}}<1/2$ and
	\begin{align}\label{eqLemma_general1}
		\sum_{i,j=1}^n|\nu(\{\SIGMA_i=\SIGMA_j=1\})-\nu(\{\SIGMA_i=1\})\nu(\{\SIGMA_j=1\})|=o(n^2).
	\end{align}
	Then $\sum_{i=1}^n\bc{\nu(\{\SIGMA_i=1\})-1/2}^2<\eps n.$
\end{claim}
\begin{proof}
	The assumption \eqref{eqLemma_general1} implies together with \cite[\Lem~2.11]{Will2} that the product measure $\nu\tensor\nu$ has the same property, i.e., that for two independent samples $\SIGMA,\SIGMA'$ from $\nu$ and for any $s,s',t,t'\in\{\pm1\}$ we have
\begin{align}\label{eqLemma_general2}
		\sum_{i,j=1}^n|\nu\tensor\nu(\{\SIGMA_i=s,\SIGMA_j=s',\SIGMA_i'=t,\SIGMA_j'=t'\})-\nu\tensor\nu(\{\SIGMA_i=s,\SIGMA_i'=t\})\cdot\nu\tensor\nu(\{\SIGMA_j=s',\SIGMA_j'=t'\})|=o(n^2).
	\end{align}
Now, for $i\in[n]$ let $p_i=\nu(\{\SIGMA_i=1\})$.
Then
\begin{align*}
\scal{\SIGMA_i\cdot\SIGMA'_i}{\nu}&=p_i^2+(1-p_i)^2-2p_i(1-p_i)=(1-2p_i)^2=4\bc{p_i-\frac{1}{2}}^2.
\end{align*}
Hence,  $\scal{\sum_{i=1}^n\SIGMA_i\cdot\SIGMA'_i}{\nu}=4\sum_{i=1}^n\bc{p_i-\frac{1}{2}}^2.$
Therefore, \eqref{eqLemma_general2} and Chebychev's inequality show
\begin{align*}
\nu\bc{\cbc{\sum_{i=1}^n\SIGMA_i\cdot\SIGMA'_i\geq\sum_{i=1}^n\bc{p_i-\frac{1}{2}}^2}}&\geq1/2.
\end{align*}
Consequently, the assumption $\nu\bc{\cbc{|\SIGMA\cdot\SIGMA'|\geq\eps n}}<1/2$ implies that $\sum_{i=1}^n\bc{p_i-\frac{1}{2}}^2\leq\eps n$, as desired.
\end{proof}

\begin{proof}[Proof of \Cor~\ref{Cor_bal}]
	The corollary is an immediate consequence of \eqref{eqCor_bal1} and Claim~\ref{Lemma_general}.
\end{proof}

\section{Proof of \Prop~\ref{Prop_Noela}}\label{Sec_Prop_Noela}
\noindent
Throughout this section we assume that $d\leq\dk(k)$.

\subsection{Outline}
The proof of \Prop~\ref{Prop_Noela} relies on a contraction argument.
Specifically, we will be able to describe the distribution of the BP marginal $\mu_{\vT,\beta,\pi,x_0,t}(1)$ at the root of the random tree $\vT$, which we aim to compute, in terms of an operator $\cR$ on the space $\cP([0,1])$ of probability measures on the unit interval.
To define this operator let $\vgamma^{+}, \vgamma^{-}$ be $\Po\left(d/2\right)$ variables and given $\mu \in \cP([0,1])$ let $\veta=(\veta_{ij}^{+}, \veta_{ij}^{-})_{i,j \geq 1}$ be random variables with distribution $\mu$.
All these random variables are mutually independent.
Then $\cR(\mu) \in \cP([0,1])$ is the law of the random variable
\begin{align} \label{rec}
R(\vgamma^{+}, \vgamma^{-}, \veta)=\frac{\prod_{i=1}^{\vgamma^+}\left(1 - \left(1- \eul^{-\beta}\right) \prod_{j=1}^{k-1}\veta_{ij}^+\right)}{\prod_{i=1}^{\vgamma^+}\left(1 - \left(1-\eul^{-\beta}\right) \prod_{j=1}^{k-1}\veta_{ij}^+\right) + \prod_{i=1}^{\vgamma^{-}}\left(1 - \left(1-\eul^{-\beta}\right) \prod_{j=1}^{k-1}\veta_{ij}^{-}\right)}\in(0,1).
\end{align}
We write $\cR^t(\nix)$ for the $t$-fold iteration of $\cR$.

We are going to investigate the operator $\cR$ on a subspace of $\cP([0,1])$.
Specifically, let $\cP^\star$ be the space of all probability measures $\mu \in \cP([0,1])$ such that $\mu([0,x]) = \mu([1-x,1])$ for all $x \in [0,1]$.
Because $\vgamma^+$ and $\vgamma^-$ are identically distributed, \eqref{rec} ensures that $\cR$ maps the subspace $\cP^\star$ into itself.
Further, for any probability measure $\pi\in\cP([0,1])$ we obtain a probability measure $\pi^\star\in\cP^\star$ as follows.
Draw $\vX$ from $\pi$ and independently draw a Rademacher variable $\vJ$ with $\ex\brk\vJ=0$.
Then $\pi^\star$ is the distribution of $(1+\vJ(2\vX-1))/2$.
The following observation links $\cR$ to the Belief Propagation message passing scheme on $\vT$.

\begin{lemma}\label{Fact_TR}
For any $\pi\in\cP([0,1])$ and any $t\geq1$ the random variable $\mu_{\vT,\beta,\pi,x_0,t}(1)$ has distribution $\cR^t(\pi^\star)$.
\end{lemma}
\begin{proof}
We first consider the case $t=1$.
The construction of $\vT$ ensures that the root $x_0$ has $\Po(d/2)$ children $a$ with $\vJ_{ax_0}=1$ and an independent number of $\Po(d/2)$ children $a$ with $\vJ_{ax_0}=-1$.
Hence, these numbers can be coupled with $\vgamma^+,\vgamma^-$.
Now consider a child $a$ of $x_0$ with $\vJ_{ax_0}=1$.
Then $a$ has children $y_1,\ldots,y_{k-1}$ and the signs $\vJ_{ay_i}$ are independent Rademacher variables.
Hence, if $\vX_1,\ldots,\vX_{k-1}$ are independent samples from $\pi$, then the message that $a$ passes to $x_0$ reads
\begin{align*}
\mu_{\vT,\beta,\pi,a\to x_0,1}(s)\propto
\vecone\cbc{\vJ_{ax_0}=s}+\vecone\cbc{\vJ_{ax_0}=-s}\bc{
1-(1-\eul^{-\beta})\prod_{i=1}^{k-1}\bc{\vecone\cbc{\vJ_{ay_i}=1}(1-\vX_i)+\vecone\cbc{\vJ_{ay_i}=-1}\vX_i}}\quad(s=\pm1).
\end{align*}
Because the $\vJ_{ay_i}$ and the $\vX_i$ are independent, the distribution of $\mu_{\vT,\beta,\pi,a\to x_0,1}(s)$ can alternatively be described as follows: let $\vX_1^\star,\ldots,\vX_{k-1}^\star$ be independent samples from $\pi^\star$; then $\mu_{\vT,\beta,\pi,a\to x_0,1}(s)$ is distributed as
\begin{align*}
\frac{\vecone\cbc{\vJ_{ax_0}=s}+\vecone\cbc{\vJ_{ax_0}=-s}\bc{ 1-(1-\eul^{-\beta})\prod_{i=1}^{k-1}\vX_i^\star}}{2-(1-\eul^{-\beta})\prod_{i=1}^{k-1}\vX_i^\star}.
\end{align*}
Consequently, $\mu_{\vT,\beta,\pi,x_0,t}(1)$ has distribution $\cR(\pi^\star)=\cR^t(\pi^\star)$ for $t=1$.
Finally, a straightforward induction extends this statement to all $t\geq1$.
\end{proof}

To prove \Prop~\ref{Prop_Noela} we first study the operator $\cR$ on a subspace $\cP^\dagger\subset\cP^\star$ of probability measures that satisfy a certain tail bound.
Specifically, we define $\cP^\dagger$  as the space of all measures $\mu\in\cP^\star$ such that
\begin{align}\label{Tail_bound0}
 \mu\left(\left[\frac{\eul^s}{1+\eul^s},1\right]\right) &\leq\exp\left(-s 2^{k/4} \right)&&\mbox{for all $s \geq 2^{-k/4}$.}
\end{align}
For future reference we observe that the function
\begin{align}\label{eqvarphi}
\varphi&:\RR\to(0,1),\ z\mapsto \frac{\eul^z}{1+\eul^z}
&\mbox{is a bijection with inverse}&&
\varphi^{-1}&:(0,1)\to\RR,\ y\mapsto\ln\frac y{1-y}. 
\end{align}
The operator $\cR$ maps the subspace $\cP^\dagger$ into itself.

\begin{proposition} \label{Tail_bounds}
For every $\mu\in\cP^\dagger$ we have $\cR(\mu)\in \cP^\dagger$.
\end{proposition}

\noindent
The proof of \Prop~\ref{Tail_bounds} can be found in \Sec~\ref{Tail_bounds}.
Next we show that $\cR$ is a contraction on $\cP^\dagger$.
Specifically, in \Sec~\ref{Sec_convergence} we will prove the following.

\begin{proposition} \label{convergence}
	For $d \leq \dk(k)$ the operator $\cR$ is a contraction on $\cP^\dagger$ with respect to the $W_r$-metric, where $r$ is the smallest even integer greater than $2^{k/10}$.
\end{proposition}

\noindent
Finally, in \Sec~\ref{Sec_Truncation_allowed} we will prove that a sufficient number of iterations of $\cR$ will map any distribution $\pi$ with slim tails to a distribution that `almost ' belongs to the subspace $\cP^\dagger$.
To formalise this, for a random variable $\veta\in[0,1]$ and a number $\eps>0$ let $\tilde\veta_\eps$ be a random variable whose distribution is characterised by
\begin{align}\label{eqFunnyConstruction}
\pr\brk{\tilde\veta_\eps\in A}=\pr\brk{\veta\in A\cap\brk{\eps,1-\eps}}
+\vecone\cbc{\frac{1}{2}\in A}\pr\brk{\veta\not\in\brk{\eps,1-\eps}}&&
\mbox{for any measurable $A\subset[0,1]$}.
\end{align}
In words, $\tilde\veta_\eps$ is obtained by truncating $\veta$ at $\eps$ and $1-\eps$ and shifting the lost probability mass to $\frac12$.

\begin{proposition}\label{Truncation_allowed}
For any $\eps >0$ there exists $\ell_0=\ell_0(\eps)>0$ such that for all $\ell>\ell_0$ and all probability measures $\pi$ with slim tails the following two statements hold.
\begin{enumerate}[(i)]
\item Let $\vxi$ be a random variable with distribution $\cR^\ell(\pi)$. Then $\pr\brk{\vxi\not\in\brk{\eps,1-\eps}} \leq \eps.$
\item The distribution of $\tilde\vxi_\eps$ belongs to $\cP^\dagger$.
\end{enumerate}
\end{proposition}

\noindent
We prove \Prop~\ref{Truncation_allowed} in \Sec~\ref{Sec_Truncation_allowed}.
These statements now easily imply \Prop~\ref{Prop_Noela}.

\begin{proof}[Proof of \Prop~\ref{Prop_Noela}]
	Let $\eps>0$, pick large enough $\ell=\ell(\eps)\ll L=L(\eps,\ell)$ and let $\pi,\pi'$ be two distributions with slim tails.
Consider the variables $U$ at distance precisely $2\ell$ from the root of $\vT$.
For each $y\in U$ let $b_y$ be the parent clause.
Suppose we initialise the variable--to--parent messages for the variables at distance $2(\ell+L)$ from the root with independent messages drawn from $\pi$.
Let $\vxi_y$ be the ensuing message that $y$ will send to its parent $b_y$ and let $\tilde\vxi_y$ be the corresponding truncated message as per \Prop~\ref{Truncation_allowed}.
Define $\vxi_y',\tilde\vxi_y'$ analogously for the initial distribution $\pi'$.
Then \Prop~\ref{Truncation_allowed} shows that the events $\cU=\{\forall y\in U:\vxi_y=\tilde\vxi_y\}$ and $\cU'=\{\forall y\in U:\vxi_y'=\tilde\vxi_y'\}$  occur with probability $1-\eps$, provided that $L\gg\ell$.
Furthermore, \Prop~\ref{convergence} shows that given $\cU\cap\cU'$ the subsequent $\ell$ iterations of Belief Propagation up to the root are contracting.
Hence, we obtain a coupling of the distributions $\pi,\pi'$ such that after $\ell+L$ iterations of Belief Propagation the $L^1$-distance of the messages is bounded by $(1-\delta)^\ell$ for some fixed $\delta=\delta(k,d,\beta)>0$.
Therefore, the assertion follows from the completeness of the space $\cP^\dagger$.
\end{proof}

The proofs of \Prop s~\ref{Tail_bounds} and \ref{convergence} are adaptations of the proofs of \cite[Lemmas 4.2 and~4.4]{DSS3} where a related but slightly more intricate distributional recursion is analysed.
Moreover, the proof of \Prop~\ref{Truncation_allowed} combines ideas and observations from the proof of \Prop~\ref{Tail_bounds} such as the random walk analysis on typical events and the stochastic domination argument with a novel bootstrapping idea to track the tightening and proliferation of certain tail bounds as $\cR$ is iteratively applied.
This allows to extend the conclusion of  \Prop~\ref{convergence} to the relevant initial distributions.
Let us delve into the details.

\subsection{Proof of \Prop~\ref{Tail_bounds}}

\noindent
Throughout this section we assume that $\mu\in\cP^\dagger$. The proof of \Prop~\ref{Tail_bounds} follows along the steps of \cite[Lemma 4.2]{DSS3}, where Ding, Sly and Sun analyse a more complex distributional recursion related to the Survey Propagation algorithm.
More precisely, they show that the application of their operator preserves a tail bound similar to \eqref{Tail_bound0}.
Here we follow the steps of their proof closely to derive the corresponding statement for the conceptually simpler Belief Propagation operator.
That said, here and there some extra care is required because we work with an arbitrary temperature parameter $0<\beta<\infty$ while in~\cite{DSS3} it suffices to study zero temperature.

\subsubsection{Overview}
We aim to show that $\hat\mu=\cR(\mu)$ belongs to $\cP^\dagger$ as well.
Letting 
\begin{align*}
\vPi^+&=\prod_{i=1}^{\vgamma^+}\left(1 - \left(1- \eul^{-\beta}\right) \prod_{j=1}^{k-1}\veta_{ij}^+\right) ,&
\vPi^-&=\prod_{i=1}^{\vgamma^-}\left(1 - \left(1- \eul^{-\beta}\right) \prod_{j=1}^{k-1}\veta_{ij}^-\right),
\end{align*}
we see from the definition of $\cR$ that
\begin{align} \label{eqNoelaRandomWalk}
\hat\mu\left(\left[\frac{\eul^s}{1+\eul^s},1\right]\right) = \pr\bc{\ln\vPi^+-\ln\vPi^- \geq s}.
\end{align}
Both $-\ln\vPi^+$ and $-\ln\vPi^-$ are sums of a random number of non-negative i.i.d.\ random variables, and  $-\ln\vPi^+$ and $-\ln\vPi^-$ are identically distributed.
Hence, estimating the probability \eqref{eqNoelaRandomWalk} is a bit like estimating the probability that a weighted symmetric random walk strays far from the origin. 
To bound this probability we first bound the large deviations of the individual summands.
More specifically, for $i \geq 1$, set
\begin{align*}
	\vX_i^{\pm}&=-\ln\left(1 - \left(1-\eul^{-\beta}\right) \prod_{j=1}^{k-1}\veta_{ij}^\pm \right)\geq0,&\mbox{ so that }&&
-\ln\vPi^+&= \sum_{i=1}^{\vgamma^+}\vX_i^+,&-\ln \vPi^- &= \sum_{i=1}^{\vgamma^-}\vX_i^-.
\end{align*}
Let $\vX\disteq\vX_1^+$ denote a generic summand.
In \Sec~\ref{Sec_tails_of_X} we are going to prove the following.

\begin{lemma}\label{Lemma_tails_of_X}
\begin{enumerate}[(i)]
\item  For all $t \geq 1$ we have $\Pr\bc{\vX \geq t} \leq  \exp\bc{-\frac{t}{2}(k-1) 2^{k/4}}.$
\item  For all $\eps \in [2^{-k/9},1]$ we have $\Pr\bc{\vX \geq (1-\eul^{-\beta})2^{-(1-\eps)(k-1)}}\leq 2^{k-1}\exp\bc{-(k-1)\bc{\frac{\eps}{3}}^2 2^{k/4}}.$
\item   For all $\eps \in [2^{-k/9},1]$ we have $\Pr\brk{\vX \leq (1-\eul^{-\beta})2^{-(1+\eps)(k-1)}}\leq  k\exp\bc{-2^{k/4} \eps}.$
\end{enumerate}
\end{lemma}

\Lem~\ref{Lemma_tails_of_X} implies the following estimate of $\Erw[\vX]$.

\begin{corollary}\label{bound_expectation}
We have $\Erw[\vX] =\bc{1-\eul^{-\beta}}2^{-(k-1)} + O\bc{k2^{-10k/9}}.$
\end{corollary}
\begin{proof}
Applying Lemma \ref{Lemma_tails_of_X} with $\eps = 2^{-k/9}$, we obtain
\begin{align*}
\Erw[\vX] =& \ \Erw\left[\vX\vecone\left\{\vX \in \left[0,\bc{1-\eul^{-\beta}}2^{-(1+\eps)(k-1)} \right] \right\}\right] + \Erw\left[\vX\vecone\left\{\vX \in  \left(\bc{1-\eul^{-\beta}}2^{-(1+\eps)(k-1)},   \bc{1-\eul^{-\beta}}2^{-(1-\eps)(k-1)}\right)\right\}\right] \nonumber \\
&+ \Erw\left[\vX\vecone\left\{\vX \in \left[ \bc{1-\eul^{-\beta}}2^{-(1-\eps)(k-1)},1 \right)\right\}\right] + \Erw[\vX\vecone\{\vX \geq 1\}]
=  \ \bc{1-\eul^{-\beta}}2^{-(k-1)} + O\bc{k2^{-(k-1)-k/9}},
\end{align*}
as claimed.
\end{proof}

In order to estimate the difference between $\ln\vPi^+$ and $\ln\vPi^-$ we are going to replace $\vX$ by the truncated random variable
\begin{align} \label{eq_truncate_X}
\bar{\vX} = \vX \vecone\left\{ \vX \leq \bc{1-\eul^{-\beta}}2^{-(9/10)(k-1)} \right\}.
\end{align}
\Lem~\ref{Lemma_tails_of_X} implies the following bound on the difference between $\vX$ and $\bar{\vX}$.

\begin{corollary}\label{eq_diff_exp}
We have $\Erw[\vX-\bar{\vX}]\leq \exp\bc{-\Omega\bc{k 2^{k/4}}}$.
\end{corollary}
\begin{proof}
Applying Lemma \ref{Lemma_tails_of_X} (i) and (ii) with $\eps = 1/10$, we obtain
\begin{align*}
\Erw[\vX-\bar{\vX}] &\leq \Pr\bc{\vX \geq \bc{1-\eul^{-\beta}}2^{-(9/10)(k-1)}} + \Erw[\vX\vecone\{\vX \geq 1\}] \leq 2^k\exp\bc{-\frac{k-1}{900} 2^{k/4}} +2\exp\bc{-(k-1) 2^{k/4-1}},
\end{align*}
as claimed.
\end{proof}

With these preparations in place we prove the desired tail bound in two steps.
First we consider the case $s \in [2^{-k/4}, 1]$.
To be precise, in \Sec~\ref{Sec_Lemma_sleq1} we are going to prove the following.

\begin{lemma}\label{Lemma_sleq1}
For $s \in [2^{-k/4}, 1]$ we have $\hat\mu\left(\left[\frac{\eul^s}{1+\eul^s},1\right]\right) \leq\exp\left(-s 2^{k/4} \right).$
\end{lemma}

\noindent 
Finally, in Section~\ref{Sec_Lemma_sgeq1} we are going to deal with $s>1$.

\begin{lemma}\label{Lemma_sgeq1}
For $s>1$ we have $\hat\mu\left(\left[\frac{\eul^s}{1+\eul^s},1\right]\right) \leq\exp\left(-s 2^{k/4} \right).$
\end{lemma}

\begin{proof}[Proof of \Prop~\ref{Tail_bounds}]
The assertion follows immediately from \Lem s~\ref{Lemma_sleq1} and~\ref{Lemma_sgeq1}.
\end{proof}

\subsubsection{Proof of \Lem~\ref{Lemma_tails_of_X}} \label{Sec_tails_of_X}
We prove the three statements separately.
With respect to (i) we notice that $\vX \leq \beta$ deterministically.
Hence, the assertion trivially holds for $t \geq \beta$.
Now consider $t \in [1,\beta)$.
Because the random variables $\veta_{1j}^+$ are bounded by one and independent, we obtain
\begin{align}
\Pr\brk{\vX \geq t} &= \Pr\brk{\prod_{j=1}^{k-1}\veta_{1j}^+ \geq \frac{1-\eul^{-t}}{1-\eul^{-\beta}}} \leq \Pr\bc{\forall j \in [k-1]: \veta_{1j}^+ \geq \frac{1-\eul^{-t}}{1-\eul^{-\beta}}}=\mu\bc{\left[\frac{1-\eul^{-t}}{1-\eul^{-\beta}},1\right]}^{k-1}.
\label{Upper_tail_0a}
\end{align}
Further, since $\frac{1-\eul^{-t}}{1-\eul^{-\beta}}<1$ for $t<\beta$, $\varphi^{-1}\bc{\frac{1-\eul^{-t}}{1-\eul^{-\beta}}}\geq t/2$ and because $\mu\in\cP^\dagger$, \eqref{Upper_tail_0a} yields $ \Pr\brk{\vX \geq t} \leq\exp\bc{-\frac{t}{2}(k-1) 2^{k/4}}, $
as claimed.

We proceed to (ii).
Since $1-\mathrm{exp}\bc{-yz} \geq y(1 - \mathrm{exp}\bc{-z})$ for any $y \in (0,1]$, $z \geq 0$,
we obtain
\begin{align} 
\Pr\bc{\vX \geq  (1-\eul^{-\beta})2^{(\eps-1)(k-1)}} & = \Pr\bc{\prod_{j=1}^{k-1}\veta_{1j}^+ \geq \frac{1-\mathrm{exp}\bc{-\bc{1-\eul^{-\beta}} 2^{(\eps-1)(k-1)}}}{1-\eul^{-\beta}}}\leq\Pr\bc{\prod_{j=1}^{k-1}\veta_{1j}^+ \geq 1-\mathrm{exp}\bc{-2^{(\eps-1)(k-1)}}}
\label{Bound_a_b}.
\end{align}
Further, since $\veta^+_{1j}\leq1$ for all $j$, for any $a \in [0,1]$ and $b \in (0,k-2]$ we have
\begin{align} \label{prod_eta}
\Pr\bc{\prod_{j=1}^{k-1}\veta_{1j}^+ \geq a} \leq \Pr\bc{\left| \left\{j \in [k-1]: \veta_{1j}^+ \geq a^{\frac{1}{k-1-b}}\right\}\right| \geq b}.
\end{align}
Moreover, for large $k$ we have
\begin{align}\label{prod_etaa}
  1-\mathrm{exp}\bc{-2^{(\eps-1)(k-1)}} \geq 2^{-(k-1)(1-\eps/3)^2}.
\end{align}

Combining \eqref{Bound_a_b}, \eqref{prod_eta} and \eqref{prod_etaa}, we obtain
\begin{align}\nonumber
\Pr\bc{\vX \geq  \bc{1-\eul^{-\beta}}2^{(\eps-1)(k-1)}} &\leq \Pr\bc{\prod_{j=1}^{k-1}\veta_{1j}^+ \geq 2^{-(k-1)(1-\eps/3)^2}}{\leq} \Pr\bc{\left| \left\{j \in [k-1]: \veta_{1j}^+ \geq 2^{\eps/3-1}\right\}\right| \geq (k-1)\eps/3} \\
& \leq \sum_{(k-1)\eps/3\leq j\leq k-1}\binom{k-1}{j} \mu\bc{[2^{\eps/3-1},1]}^j\leq 2^{k-1}  \mu\bc{[2^{\eps/3-1},1]}^{(k-1)\eps/3}\label{prod_etab}.
\end{align}
Since $\varphi^{-1}(2^{\eps/3-1})\geq\frac{2\eps}{3} \ln 2\geq 2^{-k/4}$ for all $\eps \in [2^{-k/9},1]$, the assumption $\mu\in\cP^\dagger$ and \eqref{prod_etab} imply
\begin{align}\label{prod_etaf}
\Pr\bc{\vX \geq  \bc{1-\eul^{-\beta}}2^{(\eps-1)(k-1)}} &
\leq 2^{k-1} \mathrm{exp}\bc{-(k-1)\frac{\eps}{3}2^{k/4}\varphi^{-1}\bc{2^{\eps/3-1}}}.
\end{align}
Finally, since $\varphi^{-1}\bc{2^{\eps/3-1}} \geq \frac{2\eps}{3} \ln 2$ for $\eps \in [0,1]$,
the second assertion follows from \eqref{prod_etaf}.

Coming to the statement (iii), we use the inequality $1-\eul^{-x} \leq x$ to obtain
\begin{align}\nonumber
\Pr\bc{\vX \leq \bc{1-\eul^{-\beta}}2^{-(\eps+1)(k-1)}} &\leq \Pr\bc{1-\eul^{-\vX} \leq \bc{1-\eul^{-\beta}}2^{-(\eps+1)(k-1)}}\\&  \leq \Pr\bc{\exists \ j \in [k-1]: \veta_{1j}^+ \leq 2^{-(\eps+1)}} 
\leq (k-1) \mu\bc{[0,2^{-(\eps+1)}]}.\label{eqlem_lower_eps1}
\end{align}
Since $\varphi^{-1}(1-2^{-(\eps+1)})=\ln\bc{2^{\eps+1}-1}\geq \eps\geq2^{-k/4}$, the fact that $\mu\in\cP^\dagger$ implies that
\begin{align}\label{eqlem_lower_eps2}
(k-1) \mu\bc{[0,2^{-(\eps+1)}]} &\leq  k\exp\bc{-2^{k/4} \eps}.
\end{align}
The assertion follows from \eqref{eqlem_lower_eps1} and \eqref{eqlem_lower_eps2}.

\subsubsection{Proof of \Lem~\ref{Lemma_sleq1}} \label{Sec_Lemma_sleq1}
We need to bound the tails of $-\ln\vPi^+$, which is the sum of a Poisson number of i.i.d.\ copies of $\vX$.
Having derived tail bounds for the individual summands $\vX$ and an approximation of $\Erw[\vX]$ already, we are going to deal with large deviations of the number of summands next.
Bennett's inequality Lemma \ref{Lemma_Bennett} directly implies that if $d = O\bc{k2^k}$, then
\begin{align}\label{Pois_conc}
\Pr\bc{\left|\vgamma^+- d/2 \right|\geq k^22^{5k/8}}\leq \mathrm{exp}\bc{-\Omega\bc{k^3 2^{k/4}}}.
\end{align}
In analogy to (\ref{eq_truncate_X}), for $i \geq 1$, set 
\begin{align}\label{eqbarXi+}
\bar{\vX}_i^+ = \vX_i ^+ \vecone\left\{\vX_i ^+\leq \bc{1-\eul^{-\beta}}2^{-(9/10)(k-1)} \right\}.
\end{align}
On the event 
\begin{align*}
\cE = \cbc{\left|\vgamma^+- d/2 \right| \leq k^22^{5k/8}} \cap \cbc{\max\cbc{\vX_1^+, \ldots, \vX_{\vgamma^+}^+} \leq \bc{1-\eul^{-\beta}}2^{-(9/10)(k-1)}} 
\end{align*}
Corollaries~\ref{bound_expectation} and~\ref{eq_diff_exp} yield
\begin{align}
\left|\sum_{i=1}^{\vgamma^+}\bc{\vX_i^+-\Erw[\vX_i^+]} - \sum_{i=1}^{d/2}\bc{\bar{\vX}_i^+-\Erw[\bar{\vX}_i^+]}\right|
& \leq \frac{d}{2} \Erw[\vX-\bar{\vX}] + \abs{\vgamma^+-\frac{d}{2}}\cdot\abs{\Erw[\vX] + \max\left\{\vX_1^+, \ldots, \vX_{\vgamma^+}^+\right\} } \nonumber \\
 & \leq \frac{d}{2}\mathrm{exp}\bc{-\Omega\bc{k 2^{k/4}}} + k^22^{5k/8} \bc{2^{-(k-1)} + O\bc{k2^{-10k/9}}+2^{-(9/10)(k-1)}} \nonumber \\
 &\leq k^{-5} 2^{-1-k/4} \label{bound_truncation}.
\end{align}
Therefore, using \Lem~\ref{Lemma_tails_of_X}, \Cor~\ref{bound_expectation}, (\ref{Pois_conc}) and (\ref{bound_truncation}), we obtain for any $t>0$,
\begin{align}
\Pr&\bc{\left| -\ln \vPi^+ - \frac{d}{2} \Erw[\vX]\right|\geq t} = \Pr\bc{\left| \sum_{i=1}^{\vgamma^+}\bc{\vX_i^+-\Erw[\vX_i^+]} + \Erw[\vX] \bc{\vgamma^+-\frac{d}{2}}\right|\geq t} \nonumber \\
\leq & 1-\Pr\bc{\cE} + \Pr\bc{ \cE \cap \left\{\abs{\sum_{i=1}^{\vgamma^+}\bc{\vX_i^+-\Erw[\vX_i^+]} - \sum_{i=1}^{d/2}\bc{\bar{\vX}_i^+-\Erw[\bar{\vX}_i^+]} + \sum_{i=1}^{d/2}\bc{\bar{\vX}_i^+-\Erw[\bar{\vX}_i^+]}+ \Erw[\vX] \bc{\vgamma^+-\frac{d}{2}}}\geq t\right\}} \nonumber\\
 \leq &\mathrm{exp}\bc{-\Omega\bc{k2^{k/4}}} + \Pr\bc{\left| \sum_{i=1}^{d/2}\bc{\bar{\vX}_i-\Erw[\bar{\vX}_i]}\right| \geq t-k^{-5}2^{-k/4}}. \label{Azu_Ben1}
  \end{align}
The last probability involving a sum with a deterministic number of bounded summands can be bounded by the Azuma-Hoeffding inequality, which yields for $t = \Omega\bc{2^{-k/4}}$ 
\begin{align*}
\Pr\bc{\left| \sum_{i=1}^{d/2}\bc{\bar{\vX}_i-\Erw[\bar{\vX}_i]}\right| \geq t-k^{-5} 2^{-k/4}} \leq 2 \mathrm{exp}\bc{- \frac{\bc{t-k^{-5} 2^{-k/4}}^2}{d\bc{1-\eul^{-\beta}}^{2}2^{-(18/10)(k-1)}}} = \mathrm{exp}\bc{-\Omega\bc{t^2 k^{-1}2^{4k/5}}}.
\end{align*}
Combining this estimate with \eqref{Azu_Ben1}, for $t = \Omega\bc{2^{-k/4}}$ we obtain
\begin{align}
\Pr\bc{\left|- \ln \vPi^+ - \frac{d}{2}\Erw[\vX]\right|\geq t} & \leq \mathrm{exp}\bc{-\Omega\bc{k 2^{k/4}}} + \mathrm{exp}\bc{-\Omega\bc{t^2 k^{-1} 2^{4k/5}}}  \leq \mathrm{exp}\bc{-\Omega\bc{k 2^{k/4}}} + \mathrm{exp}\bc{-\Omega\bc{t k 2^{k/4}}} \label{Compare_Poi_Azu}.
\end{align}
Because $\vPi^-,\vPi^+$ are identically distributed, \eqref{Compare_Poi_Azu} implies that 
\begin{align}\label{Compare_Poi_Azub}
\Pr\bc{\left|\ln \vPi^- + \frac{d}{2}\Erw[\vX]\right|\geq t}
& \leq \mathrm{exp}\bc{-\Omega\bc{k 2^{	k/4}}} + \mathrm{exp}\bc{-\Omega\bc{t k 2^{k/4}}}
\end{align}
as well.
Finally, for $2^{-k/4} \leq t \leq 1$ the second terms in (\ref{Compare_Poi_Azu})--\eqref{Compare_Poi_Azub} dominate
Therefore, for $s \in [2^{-k/4},1]$, we obtain 
\begin{align*}
\hat\mu\left(\left[\frac{\eul^s}{1+\eul^s},1\right]\right) & =  \Pr\bc{\bc{-\ln \vPi^- - \frac{d}{2}\Erw[\vX]} - \bc{-\ln \vPi^+ - \frac{d}{2}\Erw[\vX]} \geq s} \leq \mathrm{exp}\bc{-\Omega\bc{sk2^{k/4}}}  \leq \mathrm{exp}\bc{-s2^{k/4}},
\end{align*}
as claimed.

\subsubsection{Proof of \Lem~\ref{Lemma_sgeq1}} \label{Sec_Lemma_sgeq1}

For $s \geq 1$ let $\cE_s=\cbc{\left|\vgamma^+-d/2 \right| \leq s 2^{9k/10} k^{-5/4}}$.
Using  Bennett's inequality Lemma \ref{Lemma_Bennett}, 
we obtain for $s \geq 1$
\begin{align} \label{Poi_bound_no2}
1-\Pr\bc{\cE_s}&\leq  \mathrm{exp}\bc{-\Omega\bc{s 2^{4k/5} k^{-7/2}}}.
\end{align}
Combining (\ref{Poi_bound_no2}) and~\eqref{eqbarXi+} with the Azuma-Hoeffding inequality, we obtain for $s \geq 1$ 
\begin{align}
\Pr\bc{\left|\sum_{i=1}^{\vgamma^+}\bar{\vX}_i^+ - \frac{d}{2}\Erw[\bar{\vX}]\right|\geq \frac{s}{k}}
& \leq\pr\bc{\left| \sum_{i=1}^{d/2}\bc{ \bar{\vX}_i^+ -\Erw[\bar{\vX}]} \right| \geq \frac{s}{2k}} + \Pr\bc{\left|\vgamma^+ - d/2\right| \bc{1-\eul^{-\beta}}2^{-(9/10)(k-1)}\geq \frac{s}{2k}} \nonumber\\
& \leq 1-\pr\bc{\cE_s} +2 \exp\bc{- \frac{s^2}{4\bc{1-\eul^{-\beta}}^2k^2 d 2^{-(9/5)(k-1)}}} 
 \leq \mathrm{exp}\bc{-s2^{2k/3}}. \label{Poisson_bigs}
\end{align}
We are left to bound the difference between $-\ln \vPi^+$ and the sum of truncated random variables.
We write
\begin{align}\label{Poisson_bigs_1a}
 \sum_{i=1}^{\vgamma^+}\vX_i^+- \sum_{i=1}^{\vgamma^+} \bar{\vX}_i^+ 
& \leq \sum_{i=1}^{\vgamma^+}\vecone\cbc{\vX_i^+ \geq \bc{1-\eul^{-\beta}}2^{-(9/10)(k-1)}} + \sum_{i=1}^{\vgamma^+}\vecone\cbc{\vX_i^+ > 1 }\vX_i^+.
\end{align}

We now compare the right hand side with a more accessible distribution.
\Lem~\ref{Lemma_tails_of_X} shows that for $c:=1/1000$, for all $a \geq 1$, 
\begin{align}\label{eqNMtheta}
\Pr\bc{\vX_i^+ \geq \bc{1-\eul^{-\beta}}2^{-(9/10)(k-1)}} &\leq \mathrm{exp}\bc{-ck2^{k/4}},&
\Pr\bc{\vX_i^+\geq a} &\leq \mathrm{exp}\bc{-2ack2^{k/4}}.
\end{align}
Set $\vartheta= c k 2^{k/4}$.
Then \eqref{eqNMtheta} shows that for $z \geq 0$,
\begin{align*}
\Pr\bc{\vecone\cbc{\vX_i^+ > \bc{1-\eul^{-\beta}}2^{-(9/10)(k-1)}}\bc{1+\vecone\cbc{\vX_i^+ > 1 }\vX_i^+}> z} \leq \begin{cases}
\Pr\bc{\vX_i^+ > \bc{1-\eul^{-\beta}}2^{-(9/10)(k-1)}}\leq \eul^{-\vartheta}, &\mbox{ if } z < 1, \\
\Pr\bc{\vX_i^+ \geq 1}\leq \eul^{-2\vartheta}, &\mbox{ if } 1 \leq z < 2, \\
\Pr\bc{\vX_i^+ \geq z-1}\leq \eul^{-2(z-1)\vartheta} \leq \eul^{-z\vartheta}, &\mbox{ if } z \geq 2.
\end{cases}
\end{align*}
Hence, we can estimate $\vX_i^+$ as follows.
Let $\bc{\vI_i^+}_{i \geq 1}$ be a sequence of $\Be\bc{\eul^{-\vartheta}}$ random variables, let $(\vZ_i^+)_{i \geq 1}$ be a sequence of exponential random variables with mean $1/\vartheta$ and let $\bar{\vgamma}^+$ be a $\Po(d/2)$ random variable, all  mutually independent.
Then for all $i \geq 1$, $z \geq 0$ we have
\begin{align*}
\Pr\bc{\vecone\cbc{\vX_i^+ > \bc{1-\eul^{-\beta}}2^{-(9/10)(k-1)}}\bc{1+\vecone\cbc{\vX_i^+ > 1 }\vX_i^+}> z}\leq
\Pr\bc{\vI_i^+(1+\vZ_i^+) > z} = \begin{cases}
\eul^{-\vartheta}, & z < 1, \\
\eul^{-z\vartheta}, & z \geq 1.
\end{cases}
\end{align*}
Thus, $\vecone\cbc{\vX_i^+ > \bc{1-\eul^{-\beta}}2^{-(9/10)(k-1)}}(1+\vecone\cbc{\vX_i^+ > 1 }\vX_i^+)$ is stochastically dominated by $\vI_i^+(1+\vZ_i^+)$.
Therefore, we also obtain stochastic dominance for the sums of these random variables, i.e., 
\begin{align}\label{eqNoelaStochDom}
\sum_{i=1}^{\vgamma^+}&\vecone\bc{\vX_i^+ > \bc{1-\eul^{-\beta}}2^{-(9/10)(k-1)}}(1+\vecone\bc{\vX_i^+ > 1 }\vX_i^+) \preceq 
\vSigma^+_1+\vSigma^+_2&\mbox{where}\quad
\vSigma^+_1=\sum_{i=1}^{\bar{\vgamma}^+}\vI_i^+,\qquad
\vSigma^+_2=\sum_{i=1}^{\bar{\vgamma}^+}\vI_i^+\vZ_i^+.
\end{align}
Clearly, $\vSigma^+_1$ has distribution $\Po\bc{d\eul^{-\vartheta}/2}$.
Bennett's inequality (\ref{Poi_Ben}) therefore yields for $s \geq 1$, 
\begin{align}\label{Sigma_plus}
\Pr\bc{\vSigma^+_1 - \frac{d}{2}\eul^{-\vartheta} \geq s/\sqrt{k}} \leq \mathrm{exp}\bc{- \frac{s}{\sqrt{k}} \bc{\ln\bc{\frac{\eul^{\vartheta}}{d\sqrt{k}}}-1}} \leq \mathrm{exp}\bc{- \frac{s}{2\sqrt{k}}\bc{\vartheta -\ln\bc{d\sqrt{k}}}} \leq \mathrm{exp}\bc{-sk^{1/3}2^{k/4}}.
\end{align}

Let us now turn our attention to $\vSigma^+_2$. 

\begin{claim}\label{Tail_Gamma}
For all $s > d/(2\vartheta\eul^{\vartheta})$ we have
$\Pr\bc{\vSigma^+_2\geq s} \leq  \mathrm{exp}\bc{-\vartheta s \bc{1-\sqrt{\frac{d}{2s \vartheta \eul^{\vartheta}}}}^2}.$
\end{claim}
\begin{proof}
Given $\vSigma^+_1$, $\vSigma^+_2$ is a sum of $\vSigma^+_1$ independent exponential random variables with parameter $\vartheta$ and therefore $\Gamma\bc{\vSigma^+_1,\vartheta}$-distributed (where $\vSigma^+_1$ and $\vartheta$ denote the form and scale parameters, respectively).
Therefore, for $0 \leq t < \vartheta$,
\begin{align*}
\Erw\brk{\mathrm{exp}\bc{t\vSigma_2^+}} = \Erw\brk{\Erw\brk{\exp\bc{t\Gamma\bc{\vSigma^+_1,\vartheta}}\vert \vSigma_1^+}} = \Erw\brk{\bc{1-t/\vartheta}^{-\vSigma^+_1}} = \mathrm{exp}\bc{\frac{dt}{2\eul^{\vartheta}(\vartheta - t)}}.
\end{align*}
Consequently, for $s>0$ and $0 < t < \vartheta$,
\begin{align*}
\Pr\bc{\vSigma_2^+\geq s} \leq \frac{\Erw\brk{\mathrm{exp}\bc{t\vSigma^+_2}}}{\exp\bc{st}} = \mathrm{exp}\bc{\frac{d t}{2\eul^{\vartheta}(\vartheta - t)}-st}.
\end{align*}
With the choice $t^\ast=\vartheta - \sqrt{(d\vartheta \eul^{-\vartheta})/(2s)}$, which lies between $0$ and $\vartheta$ for $s> \frac{d}{2}\vartheta^{-1} \eul^{-\vartheta}$, 
we find
\begin{align*}
\Pr\bc{\vSigma_2^+ \geq s} &\leq \Erw\brk{\mathrm{exp}\bc{t^\ast\vSigma_2^+-t^\ast s}} = \mathrm{exp}\bc{-\vartheta s \bc{1-\sqrt{\frac{d}{2s \vartheta \eul^{\vartheta}}}}^2},
\end{align*}
as desired.
\end{proof}

\begin{proof}[Proof of \Lem~\ref{Lemma_sgeq1}]
Claim \ref{Tail_Gamma} implies that for all $s \geq 1$, 
\begin{align}\label{Sigma_infinity}
\Pr\bc{\vSigma^+_2 \geq s/\sqrt{k}} \leq \exp\bc{-sk^{1/3}2^{k/4}}.
\end{align}
Since the same estimates hold for $-\ln \vPi^-$, (\ref{Poisson_bigs}), (\ref{Sigma_plus}) and (\ref{Sigma_infinity}) show that for $s \geq 1$,
\begin{align*}
\hat\mu\left(\left[\frac{\eul^s}{1+\eul^s},1\right]\right) &  \leq 2 \Pr\bc{\left|\ln \vPi^++ \frac{d}{2}\Erw[\bar{\vX}]\right| \geq \frac{s}{2}} \leq 2 \bc{\Pr\bc{\left|\vSigma^+_1 + \vSigma^+_2 \right|\geq \frac{s}{4}} + \Pr\bc{\left|\sum_{i=1}^{\vgamma^+}\bar{\vX}_i^+ - \frac{d}{2}\Erw[\bar{\vX}]\right|\geq \frac{s}{4}}} \leq \mathrm{exp}\bc{-sk^{1/4}2^{k/4}},
\end{align*}
as claimed.
\end{proof}

\subsection{Proof of \Prop~\ref{convergence}}\label{Sec_convergence}
The proof is an adaptation of the proof of \cite[Lemma 4.4]{DSS3}, where the authors showed a corresponding statement for the Survey Propagation distributional recursion, which is more complicated than the present Belief Propagation recurrence. 
	Thus, we follow the path beaten in~\cite{DSS3}, simplifying the argument where possible.
Throughout this section $r$ denotes the smallest even integer greater than $2^{k/10}$. 
We set out to prove that on $\cP^\dagger$ the operator $\cR$ is a $W_r$-contraction.
Hence, we consider two probability distributions $\rho,\rho'\in\cP^\dagger$.
Let $(\veta_{ij}^+,\vchi_{ij}^+)_{i,j\geq1}$, $(\veta_{ij}^-,\vchi_{ij}^-)_{i,j}$ be independent identically distributed pairs of numbers in $[0,1]$
such that the first components $\veta_{ij}^{\pm}$ all have distribution $\rho$ and the second components $\vchi_{ij}^{\pm}$ have distribution $\rho'$ and $\Erw[|\veta_{ij}^\pm-\vchi_{ij}^\pm|^r]^{1/r}=W_r\bc{\rho,\rho'}$.
Let $\veta=(\veta_{ij}^+,\veta_{ij}^-)_{i,j}$, $\vchi=(\vchi_{ij}^+,\vchi_{ij}^-)_{i,j}$ and, recalling \eqref{rec}, define
\begin{align*}
\hat\veta&=R(\vgamma^+,\vgamma^-,\veta),&\hat\vchi&=R(\vgamma^+,\vgamma^-,\vchi).
\end{align*}
To prove \Prop~\ref{convergence} we are going to couple $\hat\veta$, $\hat\vchi$ such that
\begin{align}\label{Eq_contraction}
\Erw\brk{\left|\hat\veta-\hat\vchi\right|^r} \leq 2^{-kr/11} \Erw\brk{\abs{\veta_{11}^+-\vchi_{11}^+}^r} = W_r(\rho, \rho')^r.
\end{align}

To construct the coupling let $\vgamma\disteq\Po(d)$.
Moreover, let $\bc{\vJ_i}_{i\geq1}$ be the uniform i.i.d.\ signs; $\vgamma$ and the $\vJ_i$ are mutually independent and independent of everything else.
Then setting
$\vgamma^+= \sum_{i=1}^{\vgamma} \vecone\{\vJ_i = 1\},$, $\vgamma^-= \sum_{i=1}^{\vgamma}\vecone\{\vJ_i = -1\}$
ensures that $\vgamma^+$, $\vgamma^-$ are independent and $\Po\bc{d/2}$ variables.
Further, let
\begin{align*}
\vP^\pm&=\prod_{i=1}^{\vgamma^\pm}\left(1 - \left(1-\eul^{-\beta}\right) \prod_{j=1}^{k-1}\veta_{ij}^\pm\right),&\hat\veta&=\frac{\vP^+}{\vP^++\vP^-},&
\vQ^{\pm}&=\prod_{i=1}^{\vgamma^\pm}\left(1 - \left(1-\eul^{-\beta}\right) \prod_{j=1}^{k-1}\vchi_{ij}^\pm\right),&\hat\vchi&=\frac{\vQ^+}{\vQ^++\vQ^-}.
\end{align*}
Then
\begin{align}\label{eqNumber1}
\hat\veta-\hat\vchi&
= \frac{\vP^+}{\vP^++\vP^-} -\frac{\vQ^+}{\vQ^++\vQ^-}= \varphi\bc{\ln{\frac{\vP^+}{\vP^-}}} - \varphi\bc{\ln{\frac{\vQ^+}{\vQ^-}}}.
\end{align}
Finally, to estimate the last expression we introduce
\begin{align*}
\vDelta_i =\ln\left(1 - \left(1-\eul^{-\beta}\right) \prod_{j=1}^{k-1}\vchi_{ij}^+ \right)- \ln\left(1 - \left(1- \eul^{-\beta}\right) \prod_{j=1}^{k-1}\veta_{ij}^+ \right).
\end{align*}
Since $0\leq\varphi'(x)\leq1$ for all $x$, \eqref{eqNumber1} and the mean value theorem imply
\begin{align}\label{eqNumber2}
\Erw\brk{\left|\hat\veta-\hat\vchi\right|^r} & \leq\Erw\brk{\left|\ln\frac{\vP^+}{\vP^-} -\ln\frac{\vQ^+}{\vQ^-} \right|^r} = \Erw\brk{\left|\sum_{i=1}^{\vgamma}\vJ_i\vDelta_i \right|^r}.
\end{align}
Several steps are required to bound the r.h.s.\ of~\eqref{eqNumber2}.

\begin{lemma} \label{Pois_CLT}
We have
$\Erw\brk{\left|\sum_{i=1}^{\vgamma}\vJ_i\vDelta_i \right|^r} \leq  (r-1)!! \ \Erw\brk{\vgamma^{r/2}} \Erw\brk{\vDelta_1^r}.$
\end{lemma}
\begin{proof}
Since $r$ is even we have
\begin{align}\label{eqPois_CLT1}
\Erw\brk{\left|\sum_{i=1}^{\vgamma}\vJ_i\vDelta_i\right|^r} & 
= \Erw\brk{\sum_{i_1, \ldots, i_r=1}^{\vgamma}\prod_{\ell=1}^r\vJ_{i_\ell}\vDelta_{i_\ell}}
= \eul^{-d} \sum_{\gamma=0}^\infty \frac{d^\gamma}{\gamma!}\sum_{i_1, \ldots, i_r=1}^\gamma\Erw\brk{\prod_{\ell=1}^r\vJ_{i_\ell} }\Erw\brk{\prod_{\ell=1}^r\vDelta_{i_\ell}}. 
\end{align}
Let us now fix $\gamma\geq0$ and $i_1, \ldots, i_r \in [\gamma]$.
Moreover, for $h\in[\gamma]$ let $N_h=N_h(i_1, \ldots, i_r)$ be the number of indices $\ell$ such that $i_\ell=h$.
Then
\begin{align}\label{eqPois_CLT2}
E(i_1, \ldots, i_r)=\Erw\brk{\prod_{\ell=1}^r\vJ_{i_\ell} }\Erw\brk{\prod_{\ell=1}^r\vDelta_{i_\ell}}=
\Erw\brk{\prod_{h=1}^\gamma\vJ_{h}^{N_h}}\Erw\brk{\prod_{h=1}^\gamma\vDelta_{h}^{N_h}}
=\prod_{h=1}^\gamma\Erw\brk{\vJ_{h}^{N_h}}\Erw\brk{\vDelta_{h}^{N_h}}.
\end{align}
In particular, $E(i_1, \ldots, i_r)=0$ unless all $N_h$ are even.
Furthermore, if all $N_h$ are even, then Jensen's inequality yields
\begin{align}\label{eqPois_CLT3}
E(i_1, \ldots, i_r)&=\prod_{h=1}^\gamma\Erw\brk{\vDelta_{h}^{N_h}}\leq\Erw\brk{\vDelta_{1}^{r}}.
\end{align}
Finally, let  $\ff(\gamma,r)= \left|\cbc{(i_1, \ldots, i_r) \in [\gamma]^r:  N_h \hspace{-0.2 cm} \mod 2 = 0 \ \text{for all} \ h \in [\gamma]} \right|$
be the number of index sequences in which each index appears an even number of times.
Combining \eqref{eqPois_CLT1}--\eqref{eqPois_CLT3}, we conclude that
\begin{align}\label{eqPois_CLT4}
\Erw\brk{\left|\sum_{i=1}^{\vgamma}\vJ_i\vDelta_i\right|^r} &\leq 
\eul^{-d}\Erw\brk{\vDelta_{1}^{r}} \sum_{\gamma=0}^\infty \frac{d^\gamma \ff(\gamma,r)}{\gamma!}.
\end{align}

We now claim that for all $\gamma$,
\begin{align}\label{annoying_counting_argument}
\ff(\gamma,r) \leq (r-1)!! \ \gamma^{r/2}.
\end{align}
To see this consider a family $(i_1, \ldots, i_r)$ such that all $N_h$ are even and think of $1,\ldots,r$ as the vertices of a complete graph.
Since all $N_h$ are even we can find a perfect matching of $1,\ldots,r$ such that any two indices $u,v$ that are matched satisfy $i_u=i_v$.
Moreover, if we label the matching edge $uv$ with the value $i_u=i_v\in[\gamma]$, then $(i_1, \ldots, i_r)$ can be recovered from the labelled matching.
Because the total number of perfect matchings equals $(r-1)!!$ and there are $\gamma^{r/2}$ possible labellings, we obtain \eqref{annoying_counting_argument}.
Finally, combining \eqref{eqPois_CLT4} and \eqref{annoying_counting_argument}, we obtain
\begin{align*}
\Erw\brk{\left|\sum_{i=1}^{\vgamma}\vJ_i\vDelta_i \right|^r} \leq
\sum_{\gamma=0}^\infty \frac{d^\gamma \gamma^{r/2}(r-1)!!}{\gamma!\exp(d)}\Erw\brk{\vDelta_1^r} = (r-1)!! \ \Erw\brk{\vgamma^{r/2}} \Erw\brk{\vDelta_1^r},
\end{align*}
as claimed.
\end{proof}

Hence, we are left to bound $\Erw\brk{\vgamma^{r/2}}$ and $\Erw\brk{\vDelta_1^r}$.

\begin{lemma}\label{moment_bound_p2}
We have $\Erw\brk{\vgamma^{r/2}}\leq\bcfr{r(d+1)}{2}^{r/2}$.
\end{lemma}
\begin{proof}
Let $\stirling NK$ denote the Stirling number of the second kind.
Because the factorial moments of the Poisson variable $\vgamma$ satisfy $\Erw[\prod_{j=1}^h\vgamma-j+1]=d^h$, we obtain
\begin{align*}
\Erw\brk{\vgamma^{r/2}}&=\sum_{h=0}^{r/2}\stirling{r/2}{h}\Erw\brk{\prod_{j=1}^h\vgamma-j+1}=\sum_{h=0}^{r/2}\stirling{r/2}{h}d^h
\leq\sum_{h=0}^{r/2}\binom{r/2}hh^{r/2-h}d^h\leq\bcfr{r(d+1)}{2}^{r/2},
\end{align*}
as desired.
\end{proof}

To bound $\Erw\brk{\vDelta_1^r}$ set $\vU_j= \prod_{h=1}^{j-1}\vchi_{1h}^+\prod_{h=j+1}^{k-1} \veta_{1h}^+$ for $j \in [k-1]$.
Writing a telescoping sum and applying the mean value theorem, we obtain
\begin{align}
|\vDelta_1| &\disteq \left|\ln\bc{1- \bc{1-\eul^{-\beta}}\prod_{j=1}^{k-1}\veta_{1j}^+} - \ln\bc{1- \bc{1-\eul^{-\beta}}\prod_{j=1}^{k-1}{\vchi}_{1j}^+} \right| \nonumber\\
& = \left|\sum_{j=1}^{k-1}\ln{\frac{1- \bc{1-\eul^{-\beta}}\vU_j\veta_{1j}^+}{1- \bc{1-\eul^{-\beta}}{\vU_j\vchi}_{1j}^+}}\right|
\leq \sum_{j=1}^{k-1}\left|\ln\bc{\frac{1- \bc{1-\eul^{-\beta}}\vU_j\veta_{1j}^+}{1- \bc{1-\eul^{-\beta}}\vU_j\vchi_{1j}^+}}\right|
\leq  \bc{1-\eul^{-\beta}}\sum_{j=1}^{k-1}\frac{\vU_j}{1-\vU_j}\left|\veta_{1j}^+-\vchi_{1j}^+\right|.
\label{eqNoelasBound1}
\end{align}
Since $\left|\sum_{j=1}^{k-1} z_j\right|^r \leq k^{r-1}\sum_{j=1}^{k-1} z_j^r$ for all $z_1, \ldots, z_r\geq 0$ by  H\"older's inequality, \eqref{eqNoelasBound1} implies that 
\begin{align}
\Erw\brk{\left|\vDelta_1 \right|^r} &\leq \bc{1-\eul^{-\beta}}^r\Erw\brk{\left|\sum_{j=1}^{k-1}\frac{\vU_j}{1-\vU_j}\left|\veta_{1j}^+-{\vchi}_{1j}^+\right|\right|^r} 
\leq k^{r-1} \bc{1-\eul^{-\beta}}^r\sum_{j=1}^{k-1}\Erw\brk{\left|\frac{\vU_j}{1-\vU_j}\right|^r} \Erw\brk{\left|\veta_{1j}^+-\vchi_{1j}^+\right|^r} \nonumber\\
& \leq k^r \Erw\brk{\left|\veta_{1j}^+-\vchi_{1j}^+\right|^r} \max_{j=1, \ldots, k-1 }\Erw\brk{\left|\frac{\vU_j}{1-\vU_j}\right|^r}\nonumber\\
&\leq k^r\Erw\brk{\left|\veta_{1j}^+-\vchi_{1j}^+\right|^r} \max_{j=1, \ldots, k-1 } \Erw\brk{\vU_j^{2r}}^{1/2}\Erw\brk{(1-\vU_j)^{-2r}}^{1/2}\qquad\mbox{[by Cauchy-Schwarz].}\label{eqNoelasBound2}
\end{align}
Hence, we need to bound $\Erw\brk{\bc{1-\vU_j}^{-2r}}$ and $\Erw\brk{\vU_j^{2r}}$.
To this end we need the following estimate.

\begin{lemma} \label{Tail_bound_Cor}
For any $\mu\in\cP^\dagger$ and $s \in [2^{-k/4},1]$, we have $\mu([1/2 +s,1]) \leq \mathrm{exp}\left(-s 2^{k/4}\right).$
\end{lemma}
\begin{proof}
We notice that $\eul^s/(1+\eul^s) \leq 1/2+s$. 
Therefore, \eqref{Tail_bound0} implies that
$
\mu([1/2 +s,1]) \leq \mu\bc{\brk{\frac{\eul^s}{1+\eul^s},1}} \leq \mathrm{exp}\left(-s 2^{k/4} \right),
$
as claimed.
\end{proof}

\begin{corollary}\label{Lemma_Uj}
We have $\Erw\brk{\bc{1-\vU_j}^{-2r}}\leq5^r$ and $\Erw\brk{\vU_j^{2r}}\leq 2^{-9(k-2)r/5}$ for all $j\in[k-1]$.
\end{corollary}
\begin{proof}
Suppose that $z \geq 2$ satisfies $\ln(z-1)\geq2^{-k/4}$.
Then \eqref{Tail_bound0} ensures that
\begin{align}
\Pr\brk{\bc{1-\vU_j}^{-1} \geq z} &= \Pr\brk{\prod_{h=1}^{j-1}\vchi_{1h}^+\prod_{h=j+1}^{k-1} \veta_{1h}^+ \geq 1 - \frac{1}{z}} \leq 
\pr\brk{\varphi(\ln(z-1))\leq\vchi_{11}^+} ^{j-1} \pr\brk{\varphi(\ln(z-1))\leq\veta_{11}^+}^{k-1-j} \nonumber \\
&\leq \mathrm{exp}\bc{-(k-2) \ln(z-1)2^{k/4}}  \leq (z-1)^{-2^{k/4}} \label{One_minus_u}.
\end{align}
Hence, letting $\xi=2+2^{-k/5}$, we obtain
\begin{align}
\Erw\brk{\bc{1-\vU_j}^{-2r}} &= \Erw\brk{\bc{1-\vU_j}^{-2r}\vecone\left\{ \bc{1-\vU_j}^{-1} < \xi\right\}} +  \Erw\brk{\bc{1-\vU_j}^{-2r}\vecone\left\{ \bc{1-\vU_j}^{-1} \geq \xi\right\}} \nonumber \\
& \leq 2^{2r+1} + 2r\int_{\xi}^{\infty}  z^{2r-1}\Pr\brk{\bc{1-\vU_j}^{-1} \geq z} \dd z 
 \stackrel{\eqref{One_minus_u}}{\leq} 2^{2r+1} + 2r \int_\xi^{\infty} \frac{z^{2r-1}\dd z}{(z-1)^{2^{k/4}}}  \nonumber \\
 &\leq 2^{2r+1} + 2r \int_{1}^{\infty} \frac{(z+1)^{2r-1}}{z^{2^{k/4}}}\dd z \leq 2^{2r+1} + 2r \int_{1}^{\infty} \frac{(2z)^{2r-1}}{z^{2^{k/4}}} \dd z \nonumber \\
& = 2^{2r+1} + r 4^r \int_{1}^{\infty} \frac{\dd z}{z^{2^{k/4}-(2r-1)}}=  2^{2r+1} + \frac{r4^r}{2^{k/4}-2r} \leq 5^r \label{factor_inverse} .
\end{align}
Next, in order to bound $\Erw\brk{\vU_j^{2r}}$, we use the fact that all $\veta_{ij}^+$ are bounded above by $1$ and Lemma~\ref{Tail_bound_Cor} to compute
\begin{align}
\Erw\brk{(\veta_{11}^+)^{2r}}& = \Erw\brk{(\veta_{11}^+)^{2r}\vecone\left\{ (\veta_{11}^+)^{2r} \leq \frac{\bc{1+2^{-k/30}}^{2r}}{2^{2r}}\right\}} + \Erw\brk{(\veta_{11}^+)^{2r}\vecone\left\{( \veta_{11}^+)^{2r} > \frac{\bc{1+2^{-k/30}}^{2r}}{2^{2r}}\right\}}\nonumber  \\
& \leq  \frac{\bc{1+2^{-k/30}}^{2r}}{2^{2r}} + \pr\brk{\veta_{11}^+>\frac{1+2^{-k/30}}{2}}\leq\frac{\mathrm{exp}\bc{2r2^{-k/30}}}{2^{2r}} + \mathrm{exp}\bc{-2^{-k/30-1}2^{k/4}} 
 \leq 2^{-9r/5} \label{Other_factor}.
\end{align}
Since due to \Prop~\ref{Tail_bounds} the same bound holds for $\Erw\brk{(\vchi_{11}^+)^{2r}}$, we obtain $\Erw\brk{\vU_j^{2r}}\leq 2^{-9(k-2)r/5}$.
\end{proof}

\begin{corollary}\label{Cor_Noela_Deltaj}
We have $\Erw\brk{\vDelta_1^r}\leq (5k^2)^{r/2}2^{-9(k-2)r/10}\Erw\brk{\left|\veta_{1j}^+-\vchi_{1j}^+\right|^r}. $ 
\end{corollary}
\begin{proof}
This is an immediate consequence of \eqref{eqNoelasBound2} and \Cor~\ref{Lemma_Uj}.
\end{proof}

\begin{proof}[Proof of \Prop~\ref{convergence}]
Combining \eqref{eqNumber2}, Lemmas~\ref{Pois_CLT}, \ref{moment_bound_p2} and \Cor~\ref{Cor_Noela_Deltaj} and recalling that $d\leq k^22^k$, we obtain
\begin{align*}
\Erw\brk{\left|\hat\veta-\hat\vchi\right|^r}&\leq (r-1)!! \ \bcfr{r(d+1)}{2}^{r/2} k^r 5^{r/2}2^{-9r(k-2)/10} \Erw\brk{\left|\veta_{1j}^+-\vchi_{1j}^+\right|^r}\\ 
&\leq\frac{r!}{(r/2)!}\bcfr{r(d+1)}{2}^{r/2}2^{-0.89kr}\Erw\brk{\left|\veta_{1j}^+-\vchi_{1j}^+\right|^r}
\leq r^{r/2}2^{-0.4kr}\leq2^{-0.2kr}\Erw\brk{\left|\veta_{1j}^+-\vchi_{1j}^+\right|^r},
\end{align*}
which implies \eqref{Eq_contraction}.
Hence, $\cR$ is a $W_r$-contraction.
\end{proof}

\subsection{Proof of \Prop~\ref{Truncation_allowed}}\label{Sec_Truncation_allowed}
Throughout this section, additionally to $d \leq d_{k-\text{SAT}}$, we assume that $\pi$ is a probability distribution with slim tails.
Let $\veta$ be a random variable with distribution $\pi$; then $\veta$ satisfies
\begin{align} \label{initial_cond1}
\mathbb{P}\brk{\left|\veta-1/2 \right| \geq 2^{-k/10}} \leq 2^{-k/10}.
\end{align}

\subsubsection{Overview}
Our aim is to study the distribution $\cR^\ell(\pi)$.
Because of the inherent symmetry of the operator $\cR$, we may assume that $\pi\in\cP^\star$.
In the following, let $\vxi^{(\ell)}$ be a random variable with distribution $\cR^\ell(\pi)$. To prove the proposition, we need to bound the tails of $\vxi^{(\ell)}$.
We proceed in two steps.
First, in \Sec~\ref{Sec_Correct_FP} we derive the following estimate.

\begin{lemma}\label{Correct_FP}
There exists an integer $L=L(k)$ such that for all $\ell \geq L$ and all $s \in [2^{-k/4}, 1]$ the random variable $\vxi^{(\ell)}$ satisfies
\begin{align}\label{eqCorrect_FP_claim}
\Pr\brk{\abs{\ln{\frac{\vxi^{(\ell)}}{1-\vxi^{(\ell)}}}}\geq s} \leq \mathrm{exp}\bc{-s2^{k/4}}.
\end{align}
\end{lemma}

\noindent
Moreover, in \Sec~\ref{Sec_Correct_FP2} we will prove the following.
Let us introduce the shorthand $\vxi=\vxi^{(1)}$ for a random variable with distribution $\cR(\pi)$.

\begin{lemma}\label{Correct_FP2}
Assume that for a number $s\geq1/2$ we have
\begin{align*}
\Pr\brk{\abs{\ln{\frac{\veta}{1-\veta}}}\geq t}&\leq \mathrm{exp}\bc{-t2^{k/4}}&&\mbox{for all $t \in [2^{-k/4}, s]$.}
\end{align*}
Then
\begin{align*}
\Pr\brk{\abs{\ln{\frac{\vxi}{1-\vxi}}}\geq 2s} \leq \mathrm{exp}\bc{-2s \cdot 2^{k/4}}.
\end{align*}
\end{lemma}

\begin{proof}[Proof of \Prop~\ref{Truncation_allowed}]
Given $\eps>0$ choose $\ell_0=\ell_0(\eps)>0$ large enough and assume that $\ell\geq\ell_0(\eps)$.
As in \eqref{eqFunnyConstruction}, set
\begin{align*}
\tilde\vxi_\eps^{(\ell)} = \vxi^{(\ell)} \vecone\cbc{\vxi^{(\ell)} \in [\eps, 1-\eps]} + \frac{1}{2} \vecone\cbc{\vxi^{(\ell)} \notin [\eps, 1-\eps]}.
\end{align*}
Then \Lem s~\ref{Correct_FP} and~\ref{Correct_FP2} imply that $\tilde\vxi_\eps^{(\ell)}$ satisfies \eqref{Tail_bound0}.
Hence, $\cL(\tilde\vxi_\eps^{(\ell)})\in\cP^\dagger$; this establishes part (ii) of the proposition.
Furthermore, \Lem~\ref{Correct_FP2} also implies that $$\pr\brk{\vxi^{(\ell)} \notin[\eps,1-\eps]}<\eps$$ for large enough $\ell$, which is part (i).
\end{proof}

\subsubsection{Proof of \Lem~\ref{Correct_FP}}\label{Sec_Correct_FP}
The following lemma summarises the key step towards the proof of \Lem~\ref{Correct_FP}.
\begin{lemma}\label{Induction_for_smalls}
Let $\ell \geq 1$ be an integer such that $2^{-k/10-2(l-1)} \geq 2^{-k/4}$ and assume that
\begin{align} \label{initial_cond_ell}
\mathbb{P}\brk{\left|\veta-1/2 \right| \geq2^{-k/10-2(\ell-1)}} \leq \max\cbc{2^{-\frac{k}{10}\bc{\frac{k}{50}}^{\ell-1}}, \mathrm{exp}\bc{-2^{k/4}}}.
\end{align}
Then $ \mathbb{P}\brk{\left|\vxi-1/2 \right| \geq 2^{-k/10-2\ell}} \leq \max\cbc{2^{-\frac{k}{10}\bc{\frac{k}{50}}^{\ell}}, \mathrm{exp}\bc{-2^{k/4}}}.  $
\end{lemma}

The proof of Lemma \ref{Induction_for_smalls} proceeds in three steps.
We continue to let $(\veta_{i,j}^\pm)_{i,j\geq 1}$ be an array of independent random variables distributed as $\veta$.
Also let $\vgamma^+, \vgamma^-$ be independent Poisson variables with mean $d/2$, independent of $(\veta_{i,j}^\pm)_{i,j\geq 1}$, so that
\begin{align}\label{Random_walk}
\ln\frac{\vxi}{1-\vxi}&\disteq\sum_{i=1}^{\vgamma^+}\ln\bc{1-\bc{1-\mathrm{e}^{-\beta}}\prod_{j=1}^{k-1}\veta^+_{i,j}} - \sum_{i=1}^{\vgamma^-}\ln\bc{1-\bc{1-\mathrm{e}^{-\beta}}\prod_{j=1}^{k-1}\veta^-_{i,j}}.
\end{align}
For $i\geq1$ let
\begin{align*}
\cA_{i,\ell}^{\pm}= \cbc{\sum_{j=1}^{k-1}\vecone\cbc{\left|\veta_{i,j}^{\pm}-\frac{1}{2} \right| \geq2^{-k/10-2\ell+2}}\leq\frac{k-1}{10} }
\end{align*}
be the event that a `clause' $i$ receives no more than $(k-1)/10$ `atypical messages'.
Moreover, let
\begin{align*}
\cA_{\ell} =\bigcup_{m_1=0}^{\infty}\bigcup_{m_2=0}^{\infty}\bc{\cbc{\vgamma^+=m_1, \vgamma^-=m_2} \cap \bigcap_{i=1}^{m_1}\cA_{i,\ell}^+\cap\bigcap_{i=1}^{m_2}\cA_{i,\ell}^-}
\end{align*}
be the event that there is no clause with too many atypical messages.
Our first goal is to bound the probability that $\cA_{\ell}$ fails to occur given that for each message the probability of being atypical is small.

\begin{lemma}\label{Less_than_half}
Assume that 
for some $\frac{k}{10} \log 2\leq z\leq 2^{k/4}$ we have
\begin{align}\label{eqLess_than_half}
\pr\brk{\left|\veta-1/2 \right| \geq2^{-k/10-2\ell+2}} \leq \mathrm{exp}\bc{-z}.
\end{align}
Then
$\pr\brk{\cA_\ell} \geq 1-\exp\bc{-z k/50}/3.$
\end{lemma}
\begin{proof}
We first estimate $\pr[\cA_{i,\ell}^{\pm}]$.
Since the $\veta_{i,1}^\pm, \ldots, \veta_{i,(k-1)}^\pm$ are independent, the assumption \eqref{eqLess_than_half} implies together with the Chernoff bound that 
\begin{align*}
1-\pr\brk{\cA_{i,\ell}^{\pm}}&\leq \mathbb{Pr}\bc{\mathrm{Bin}\bc{k-1, \mathrm{exp}(-z)} > \frac{k-1}{10}} \leq
 \exp\bc{-(k-1)\KL{1/10}{\mathrm{exp}\bc{-z}}}\leq\mathrm{exp}\bc{-\frac{k-1}{20}z}.
\end{align*}
Hence, by (\ref{Poi_Ben}) and subadditivity,
\begin{align}\nonumber
1-\pr\brk{\cA_\ell} &\leq \pr\brk{\vgamma^++\vgamma^- > 2k2^k} + \pr\brk{\cbc{\vgamma^+ + \vgamma^- \leq 2k2^k} \setminus\cA_\ell} \\
& \leq \mathrm{exp}\bc{-\frac{3}{8}k2^k} + 2k2^k\mathrm{exp}\bc{-\frac{k-1}{20}z}\leq
\exp\bc{-\frac{3}{8}k2^k} + \exp\bc{-zk/40}.\label{eqLess_than_half1}
\end{align}
As $z \leq 2^{k/4} < 3k2^k/8$, \eqref{eqLess_than_half1} implies the claim.
\end{proof}

We now analyse the sum (\ref{Random_walk}) on the event $\cA_{\ell}$.
Let $c_k=  \bc{\frac{1}{2}+2^{-k/10-2\ell+2}}^{9(k-1)/10}.$

\begin{claim}\label{Lemma_NoelaonA}
For $s\geq 2^{-k/4}$ 
we have
 $\Pr\brk{\vecone\cA_\ell\cdot\left|\ln{\frac{\vxi}{1-\vxi}}\right|\geq s}\leq\mathrm{exp}\bc{-s2^{k/2}-1}.$
\end{claim}
\begin{proof}
On $\cA_\ell$ we have $\prod_{j=1}^{k-1}\veta^+_{i,j} \leq c_k$ for all $i \leq \vgamma^+$ and $\prod_{j=1}^{k-1}\veta^-_{i,j} \leq c_k$ for all $i \leq \vgamma^-$.
Letting
\begin{align*}
\vY_i^\pm&=\ln\bc{1-\bc{1-\mathrm{e}^{-\beta}}\prod_{j=1}^{k-1}\veta^\pm_{i,j}}\vecone\cbc{\prod_{j=1}^{k-1}\veta^\pm_{i,j} \leq c_k}
\end{align*}
we thus obtain
\begin{align}
 \Pr\brk{\vecone\cA_{\ell}\cdot\left|\ln{\frac{\vxi}{1-\vxi}}\right|\geq s}& = 
 \Pr\brk{\cA_{\ell} \cap \cbc{\left|\sum_{i=1}^{\vgamma^+}\vY_i^+ - \sum_{i=1}^{\vgamma^-}\vY_i^-\right| \geq s}}\nonumber\\
& \leq 
\Pr\brk{\left|\sum_{i=1}^{\vgamma^+}\bc{\vY_i^+ - \Erw\brk{\vY_i^+}}- \sum_{i=1}^{\vgamma^-}\bc{\vY_i^- - \Erw\brk{\vY_i^-}} + \Erw\brk{\vY_1^+}\bc{\vgamma^+-\frac{d}{2}} -  \Erw\brk{\vY_1^-}\bc{\vgamma^--\frac{d}{2}}\right| \geq s} \nonumber\\
& \leq
2\Pr\brk{\left|\sum_{i=1}^{\vgamma^+}\vY_i^+ - \Erw\brk{\vY_i^+}\right| \geq \frac{s}{4}} + 2\Pr\brk{\left| \Erw\brk{\vY_1^+}\bc{\vgamma^+-\frac{d}{2}}\right| \geq \frac{s}{4}}.
\label{eqNM1}
\end{align}
To estimate the last summand we first bound $\left|\Erw\brk{\vY_1^+}\right|$: using $-\ln(1-x)\leq x+O(x^2)$, we obtain
\begin{align*}
\left|\Erw\brk{\vY_1^+}\right| \leq - \ln\bc{1-\bc{1-\mathrm{e}^{-\beta}}c_k} \leq \frac{5}{4}c_k
\end{align*}
for $k$ sufficiently large.
Therefore, Bennett's inequality \Lem~\ref{Lemma_Bennett} shows that for $s \geq 2^{-k/4}$,
\begin{align}\label{Truncated_expectation}
\Pr\brk{\left| \Erw\brk{\vY_1^+}\bc{\vgamma^+-\frac{d}{2}}\right| \geq \frac{s}{4}} &
\leq \Pr\brk{\left|\vgamma^+-\frac{d}{2}\right| \geq \frac{s}{5c_k}} \leq \Pr\brk{\left|\vgamma^+-\frac{d}{2}\right| \geq \frac{s}{10}2^{9(k-1)/10}} \leq \exp\bc{-s2^{k/2}-2}.
\end{align}

For the second last summand in \eqref{eqNM1}, we condition on $\vgamma^+$ and apply the Azuma-Hoeffding inequality.
The definition of the random variable $\vY_i^+$ ensures that $ \left|\vY_i^+ - \Erw\brk{\vY_i^+}\right|  \leq - \ln\bc{1-\bc{1-\mathrm{e}^{-\beta}}c_k}\leq c_k + O\bc{c_k^2}.  $
Hence, as in the computation towards \eqref{eqLess_than_half1} for $s \geq 2^{-k/4}$ we obtain
\begin{align}
\Pr\brk{\left|\sum_{i=1}^{\vgamma^+}{\vY_i^+ - \Erw\brk{\vY_i^+}}\right| \geq \frac{s}{4}} &
\leq \Pr\brk{\vgamma^+ > k2^k} + \Pr\brk{\cbc{\vgamma^+ \leq k2^k} \cap \cbc{\left|\sum_{i=1}^{\vgamma^+}\bc{\vY_i^+ - \Erw\brk{\vY_i^+}}\right| \geq \frac{s}{4}}}\nonumber\\
& \leq  \mathrm{exp}\bc{-\frac{3}{8}k2^k} + 2 \mathrm{exp}\bc{-\frac{s^2}{100k2^k
c_k^2}} \leq \mathrm{exp}\bc{-s2^{k/2}-3}. 
\label{Truncated_expectation1}
\end{align}
Combining \eqref{eqNM1}, \eqref{Truncated_expectation} and \eqref{Truncated_expectation1} completes the proof.
\end{proof}

The last ingredient that we need for the proof of \Lem~\ref{Induction_for_smalls} is the following. 
\begin{claim}\label{From_eta_to_Like}
For each $s \in [0, 4]$ we have $\Pr\bc{\abs{\vxi-\frac{1}{2}} \geq \frac{s}{4}}\leq \Pr\bc{\abs{\ln{\frac{\vxi}{1-\vxi}}}\geq s}.$
\end{claim}
\begin{proof}
For all $s \geq 0$ we have  $\frac{\mathrm{e}^s}{1+\mathrm{e}^s} \leq \frac{1}{2} + \frac{s}{4}.$ 
Therefore,
\begin{align*}
\Pr\bc{\vxi - \frac{1}{2} \geq \frac{s}{4}} \leq \Pr\bc{ \vxi - \frac{1}{2} \geq \frac{\mathrm{e}^s}{1+\mathrm{e}^s} - \frac{1}{2}} = \Pr\bc{\ln\bc{\frac{\vxi}{1-\vxi}}\geq s}.
\end{align*}
The symmetry of $\vxi$, i.e., that due to the definition \eqref{rec} of $\cR$ the random variables $\vxi$ and $1-\vxi$ are identically distributed, therefore implies the claim.
\end{proof}

\begin{proof}[Proof of \Lem~\ref{Induction_for_smalls}]
Assume that  (\ref{initial_cond_ell}) is satisfied for some $\ell \geq 1$.
Then also the assumption of Lemma \ref{Less_than_half} with $$z=\min\cbc{\frac{k}{10}\bc{\frac{k}{50}}^{\ell-1}\ln2, 2^{k/4}}$$ is satisfied, and we will use this estimate to bootstrap (\ref{initial_cond_ell}).
Indeed, \Lem~\ref{Less_than_half} and Claims~\ref{Lemma_NoelaonA} and~\ref{From_eta_to_Like} yield 
\begin{align*} 
\Pr\bc{\abs{\vxi-\frac{1}{2}} \geq 2^{-k/10-2\ell}}
\leq \Pr\bc{\abs{\ln\frac{\vxi}{1-\vxi}}\geq  2^{-k/10-2\ell+2}} 
\leq \max\cbc{2^{- \frac{k}{10}\bc{\frac{k}{50}}^\ell}, \mathrm{exp}\bc{-2^{k/4}}},
\end{align*}
as claimed.
\end{proof}

\begin{proof}[Proof of \Lem~\ref{Correct_FP}]
The proof is by induction on $\ell$.
Since we assume that $\veta$ satisfies (\ref{initial_cond1}), (\ref{initial_cond_ell}) holds for $\ell=1$.
Therefore, we may repeatedly apply \Lem~\ref{Induction_for_smalls} until for the first time $ 2^{-\frac{k}{10}-2\ell} < 2^{-k/4}, $which happens after $\lceil 3k/40 \rceil -1$ steps. At this point, also 
\begin{align}\label{eqCorrect_FP1}
\max\cbc{2^{- \frac{k}{10}\bc{\frac{k}{50}}^\ell}, \mathrm{exp}\bc{-2^{k/4}}} = \mathrm{exp}\bc{-2^{k/4}}.
\end{align}
Therefore, \eqref{eqCorrect_FP1} implies that $ \Pr\bc{\abs{\vxi^{(\ell)}-\frac{1}{2}} \geq s} \leq \mathrm{exp}\bc{-s2^{k/4}}\ \mbox{for all $s \in [2^{-k/4},1]$ }, $
whence \eqref{eqCorrect_FP_claim} holds for all $\ell \geq \lceil 3k/40 \rceil +1$.
\end{proof}

\subsubsection{Proof of \Lem~\ref{Correct_FP2}}\label{Sec_Correct_FP2}
 We combine some of the elements of the proofs of \Lem s~\ref{Tail_bounds} and~\ref{Correct_FP}.
We continue to denote by $(\veta_{i,j}^\pm)_{i,j\geq 1}$ an array of independent copies of $\veta$.
Moreover, $\vgamma^+, \vgamma^-$ are Poisson variables with mean $d/2$, independent of $(\veta_{i,j}^\pm)_{i,j\geq 1}$.
Further, for $s \geq 2^{-k/4}$ and $i=1, \ldots, \vgamma^\pm$ let
\begin{align*}
 \cA_{i,s}^{\pm}&=\cbc{\sum_{j=1}^{k-1}\vecone\cbc{\abs{\ln{\frac{\veta^\pm_{i,j}}{1-\veta^\pm_{i,j}}}} \geq s}\leq\frac{k-1}{10}}, &
\cA_s&=\bigcup_{m_1=0}^{\infty}\bigcup_{m_2=0}^{\infty}\bc{\cbc{\vgamma^+=m_1, \vgamma^-=m_2} \cap \bigcap_{i=1}^{m_1}\cA_{i,s}^+\cap\bigcap_{i=1}^{m_2}\cA_{i,s}^-}.
\end{align*}
For the event $\cA_s$ we prove an analogue of Lemma \ref{Less_than_half}.

\begin{lemma}\label{Less_than_half2}
Assume that for some $s \geq \frac{1}{2}$ we have
\begin{align*}
\pr\brk{\abs{\ln\frac{\veta}{1-\veta}}\geq s} \leq \mathrm{exp}\bc{-s2^{k/4}}.
\end{align*}
Then $\pr\bc{\cA_s} \geq1-\frac{1}{3} \mathrm{exp}\bc{-s2^{1+k/4}}.$
\end{lemma}
\begin{proof}
Because the $\veta_{i,j}^{\pm}$ are mutually independent, the Chernoff bound yields
\begin{align}\label{eqLess_than_half2_1}
\pr\brk{\cA_{i,s}^{\pm}} \geq 1-\mathrm{exp}\bc{-\frac{s(k-1)}{20}2^{k/4}}.
\end{align}
Further, 
for $s>1/2$ \Lem~\ref{Lemma_Bennett} and \eqref{eqLess_than_half2_1} yield
\begin{align*}
1-\pr\bc{\cA_s} &\leq \pr\brk{\vgamma^++\vgamma^- > 4sk2^k} + \pr\brk{\cbc{\vgamma^++\vgamma^- \leq 4sk2^k}\setminus\cA_s} \\
& \leq \mathrm{exp}\bc{-\bc{\ln\bc{4s}-\frac{4s-1}{4s}}4sk2^k} + 4sk2^k\mathrm{exp}\bc{-\frac{k-1}{20}s2^{k/4}} \leq \frac{1}{3} \mathrm{exp}\bc{-2s2^{k/4}},
\end{align*}
as desired.
\end{proof}

Consider the non-negative random variables
\begin{align*}
\vX_i^\pm&= -\ln\bc{1-\bc{1-\eul^{-\beta}}\prod_{j=1}^{k-1}\veta_{ij}^\pm}\geq0,&
\vY_i^\pm= \vX_i^{\pm}\vecone\cbc{\vX_i^\pm \leq \bc{1-\eul^{-\beta}}2^{-9(k-1)/10}}.
\end{align*}
Then
\begin{align}
\pr&\bc{\vecone\cA_s\ln{\frac{\vxi}{1-\vxi}}\geq 2s} \leq\pr\bc{\cA_s \cap \cbc{\sum_{i=1}^{\vgamma^-}\vX_i^- - \sum_{i=1}^{\vgamma^+}\vX_i^+ \geq 2s}}\nonumber\\
& \leq\pr\bc{\cA_s \cap \cbc{\sum_{i=1}^{\vgamma^-}\bc{\vX_i^- - \vY_i^-} - \sum_{i=1}^{\vgamma^+}\bc{\vX_i^+ - \vY_i^+} + \sum_{i=1}^{\vgamma^-}\bc{\vY_i^- - \Erw\brk{\vY_i^-}} - \sum_{i=1}^{\vgamma^+}\bc{\vY_i^+ - \Erw\brk{\vY_i^+}} + \Erw\brk{\vY_1^-}\bc{\vgamma^- - \vgamma^+}\geq 2s}}\nonumber \\
& \leq\pr\bc{\cA_s \cap \cbc{\sum_{i=1}^{\vgamma^-}\bc{\vX_i^--\vY_i^-} \geq \frac{s}{3}}} +  2 \cdot \pr\bc{\abs{\sum_{i=1}^{\vgamma^-}\bc{\vY_i^- - \Erw\brk{\vY_i^-}}} \geq \frac{s}{3}} +  2 \cdot \pr\bc{\Erw\brk{\vY_1^-}\abs{\vgamma^- -\frac{d}{2}}\geq \frac{s}{3}}.
\label{eqalmostDone1}
\end{align}

We proceed to bound the three terms on the r.h.s.\ of \eqref{eqalmostDone1} separately, starting with the one in the middle.

\begin{lemma}\label{bound_Azuma}
We have $\pr\brk{\abs{\sum_{i=1}^{\vgamma^-}{\vY_i^--\Erw\brk{\vY_i^-}}} \geq \frac{s}{3}}  \leq \frac{1}{9} \mathrm{exp}\bc{-2s2^{k/4}}.$
\end{lemma}
\begin{proof}
\Lem~\ref{Lemma_Bennett} and Azuma's inequality yield
\begin{align*}
\pr\bc{\abs{\sum_{i=1}^{\vgamma^-}{\vY_i^- - \Erw\brk{\vY_i^-}}} \geq \frac{s}{3}} & \leq \Pr\bc{\vgamma^- > 2sk2^k} + \Pr\bc{\cbc{\vgamma^- \leq 2sk2^k} \cap \cbc{\abs{\sum_{i=1}^{\vgamma^-}{\vY_i^- - \Erw\brk{\vY_i^-}} }\geq \frac{s}{3}}}\\
& \leq \mathrm{exp}\bc{-\bc{\ln\bc{4s} - \frac{4s-1}{4s}}2sk2^k} + 2\mathrm{exp}\bc{-\frac{s^2}{36 sk2^k2^{-9(k-1)/5}}}\leq \frac{1}{9} \mathrm{exp}\bc{-2s2^{k/4}},
\end{align*}
as claimed.
\end{proof}

The rightmost term from \eqref{eqalmostDone1} is next.

\begin{lemma}\label{bound_Bennett}
We have
$\pr\brk{\Erw\brk{\vY_1^-}\abs{\vgamma^--\frac{d}{2}}\geq \frac{s}{3}} \leq \frac{1}{9} \mathrm{exp}\bc{-2s2^{k/4}}.$
\end{lemma}
\begin{proof}
The definition of $\vY_1^\pm$ ensures that $\Erw\brk{\vY_1^-} \leq 2^{-9(k-1)/10}$.
Therefore, \Lem~\ref{Lemma_Bennett} yields
\begin{align*}
\pr\bc{\Erw\brk{\vY_1^-}\abs{\vgamma^--\frac{d}{2}}\geq \frac{s}{3}}&\leq\pr\bc{\left|\vgamma^- -\frac{d}{2}\right|\geq \frac{s}{3}2^{9(k-1)/10}} \leq  \frac{1}{9} \mathrm{exp}\bc{-2s2^{k/4}},
\end{align*}
as claimed.
\end{proof}

Finally, the following lemma deals with the first term from \eqref{eqalmostDone1}.

\begin{lemma}\label{Lemma_NMAs}
We have $\pr\brk{\cA_s \cap \cbc{\sum_{i=1}^{\vgamma^-}\bc{\vX_i^--\vY_i^-}\geq\frac{s}{3}}}\leq\frac{1}{9}\mathrm{exp}\bc{-2s2^{k/4}}.$
\end{lemma}

\noindent
The proof of \Lem~\ref{Lemma_NMAs} requires several steps.

\begin{claim}\label{Bound_for_X}
Assume that $s \geq 1/2$ and that for all $t \in [2^{-k/4},s]$ we have
\begin{align*}
\Pr\bc{\ln{\frac{\veta}{1-\veta}}\geq t} \leq \mathrm{exp}\bc{-t2^{k/4}}.
\end{align*}
Then $\Pr\bc{\vX_1^- \geq t} \leq \mathrm{exp}\bc{-\frac{t}{2}(k-1)2^{k/4}}$ for all $t \in [1,2s]$.
\end{claim}
\begin{proof}
For $t \in [1, 2s]$ we obtain
\begin{align*}
\Pr\brk{\vX_1^+\geq t}&\leq \Pr\brk{\forall j \in [k-1]: \veta_{1j}^- \geq \frac{1-\eul^{-t}}{1-\eul^{-\beta}}}\\
&=\Pr\brk{\ln{\frac{\veta}{1-\veta}} \geq \ln{\frac{1-\eul^{-t}}{\eul^{-t}-\eul^{-\beta}}}}^{k-1}
\leq\Pr\brk{\ln{\frac{\veta}{1-\veta}} \geq \frac{t}{2}}^{k-1}
\leq \mathrm{exp}\bc{-\frac{t}{2}(k-1)2^{k/4}},
\end{align*}
as desired.
\end{proof}

Thus, we have a tail bound for $\vX_1^-$ up to $2s$.
To bound the probability that $\vX_1^-$ grows even larger, we are going to condition on the event $\cA_s$.
Indeed, on $\cA_s$ for all $i=1, \ldots, \vgamma^-$ we have
\begin{align}\label{eqalmostDone2}
\vX_i^- = -\ln\bc{1-\bc{1-\eul^{-\beta}}\prod_{j=1}^{k-1}\veta_{ij}^-} \leq - \ln\bc{1-\bc{1-\mathrm{e}^{-\beta}}\bc{\frac{\mathrm{e}^s}{1+\mathrm{e}^s}}^{9(k-1)/10}}.
\end{align}
The following two claims show that the $\vX_i^-$ are bounded by $2s$ deterministically on $\cA_s$.

\begin{claim} \label{X_for_log}
For all $s \leq \ln k$ we have
$- \ln\bc{1-\bc{1-\mathrm{e}^{-\beta}}\bc{\frac{\mathrm{e}^s}{1+\mathrm{e}^s}}^{9(k-1)/10}} \leq 1.$
\end{claim}
\begin{proof}
This is equivalent to showing that
\begin{align*}
1-\bc{1-\mathrm{e}^{-\beta}}\bc{\frac{\mathrm{e}^s}{1+\mathrm{e}^s}}^{9(k-1)/10} \geq \frac{1}{\mathrm{e}}.
\end{align*}
The left hand side is strictly decreasing in $s$ and thus is is sufficient to show the claim for $s=\ln k$, in which case
\begin{align*}
1-\bc{1-\mathrm{e}^{-\beta}}\bc{\frac{\mathrm{e}^{\ln k}}{1+\mathrm{e}^{\ln k}}}^{9(k-1)/10} &= 1-\bc{1-\mathrm{e}^{-\beta}}\bc{1-\frac{1}{k+1}}^{9(k-1)/10} \geq 1-\bc{1-\mathrm{e}^{-\beta}}\mathrm{exp}\bc{-\frac{9(k-1)}{10(k+1)}}\geq\frac12,
\end{align*}
as desired.
\end{proof}

\begin{claim}\label{log_is_no_more_than_s}
For $s > \ln k$ we have
$- \ln\bc{1-\bc{1-\mathrm{e}^{-\beta}}\bc{\frac{\mathrm{e}^s}{1+\mathrm{e}^s}}^{9(k-1)/10}} \leq s.$
\end{claim}
\begin{proof}
We have
\begin{align*}
\bc{1-\frac{1}{1+\mathrm{e}^s}}^{9(k-1)/10} \leq  1 - \bc{\frac{9}{10}}^2 (k-1)\mathrm{e}^{-s} + \frac{1}{2} \bc{\frac{9(k-1)}{10}}^2 \mathrm{e}^{-2s}.
\end{align*}
Therefore,
\begin{align*}
- \ln\bc{1-\bc{1-\mathrm{e}^{-\beta}}\bc{\frac{\mathrm{e}^s}{1+\mathrm{e}^s}}^{9(k-1)/10}} & \leq - \ln\bc{1-\bc{\frac{\mathrm{e}^s}{1+\mathrm{e}^s}}^{9(k-1)/10}} \leq - \ln\bc{ \bc{\frac{9}{10}}^2 (k-1) \mathrm{e}^{-s} -\frac{1}{2} \bc{\frac{9(k-1)}{10}}^2 \mathrm{e}^{-2s}}\\
& = -2 \ln\bc{\frac{9}{10}}-\ln(k-1)+s-\ln\bc{1-\frac{k-1}{2}\mathrm{e}^{-s}}\leq s,
\end{align*}
as claimed.
\end{proof}

\begin{proof}[Proof of \Lem~\ref{Lemma_NMAs}]
On the event $\cA_s$ we have 
\begin{align*}
\sum_{i=1}^{\vgamma^-}{\vX_i^--\vY^-_i} \leq \sum_{i=1}^{\vgamma^-} \vecone\cbc{\vX_i^- \geq \bc{1-\mathrm{e}^{-\beta}}2^{-9(k-1)/10}} + \sum_{i=1}^{\vgamma^-}\vX_i^-\vecone\cbc{\vX_i^- \in [1,s]} \vecone\cbc{\vX_i \geq \bc{1-\mathrm{e}^{-\beta}}2^{-9(k-1)/10} }. 
\end{align*}
We now reprove Lemma \ref{Lemma_tails_of_X}, part (ii), with $\eps=1/10$ to get an upper bound on $\pr\brk{\vX_i^- \geq \bc{1-\mathrm{e}^{-\beta}}2^{-9(k-1)/10}}$. First of all, as in (\ref{Bound_a_b}), rearranging and elimination of $\beta$ gives
\begin{align} \label{Bound_a_b_2}
\Pr\brk{\vX_1^- \geq  (1-\eul^{-\beta})2^{-(9/10)(k-1)}} \leq\Pr\brk{\prod_{j=1}^{k-1}\veta_{1j}^+ \geq 1-\mathrm{exp}\bc{-2^{-(9/10)(k-1)}}}.
\end{align}
Further, since $\veta_{1j}\leq1$ for all $j$, for any $a \in [0,1]$ and $b \in (0,k-2]$ we have
\begin{align}\label{prod_eta_2}
\Pr\brk{\prod_{j=1}^{k-1}\veta_{1j}^+ \geq a} \leq \Pr\brk{\left| \left\{j \in [k-1]: \veta_{1j}^+ \geq a^{\frac{1}{k-1-b}}\right\}\right| \geq b}.
\end{align}
Moreover, for large $k$ we have 
\begin{align} \label{prod_etaa_2}
  1-\mathrm{exp}\bc{-2^{-(9/10)(k-1)}} \geq 2^{-(k-1)(29/30)^2},
\end{align}
which was proved in Section \ref{Sec_tails_of_X}. Combining \eqref{Bound_a_b_2}, \eqref{prod_eta_2} and \eqref{prod_etaa_2}, we obtain
\begin{align*}
\Pr\bc{\vX_1^- \geq  \bc{1-\eul^{-\beta}}2^{-(9/10)(k-1)}} &\leq \Pr\bc{\prod_{j=1}^{k-1}\veta_{1j}^+ \geq 2^{-(k-1)(29/30)^2}}{\leq} \Pr\bc{\left| \left\{j \in [k-1]: \veta_{1j}^+ \geq 2^{-29/30}\right\}\right| \geq (k-1)/30} \\
& \leq \sum_{(k-1)\eps/30\leq j\leq k-1}\binom{k-1}{j} \pr\brk{ \veta_{11}^+ \geq 2^{-29/30}}^j\leq 2^{k-1} \pr\brk{ \veta_{11}^+ \geq 2^{-29/30}}^{(k-1)/30}.
\end{align*}
Since $1/2 \geq \varphi^{-1}(2^{-29/30})\geq 2^{-k/4}$, the assumption of Lemma \ref{Correct_FP2} implies that
\begin{align*}
\Pr\bc{\vX \geq  \bc{1-\eul^{-\beta}}2^{-(9/10)(k-1)}} &
\leq 2^{k-1} \mathrm{exp}\bc{-\frac{(k-1)}{30}2^{k/4}\varphi^{-1}\bc{2^{-29/30}}}.
\end{align*}
Finally, since $\varphi^{-1}\bc{2^{-29/30}} \geq \frac{2}{30} \ln 2$ and for $k \geq 100$ and $c:=1/1000$
\begin{align}\label{bound_small_z}
\Pr\brk{\vX_1^- \geq \bc{1-\mathrm{e}^{-\beta}}2^{-9(k-1)/10}} \leq \mathrm{exp}\bc{-ck2^{k/4}}.
\end{align}
Moreover, for all $a \in [1,s]$, Lemma \ref{Bound_for_X} implies that also
\begin{align}\label{bound_large_z}
\Pr\bc{\vX_1^-\geq a} \leq \mathrm{exp}\bc{-2cak2^{k/4}}.
\end{align}
Set $\vartheta= c k 2^{k/4}$.
Then (\ref{bound_small_z}), the definition of the random variable and (\ref{bound_large_z}) show that for $z \geq 0$,
\begin{align*}
\Pr\bc{\vecone\cbc{\vX_i^- > \bc{1-\eul^{-\beta}}2^{-(9/10)(k-1)}}\bc{1+\vecone\cbc{\vX_i^- \in [1,s] }\vX_i^-}> z} \leq \begin{cases}
\Pr\bc{\vX_i^- > \bc{1-\eul^{-\beta}}2^{-(9/10)(k-1)}}\leq \eul^{-\vartheta}, &\mbox{ if } z < 1, \\
\eul^{-2\vartheta}, &\mbox{ if } 1 \leq z < 2, \\
\Pr\bc{\vecone\cbc{\vX_i^- \in [1,s] }\vX_i^- \geq z-1} \leq \eul^{-z\vartheta}, &\mbox{ if } z \geq 2.
\end{cases}
\end{align*}
Hence, we can estimate these random variables as follows.
Let $\bc{\vI_i^-}_{i \geq 1}$ be a sequence of $\Be\bc{\eul^{-\vartheta}}$ random variables, let $(\vZ_i^-)_{i \geq 1}$ be a sequence of exponential random variables with mean $1/\vartheta$ and let $\bar{\vgamma}^-$ be a $\Po(d/2)$ random variable, all  mutually independent.
Then for all $i \geq 1$, $z \geq 0$ we have
\begin{align*}
\Pr\bc{\vecone\cbc{\vX_i^- > \bc{1-\eul^{-\beta}}2^{-(9/10)(k-1)}}\bc{1+\vecone\cbc{\vX_i^- \in [1,s] }\vX_i^-}> z}\leq
\Pr\bc{\vI_i^-(1+\vZ_i^-) > z} = \begin{cases}
\eul^{-\vartheta}, & z < 1, \\
\eul^{-z\vartheta}, & z \geq 1.
\end{cases}
\end{align*}
Thus, $\vecone\cbc{\vX_i^- > \bc{1-\eul^{-\beta}}2^{-(9/10)(k-1)}}(1+\vecone\cbc{\vX_i^- > 1 }\vX_i^-)$ is stochastically dominated by $\vI_i^-(1+\vZ_i^-)$.
Therefore, we also obtain stochastic dominance for the sums of these random variables, i.e., 
\begin{align}\label{eqNoelaStochDom_2}
\sum_{i=1}^{\vgamma^-}&\vecone\bc{\vX_i^- > \bc{1-\eul^{-\beta}}2^{-(9/10)(k-1)}}(1+\vecone\bc{\vX_i^- > 1 }\vX_i^-) \preceq 
\vSigma^-_1+\vSigma^-_2&\mbox{where}\quad
\vSigma^-_1=\sum_{i=1}^{\bar{\vgamma}^-}\vI_i^-,\qquad
\vSigma^-_2=\sum_{i=1}^{\bar{\vgamma}^-}\vI_i^-\vZ_i^-.
\end{align}
Hence, as in the proof of Lemma \ref{Tail_bounds}, we can stochastically dominate the difference $\sum_{i=1}^{\vgamma^-}\bc{\vX_i^--\vY^-_i}$ by a sum of a Poisson random variable and a random variable that is Gamma distributed, conditionally on the Poisson variable.
Thus, we obtain
\begin{align*}
\pr\brk{\cA_s \cap \cbc{\sum_{i=1}^{\vgamma^-}\bc{\vX_i^--\vY_i^-} \geq s/3}} \leq
\Pr\brk{\vSigma_1^- + \vSigma_2^- \geq s/3} \leq \Pr\brk{\vSigma_1^- \geq s/6} + \Pr\brk{\vSigma_2^- \geq s/6}, 
\end{align*}
where $\vSigma_1^-$ has distribution $\Po(\frac d2\exp\bc{-ck2^{k/4}})$ and $\vSigma_2^-$ has distribution $\Gamma(\vSigma_1,1/(2ck2^{k/4}))$.
Bennett's inequality yields
\begin{align*}
\Pr\brk{\vSigma_1^- \geq s/6} &\leq \mathrm{exp}\bc{\frac{s}{6} - \frac{d}{2}\mathrm{exp}\bc{-ck2^{k/4}} + \frac{s}{6}\ln\bc{ \frac{d}{2}\mathrm{exp}\bc{-ck2^{k/4}} }-\frac{s}{6}\ln\bc{\frac{s}{6}}}\leq \frac{1}{18} \mathrm{exp}\bc{-2s2^{k/4}}
\end{align*}
Moreover, we again set $\vartheta := ck2^{k/4}$. Then 
$d/(\vartheta \mathrm{exp}(\vartheta)) \to 0$ as $k \to \infty$, and Claim \ref{Tail_Gamma} yields that for $k$ sufficiently large, 
\begin{align*}
\Pr\bc{\vSigma_2^- \geq s/6}\leq \mathrm{exp}\bc{-ck2^{k/4}\frac{s}{6}\bc{1- 2\sqrt{\frac{d}{12s\vartheta\mathrm{exp}(\vartheta)}}+\frac{d}{12s\vartheta\mathrm{exp}(\vartheta)}}} \leq \frac{1}{18}\exp\bc{-2s2^{k/4}},
\end{align*}
which completes the proof.
\end{proof}

\begin{proof}[Proof of \Lem~\ref{Correct_FP2}]
Combining \eqref{eqalmostDone1} with the estimates from Lemma  \ref{Less_than_half2}, \ref{bound_Azuma}, \ref{bound_Bennett} and \ref{Lemma_NMAs} yields
\begin{align*}
\Pr\brk{\abs{\ln{\frac{\vxi}{1-\vxi}}}\geq 2s} \leq 1-\pr\brk{\cA_s} + \pr\brk{\vecone\cA_s\ln{\frac{\vxi}{1-\vxi}}\geq 2s} \leq \mathrm{exp}\bc{-2s \cdot 2^{k/4}},
\end{align*}
as claimed.
\end{proof}


\section{Proof of \Prop~\ref{Prop_Amin}}\label{Sec_Amin}

\noindent
{\em Throughout this section we assume that \eqref{eqRS} is satisfied.}
We start by estimating the difference of the actual variable-to-clause messages and the pseudo-messages.

\begin{lemma}\label{Lemma_v2c}
	For any $\eps>0$ there is $t_0$ such that for $t>t_0$ and for large enough $n$ we have $$\Erw\sum_{i=1}^n\sum_{a\in\partial x_i}\abs{\mu_{\PHI,\beta,x_i\to a}(1)-\mu_{\PHI,\beta,x_i\to a,t}(1)}<\eps .$$
	\end{lemma}
\begin{proof}
Observe that the double sum amounts to choosing a random clause $\va$ of $\PHI$ and then a random variable $\vx$ that appears in $\va$.
In other words, it suffices to prove that 
\begin{align}\label{eqAmin1}
	\Erw\abs{\mu_{\PHI,\beta,\vx\to \va}(1)-\mu_{\PHI,\beta,\vx\to \va,t}(1)}&=o_t(1).
\end{align}
Furthermore, because the total number $\vm$ of clauses of the random formula $\PHI$ is a Poisson variable with standard deviation $\Theta(\sqrt n)$, the random formulas $\PHI$ and $\PHI-\va$ (obtained by removing $\va$ from $\PHI$) have total variation distance $o(1)$.
Hence, recalling \eqref{eqBP1}--\eqref{eqBP2}, we see that in order to establish \eqref{eqAmin1} it is enough to show that 
\begin{align}\label{eqAmin2}
	\Erw\abs{\mu_{\PHI,\beta}(\{\vec\sigma_{x_1}=1\})-\mu_{\PHI,\beta,x_1,t}(1)}&=o_t(1),&&\mbox{where }\mu_{\PHI,\beta,x_1,t}(s)=\frac{\prod_{a\in\partial x_1}\mu_{\PHI,\beta,a \to x_1,t}(s)}{\prod_{a\in\partial x_1}\mu_{\PHI,\beta,a \to x_1,t}(1)+\prod_{a\in\partial x_1}\mu_{\PHI,\beta,a \to x_1,t}(-1)}.
\end{align}
Indeed, picking a large enough $t>0$ and assuming that $n$ is sufficiently large, we may also condition on the event $\cT$ that the depth-$2t$ neighbourhood of $x_1$ in the factor graph $G(\PHI)$ is acylic and that the total number of variables and clauses in this neighbourhood is bounded by $(kd)^{2t}$.

To prove \eqref{eqAmin2}  let $\PHI^-$ denote the random formula obtained by deleting all clauses $b_1,\ldots,b_\ell$ at distance exactly $2t-1$ from $x_1$ in $\PHI$. 
Further, obtain $\PHI^+$ from $\PHI^-$ by inserting new clauses $b_1',\ldots,b_\ell'$ instead such that 
\begin{itemize}
	\item each $b_i'$ is connected with the same variable $y_i$ at distance $2t-2$ from $x_1$ as $b_i$ with the same sign, i.e., $\sign(y_i,b_i)=\sign(y_i,b_i')$.
	\item the other variables $\vy_{ij}$, $j\in[k-1]$, that occur in the clauses $b_i'$ and their signs are chosen uniformly and independently of the $b_i$.
\end{itemize}
Then $\PHI^+$ and $\PHI$ are identically distributed.
Therefore, to prove~\eqref{eqAmin2} we just need to show that
\begin{align}\label{eqAmin3}
	\Erw\brk{\abs{\mu_{\PHI^+,\beta}(\{\vec\sigma_{x_1}=1\})-\mu_{\PHI^+,\beta,x_1,t}(1)}\mid\cT}&=o_t(1).
\end{align}

Because the formula $\PHI^-$ is obtained from $\PHI$ by merely deleting a bounded number of at most $(kd)^{2t}$ clauses, by their definition \eqref{eqBoltz} the Boltzmann distributions $\mu_{\PHI,\beta}$ and $\mu_{\PHI^-,\beta}$ are mutually $1/\delta$-contiguous for some $\delta=\delta(t)>0$.
The assumption \eqref{eqRS} and \Lem~\ref{Lemma_contig} therefore imply that \whp\
\begin{align}\label{eqAmin3}
\sum_{i=1}^n\abs{\mu_{\PHI,\beta}(\{\vec\sigma_{x_i}=1\})-\mu_{\PHI^-,\beta}(\{\vec\sigma_{x_i}=1\})}&=o(n),\\
\sum_{i,j=1}^n\abs{\mu_{\PHI^-,\beta}(\{\SIGMA_{x_i}=\SIGMA_{x_j}=1\})-\mu_{\PHI^-,\beta}(\{\SIGMA_{x_i}=1\})\mu_{\PHI^-,\beta}(\{\SIGMA_{x_j}=1\})}&=o(n^2).\label{eqAmin3a}
\end{align}
In particular, \eqref{eqAmin3} ensures together with \Cor~\ref{Cor_bal} that the empirical distribution $\pi_{\PHI^-,\beta}$ of the Boltzmann marginals of $\PHI^-$ has slim tails \whp\
Furthermore, \Lem~\ref{Lemma_Victor} implies that \whp\ over the choice of $\PHI^-$ and of the random attachment points $\vy_{ij}$ of the clauses  $b_i'$ in $\PHI^+$ for a sample $\SIGMA^-$ from $\mu_{\PHI^-,\beta}$ we have
\begin{align}\label{eqAmin4}
	\abs{\mu_{\PHI^-,\beta}\bc{\cbc{\forall i\in[\ell],j\in[k-1]:\SIGMA^-_{\vy_{ij}}=\sigma_{ij}}}-\prod_{i=1}^\ell\prod_{j=1}^{k-1}\mu_{\PHI^-,\beta}\bc{\cbc{\SIGMA^-_{\vy_{ij}}=\sigma_{ij}}}}&=o(1)&&
	\mbox{for all }\sigma=(\sigma_{ij})\in\{\pm1\}^{(k-1)\ell}.
\end{align}
In other words, the joint distribution of the $\SIGMA^-_{\vy_{ij}}$ factorises.

Finally, now that we have pinned down the distribution of the boundary condition $(\SIGMA^-_{\vy_{ij}})_{i,j}$, we can easily get a handle on the marginal of $x_1$.
Namely, let $\PHI'$ denote the formula comprising all clauses and variables at distance at most $2t$ from $x_1$ in $\PHI^+$.
Then \eqref{eqAmin4} shows that  for a sample $\SIGMA^+$ from $\mu_{\PHI^+,\beta}$ \whp
\begin{align}\label{eqAmin5}
\mu_{\PHI^+,\beta}(\{\SIGMA_{x_1}^+=s\})&\propto o(1)+\sum_{\sigma\in\{\pm1\}^{V(\PHI')}}\vecone\cbc{\sigma_{x_1}=s}\mu_{\PHI',\beta}(\sigma)\prod_{j=1}^{k-1}\mu_{\PHI^-,\beta}\bc{\cbc{\SIGMA^-_{\vy_{ij}}=\sigma_{\vy_{ij}}}}&&(s=\pm1).
\end{align}
Since $\PHI'$ is acyclic \whp\ and Belief Propagation is exact on acyclic factor graphs by \Thm~\ref{Thm_treeBP}, \eqref{eqAmin5} shows that the Boltzmann marginal $\mu_{\PHI^+,\beta,x_1}(1)$ can be computed by running $t$ iterations of Belief Propagation, with the boundary messages initialised by the marginals $\mu_{\PHI^-,\beta,\vy_{ij}}$.
Moreover, because the marginals $\mu_{\PHI^-,\beta,\vy_{ij}}$ are actually independent samples from the slim-tailed distribution $\pi_{\PHI^-,\beta}$ and since the tree $\PHI'$  is asymptotically distributed as the Galton-Watson tree $\vT$, \Prop~\ref{Prop_Noela} implies the desired bound \eqref{eqAmin2}.
\end{proof}

\begin{lemma}\label{Lemma_c2v}
	For any $\eps>0$ there is $t_0$ such that for $t>t_0$ and for large enough $n$ we have $$\Erw\sum_{i=1}^n\sum_{a\in\partial x_i}\abs{\mu_{\PHI,\beta,a\to x_i}(1)-\mu_{\PHI,\beta,a\to x_i,t}(1)}<\eps n.$$
	\end{lemma}
\begin{proof}
While we could repeat a similar argument as in the proof of \Lem~\ref{Lemma_c2v}, there is a shorter route that uses the Belief Propagation recurrence.
Specifically, \cite[\Thm~1.1]{Will} implies together with the assumption \eqref{eqRS} that \whp\ for all but $o(n)$ adjacent clause/variable pairs $a,x$ we have
\begin{align}\label{eqLemma_c2v_1}
	\mu_{\PHI,\beta,a\to x}(s)&\propto o(1)+\sum_{\sigma\in\{\pm1\}^{\partial a}}\vecone\cbc{\sigma_x=s}\exp(-\beta\vecone\{\sigma\not\models a\})\prod_{y\in\partial a\setminus\cbc x}\mu_{\PHI,\beta,y\to a}(\sigma_y).
\end{align}
Since \Lem~\ref{Lemma_v2c} shows that for large enough $t$ \whp\ we have $\mu_{\PHI,\beta,y\to a}(1)=\mu_{\PHI,\beta,y\to a,t-1}(1)+o_t(1)$ for all $y\in\partial a$ and since \eqref{eqLemma_c2v_1} matches \eqref{eqBP1}, we conclude that $	\mu_{\PHI,\beta,a\to x}(1)=	\mu_{\PHI,\beta,a\to x,t}(1)+o_t(1)$ \whp
\end{proof}

\begin{proof}[Proof of \Prop~\ref{Prop_Amin}]
The proposition is an immediate consequence of \Lem s~\ref{Lemma_v2c} and \ref{Lemma_c2v}.
\end{proof}

Finally, for later use we make a note of the following consequence of the arguments presented in this section.
We recall the distribution $\pi_{d,\beta}^\star$ from \Prop~\ref{Prop_Noela}.

\begin{corollary}\label{Cor_slim}
	Assume that \eqref{eqRS} holds, that $d\leq\dk(k)$ and that the event $\fE$ that $\pi_{\PHI,\beta}$ has very slim tails satisfies $\limsup_{n\to\infty}\pr\brk\fE>0$.
	Then along any subsequence where $\pr\brk\fE>0$ we have $\ex[W_1(\pi_{\PHI,\beta},\pi_{d,\beta}^\star)\mid\fE]=o(1).$ 
	\end{corollary}

\begin{proof}
	The same argument as in the proof of \Lem~\ref{Lemma_v2c} shows that $\pi_{\PHI,\beta}$ can be coupled within total variation distance $o(1)$ to coincide with the distribution of a random variable-to-clause message $\mu_{\PHI,\vx\to\va}(1)$.
	Furthermore, \Lem~\ref{Lemma_v2c} implies that this message is well approximated by $\mu_{\PHI,\vx\to\va,t}(1)$ for a sufficiently large $t$ \whp\ 
	Moreover, due to local weak convergence of the factor graph, the depth-$2t$ neighbourhood of $\vx$ converges weakly in probability to the top $2t$ layers of $\vT$.
	Therefore, the reattachment argument from the proof of \Lem~\ref{Lemma_v2c} implies together with the contraction result from \Prop~\ref{Prop_Noela} that $\mu_{\PHI,\vx\to\va,t}(1)$ converges weakly to $\pi_{d,\beta}^\star$ on $\fE$.
\end{proof}

\section{Replica symmetry breaking}\label{Sec_rsb}

\noindent
In this section we prove \Thm~\ref{Thm_rsb} by way of establishing \Prop s~\ref{Prop_rsint} and~\ref{Prop_rsbint}; the former is required to derive the latter.
Throughout this section we assume that $d,\beta$ satisfy the assumptions of \Thm~\ref{Thm_rsb} for suitable sequences $\eps_k,\beta_0(k)$.
Let $c>0$ be such that $d/k=2^k\log2-c$.

\subsection{Proof of \Prop~\ref{Prop_rsint}}\label{Sec_Prop_rsint}
The proposition asserts a lower bound on $\ex[\log Z(\PHI,\beta)]$ under the assumption that the replica symmetry condition \eqref{eqRS} is satisfied.
The starting point for the proof is the following statement, which is implicit in~\cite{Panchenko}.
For the convenience of the reader a self-contained proof is contained in the appendix.

\begin{lemma}\label{Lemma_ASS}
Assume that \eqref{eqRS} is satisfied.
Then $\liminf_{n\to\infty}\frac{1}{n}\ex[\log Z(\PHI,\beta)]\geq\liminf_{n\to\infty}\ex[\fB_{d,\beta}(\pi_{\PHI,\beta})]$. 
\end{lemma}

As \Lem~\ref{Lemma_ASS} provides a lower bound on $\ex[\log Z(\PHI,\beta)]$ in terms of the Bethe free energy of the empirical distribution $\pi_{\PHI,\beta}$, the logical next step for us is to get a handle on $\pi_{\PHI,\beta}$.
To this end we are going to harness some of the intermediate results from the second moment calculation from \Sec~\ref{Sec_Prop_bal}, particularly \Prop~\ref{Prop_f}.
Of course the techniques deployed in that section are relatively crude.
Instead of dealing with as fine-grained an object as the empirical marginal distribution, the moment calculation has the overlap as its protagonist.
More precisely, let
\begin{align*}
	\fa=\fa(\PHI,\beta)=\scal{\alpha(\SIGMA,\SIGMA')}{\mu_{\PHI,\beta}}
\end{align*}
denote the average overlap of two independent random samples $\SIGMA,\SIGMA'$ from the Boltzmann distribution $\mu_{\PHI,\beta}$.
In \Sec~\ref{Sec_Lem_nomiddle} we will derive the following consequence of \Prop~\ref{Prop_f}.
\begin{lemma}\label{Lemma_nomiddle}
	We have $\fa\in(1/2-k^{100}2^{-k/2},1/2+k^{100}2^{-k/2})\cup(1-k^22^{-k},1)$ \whp 
\end{lemma}

\noindent
Furthermore, a simple consequence of \eqref{eqRS} is that the overlap concentrates about its expectation.

\begin{lemma}[{\cite[\Cor~1.14]{Max}}]\label{Lemma_overlapconc}
	If \eqref{eqRS} holds, then $ \lim_{n\to\infty}\ex\scal{\abs{\alpha(\SIGMA,\SIGMA')-\fa}}{\mu_{\PHI,\beta}}=0.  $ 
\end{lemma}

Combining \Lem s~\ref{Lemma_nomiddle} and~\ref{Lemma_overlapconc}, we see that there are two possibilities: either the typical inner product of two Boltzmann samples is close to zero, i.e., typical Boltzmann samples are essentially orthogonal.
Or the inner product is close to one, in which case $\SIGMA,\SIGMA'$ largely agree.
It is easy to see that in the latter case many of the Boltzmann marginals $\mu_{\PHI,\beta}(\{\SIGMA_{x_i}=1\})$ are strongly polarised, i.e, $\mu_{\PHI,\beta}(\{\SIGMA_{x_i}=1\})$ is fairly close to either zero or one.
To be more precise, it is very easy to derive from \Lem~\ref{Lemma_overlapconc} that if $\fa\geq1-2^{4-k}$, say, then all but $2^{-0.99k}n$ marginals $\mu_{\PHI,\beta}(\{\SIGMA_{x_i}=1\})$ either belong to the interval $(0,2^{-0.99k})$ or to the interval $(1-2^{-0.99k},1)$.

But unfortunately this estimate is far too rough to be useful.
Indeed, recalling \Lem~\ref{Lemma_ASS}, we need to estimate the empirical marginal distribution $\pi_{\PHI,\beta}$ precisely enough to actually estimate $\fB_{d,\beta}(\pi_{\PHI,\beta})$.
Due to the $\eul^{-\beta}$ terms that occur in $\fB_{d,\beta}$, the mere knowledge that most marginals belong to $(0,2^{-0.99k})\cup(1-2^{-0.99k},1)$ does not suffice to calculate the Bethe free energy as even a single unlucky $\eul^{-\beta}$ term might have a huge impact on the expression \eqref{eqBetheFunctional}.
Yet remarkably, thanks to a delicate expansion argument in \Sec~\ref{Sec_Lemma_polarised} we will be able to bootstrap on the rough estimate and derive a much tighter estimate of the Boltzmann marginals.
Let  $\fA$ be the event that
\begin{align}\label{eqO}
	\frac{1}{n}\sum_{i=1}^n\vecone\{\mu_{\PHI,\beta}(\{\SIGMA_{x_i}=1\})\in (0,\exp(-\beta))\cup(1-\exp(-\beta),1) \}\geq 1-2^{-0.98k}.
\end{align}

\begin{lemma}\label{Lemma_polarised}
	If \eqref{eqRS} holds then $\pr[\{\fa\geq1-k^22^{-k}\}\setminus\fA]=o(1)$.
\end{lemma}

Thus, combining \Lem s~\ref{Lemma_nomiddle} and~\ref{Lemma_polarised}, we learn that unless $\fa$ is close to $1/2$, most Boltzmann marginals are actually extremely polarised.
This polarisation is strong enough for us to derive the following explicit lower bound on the Bethe free energy, whose proof can be found in \Sec~\ref{Sec_Lemma_polarisedBethe}.

\begin{lemma}\label{Lemma_polarisedBethe}
	On the event $\fA$ we have $\fB_{d,\beta}(\pi_{\PHI,\beta})\geq2^{-k}(c-\log2/2+o(1))$.
\end{lemma}

Hence, we are left to deal with the scenario that $\fa$ is close to $1/2$, i.e., $\fa\in(1/2-k^{100}2^{-k/2},1/2+k^{100}2^{-k/2})$ as in \Lem~\ref{Lemma_nomiddle}.
In this case it is not difficult to verify that the empirical marginal distribution $\pi_{\PHI,\beta}$ has very slim tails.
Consequently, \Cor~\ref{Cor_slim} shows that $\pi_{\PHI,\beta}$ is close to the distribution $\pi_{d,\beta}^\star$ from \Prop~\ref{Prop_Noela}.
Further, in \Sec~\ref{Sec_Lemma_rsint_Noela} we will be able to derive the following estimate of the latter distribution's Bethe free energy.

\begin{lemma}\label{Lemma_rsint_Noela}
	We have $\fB(\pi_{d,\beta}^\star)=2^{-k}\bc{c-\log 2/2}+o(2^{-k})$. 
\end{lemma}

\noindent
With these ingredients we can now deduce \Prop~\ref{Prop_rsint}.

\begin{proof}[Proof of \Prop~\ref{Prop_rsint}]
Assume that \eqref{eqRS} holds \whp\ 
Let $\fE$ be the event that $\pi_{\PHI, \beta}$ has very slim tails.
	Then Claim~\ref{Lemma_general} and \Lem~\ref{Lemma_nomiddle} imply that $\pr\brk{\fE\cup\fA}=1-o(1)$. 
	There are three cases to consider.
	\begin{description}
		\item[Case 1: $\pr\brk\fE=o(1)$] then the assertion follows from \Lem~\ref{Lemma_polarisedBethe} in combination  with \Lem~\ref{Lemma_ASS}.
		\item[Case 2: $\pr\brk\fA=o(1)$] the assertion follows from \Cor~\ref{Cor_slim}, \Lem~\ref{Lemma_rsint_Noela} and \Lem~\ref{Lemma_ASS}.
		\item[Case 3: neither $\pr\brk\fA=o(1)$ nor $\pr\brk\fE=o(1)$] in this case we combine \Cor~\ref{Cor_slim},  \Lem~\ref{Lemma_rsint_Noela}, \Lem~\ref{Lemma_polarisedBethe} and \Lem~\ref{Lemma_ASS}.
	\end{description}
	Thus, in any case we obtain the desired lower bound on $\ex[\log Z(\PHI,\beta)]$.
\end{proof}

\subsubsection{Proof of \Lem~\ref{Lemma_nomiddle}}\label{Sec_Lem_nomiddle}
We use the techniques and results from \Sec~\ref{Sec_Prop_f} to estimate $\ex\brk{\mu_{\PHI,\beta}(\{\alpha(\SIGMA,\SIGMA')\not\in\cA\})}$ with $\cA=(1/2-k^{100}2^{-k/2},1/2+k^{100}2^{-k/2})\cup(1-k^22^{-k},1)$.
Recalling the function $f(\alpha)$ from \eqref{eqSecondMmt}, we see that
\begin{align}\label{eqLemma_nomiddle1}
	\frac{1}{n}\log\ex\brk{Z(\PHI,\beta)^2\mu_{\PHI,\beta}(\{\alpha(\SIGMA,\SIGMA')\not\in\cA\})}&\leq\max_{\alpha\in\cA}f(\alpha)+o(1).
\end{align}
Hence, because \Thm~\ref{Thm_DSS} implies that $Z(\PHI,\beta)\geq1$ \whp, it suffices to show that $f(\alpha)<0$ for all $\alpha\notin\cA$.
Indeed, Claims~\ref{Claim_f1}--\ref{Claim_f4} reduce our task to proving that $f(1/2+k^{100}2^{-k/2})<0$ and $f(1-2^{4-k})<0$.
Applying Taylor's formula, we obtain $	f(1/2+k^{100}2^{-k/2})\leq f(1/2)-k^{200}2^{-k}<0$ and
\begin{align*}
	f(1-2^{4-k})&\leq\log2+k^32^{-k}\log(2)-k^22^{-k}\log(k)+\frac{d}{k}\log\bc{1-(1-\eul^{-\beta})2^{1-k}+(1-\eul^{-\beta})^22^{-k}(1-k^22^{-k})^k}\\
				&\leq -k^22^{-k}\log(k)+\bc{\frac{d}{k}-2^k\log2}(1-\eul^{-\beta})2^{-k}+O(2^{-k})<0.
\end{align*}
Thus, the assertion follows from \eqref{eqLemma_nomiddle1}.

\subsubsection{Proof of \Lem~\ref{Lemma_polarised}}\label{Sec_Lemma_polarised}
We seize upon the expansion properties of the hypergraph underlying the random formula $\PHI$.
To set up the necessary terminology let $\Phi$ be any $k$-CNF formula on the variable set $V_n=\{x_1,\ldots,x_n\}$ and let $\sigma\in\{\pm1\}^{V_n}$ be a truth assignment.
We say that a variable $x_i$ {\em supports} a clause $a$ of $\Phi$ under $\sigma$ if $x_i\in\partial a$, $\sign(x_i,a)=\sigma_{x_i}$ and $\sign(x,a)\neq \sigma_x$ for all $x\in\partial a\setminus\{x_i\}$.
Hence, $x_i$ contributes the single true literal of $a$.
Let $\supp_{\Phi,\sigma}(x_i)$ be the set of all clauses that $x_i$ supports.
Further, call a set $S\subset V_n$ of variables {\em stable} in $(\Phi,\sigma)$ if 
\begin{description}
	\item[ST1] every $x\in S$ supports at least $10^{-5}k$ clauses that contain variables from $S$ only, and
	\item[ST2] no $x\in S$ appears in more than $10^{-6}k$ clauses $a$ that fail to contain a variable $y\in S$ with $\sign(y,a)=\sigma_y$.
\end{description}
Since the union of two stable sets is stable, we denote by $S(\Phi,\sigma)$ the largest stable set of $(\Phi,\sigma)$.
The following lemma, whose proof we defer to \Sec~\ref{Sec_Lemma_Jean}, asserts that an assignment $\SIGMA$ drawn from the Boltzmann distribution of the random formula $\PHI$ induces a very large stable set \whp

\begin{lemma}\label{Lemma_Jean}
	We have $\ex\brk{\mu_{\PHI,\beta}\bc{\{|S(\PHI,\SIGMA)|\geq 2^{-0.99k}n\}}}\sim1.$
\end{lemma}

In light of \Lem~\ref{Lemma_Jean} we call a formula $\Phi$ {\em normal} if $\scal{\vecone\{|S(\Phi,\SIGMA)|\geq 2^{-0.99k}n\}}{\mu_{\Phi,\beta}}\sim1$, i.e., if its typical Boltzmann samples induce stable sets as large as promised by \Lem~\ref{Lemma_Jean}. 
Furthermore, we call $\Phi$ {\em separable} if 
\begin{align*}
	\mu_{\Phi,\beta}\bc{\cbc{1-k^32^{-k}\leq\alpha(\SIGMA,\SIGMA')\wedge\sum_{y\in S(\Phi,\SIGMA)}\vecone\{\SIGMA_y\neq\SIGMA'_y\}>n\eul^{-10\beta}}}=o(1).
\end{align*}
Hence, it is unlikely that $\SIGMA,\SIGMA'$ have a high overlap but differ on a lot of variables from $S(\Phi,\SIGMA)$. 
In \Sec~\ref{Sec_Lemma_separable} we are going to prove that $\PHI$ is separable \whp

\begin{lemma}\label{Lemma_separable}
The random formula $\PHI$ is separable \whp
\end{lemma}

\noindent
\Lem~\ref{Lemma_polarised} is now an easy consequence of \Lem s~\ref{Lemma_Jean} and~\ref{Lemma_separable}.

\begin{proof}[Proof of \Lem~\ref{Lemma_polarised}]
	Thanks to \Lem s~\ref{Lemma_Jean} and~\ref{Lemma_separable} we may assume that $\PHI$ is normal and separable.
We also assume that the event  $\cbc{\fa \geq 1-k^22^{-k}}$ occurs and that the replica symmetry condition \eqref{eqRS} holds, i.e.,
\begin{align}\label{eqLemma_polarised1}
	\frac{1}{n^2}\sum_{i,j=1}^n\dTV(\mu_{\PHI,\beta,x_i,x_j},\mu_{\PHI,\beta,x_i}\tensor\mu_{\PHI,\beta,x_j})&=o(1).
\end{align}
Draw a random $\SIGMA$ from $\mu_{\PHI,\beta}$ and let $\cV=\cV(\SIGMA)$ be the set of all variables $x\in S(\PHI,\SIGMA)$ such that $\mu_{\PHI,\beta,x}(1)\in(0,\eul^{-\beta})\cup(1-\eul^{-\beta},1)$.
We claim that $|\cV|\geq(1-\eul^{ -\beta })|S(\PHI,\SIGMA)|$ \whp\ over the choice of $\SIGMA$, which would clearly imply the lemma.

To verify this claim assume that $|\cV|<(1-\eul^{-\beta})|S(\PHI,\SIGMA)|$ and draw a second, independent sample $\SIGMA'$ from $\mu_{\PHI,\beta}$.
Let $\vX$ be the number of variables $x\in\cV$ such that $\SIGMA_x\neq\SIGMA'_x$.
Then the asymptotic pairwise independence property \eqref{eqLemma_polarised1} implies together with Chebyshev's inequality that \whp\ over the choice of $\SIGMA,\SIGMA'$ we have
\begin{align*}
	\vX&\geq|\cV|-\sum_{x\in\cV}(\mu_{\PHI,\beta,x}(1)^2+\mu_{\PHI,\beta,x}(-1)^2)+o(n)=\sum_{x\in\cV}\mu_{\PHI,\beta,x}(1)\mu_{\PHI,\beta,x}(-1)+o(n)\geq \exp(-2\beta)n/2+o(n),
\end{align*}
in contradiction to separability.
\end{proof}

\subsubsection{Proof of \Lem~\ref{Lemma_Jean}}\label{Sec_Lemma_Jean}
We prove the lemma by way of a distribution on random $k$-CNF formulas known as the {\em planted model}.
Recall that $\vm=\Po(dn/k)$ is a Poisson variable and consider the following experiment.
\begin{description}
	\item[PL1] draw a truth assignment $\SIGMA^*\in\{\pm1\}^{V_n}$ uniformly at random 
	\item[PL2] then draw a $k$-CNF $\PHI^*=\PHI^*(\SIGMA^*)$ with $\vm\sim\Po(dn/k)$ clauses from the distribution
		\begin{align*}
			\pr[\PHI^*=\Phi\mid\vm,\SIGMA^*]&=\frac{\pr\brk{\PHI=\Phi\mid\vm}\exp(-\beta\cH_\Phi(\SIGMA^*))}{(1-(1-\eul^{-\beta})2^{-k})^{\vm}}.
		\end{align*}
\end{description}

The planted model $(\PHI^*,\SIGMA^*)$ is a tried and tested device for studying the Boltzmann distribution of random formulas~\cite{Barriers}.
Indeed, while it is difficult to tackle the Boltzmann distribution directly, the planted model is amenable to the toolbox of probabilistic combinatorics thanks to its constructive definition {\bf PL1--PL2}.
The following statement ties the two models together.

\begin{lemma}\label{Lemma_noisy}
Let $\cE$ be a set of formula/assignment pairs.
Then $$ \ex\brk{\scal{\vecone\cbc{(\PHI,\SIGMA)\in\cE}}{\mu_{\PHI,\beta}}\mid\vm}\leq\ex[Z(\PHI,\beta)\mid\vm]\pr\brk{(\PHI^*,\SIGMA^*)\in\cE\mid\vm}+o(1).  $$
\end{lemma}
\begin{proof}
	Let $\cE'=\cE\cap\{Z(\Phi,\beta)\geq1\}$.
	Then \Thm~\ref{Thm_DSS} implies that 
	\begin{align}\label{eqLemma_noisy1}
		\ex\brk{\scal{\vecone\cbc{(\PHI,\SIGMA)\in\cE}}{\mu_{\PHI,\beta}}\mid\vm}=\ex\brk{\scal{\vecone\cbc{(\PHI,\SIGMA)\in\cE'}}{\mu_{\PHI,\beta}}\mid\vm}+o(1).
	\end{align}
	Furthermore, the definition {\bf PL1--PL2} of the planted model ensures that
	\begin{align}
		\ex\brk{\scal{\vecone\cbc{(\PHI,\SIGMA)\in\cE'}}{\mu_{\PHI,\beta}}\mid\vm}&=\sum_{(\Phi,\sigma)\in\cE'}\pr\brk{\PHI=\Phi\mid\vm}\mu_{\Phi,\beta}(\sigma)\nonumber\\
																				  &=\sum_{(\Phi,\sigma)\in\cE'}\pr\brk{\PHI=\Phi\mid\vm} \eul^{-\beta\cH_\Phi(\sigma)}  /Z(\Phi,\beta)\leq\sum_{(\Phi,\sigma)\in\cE}\pr\brk{\PHI=\Phi\mid\vm} \eul^{-\beta\cH_\Phi(\sigma)} \nonumber\\
																				  &\leq2^n(1-(1-\eul^{-\beta})2^{-k})^{\vm}\sum_{(\Phi,\sigma)\in\cE} \pr\brk{(\PHI^*,\SIGMA^*) = (\Phi,\sigma)\vert\vm}\nonumber\\&=\ex[Z(\PHI,\beta)\mid\vm]\pr\brk{(\PHI^*,\SIGMA^*)\in\cE\mid\vm}.
																				  \label{eqLemma_noisy2}
	\end{align}
	The assertion follows from \eqref{eqLemma_noisy1} and \eqref{eqLemma_noisy2}.
\end{proof}

To facilitate the use of the planted model we make a note of the following easy upper bound.

\begin{lemma}\label{Lemma_quickUpper}
	We have $2(1-(1-\eul^{-\beta})2^{-k})^{d/k}\leq\exp(2^{-k-1}).$
\end{lemma}
\begin{proof}
Using the bound $d/k\geq 2^k\log 2-3\log(2)/2$ we obtain in the limit of large $\beta$,
\begin{align*}
	\limsup_{\beta\to\infty}\log2+\frac{d}{k}\log(1-(1-\eul^{-\beta})2^{-k})&\leq\log2-\frac{d}{k}\brk{2^{-k}+2^{-2k-1}}\leq \frac{\log2}{2^{k+1}},
\end{align*}
as desired.
\end{proof}

As a final preparation we reformulate the second part {\bf PL2} of the experiment above as follows.
\begin{description}
	\item[PL2a] for each of the $\vm$ clauses $a_1,\ldots,a_{\vm}$ of $\PHI^*$ draw the $k$-tuple of variables that occur in the clause uniformly and independently.
	\item[PL2b] subsequently, once more independently for each $i\in[\vm]$, draw the signs with which the variables appear in clause $a_i$ such that $\pr[\SIGMA^*\not\models a_i\mid\vm,\SIGMA^*]=\eul^{-\beta}/(2^k-1+\eul^{-\beta})$.
\end{description}
The distributions produced by {\bf PL2} and {\bf PL2a--PL2b} coincide because the clauses of $\PHI$ are mutually independent.

We proceed to exhibit a large stable set.
The following lemma shows that in $(\PHI^*,\SIGMA^*)$ most variables support a good number of clauses.
To be precise, for a variable $x$ let $\vs_x$ be the number of clauses $a$ of $\PHI^*$ to which $x$ contributes the only literal that is satisfied under $\SIGMA^*$.

\begin{lemma}\label{Lemma_support}
We have $\pr\brk{\sum_{x\in V_n}\vecone\{\vs_x<10^{-4}k\}>2^{-0.997k}n}<\exp(-n/2^{k})$.
\end{lemma}
\begin{proof}
	Because the total number of clauses of $\PHI^*$ is Poisson, the random variables $(\vs_x)_{x\in V_n}$ are mutually independent Poissons.
	Moreover, {\bf PL2b} shows that $\ex[\vs_x]\sim d/(2^k-1+\eul^{-\beta})=k\log(2)+O(2^{-k})$ for every $x$.
	Therefore, Bennett's inequality from \Lem~\ref{Lemma_Bennett} yields
	\begin{align}\label{eqLemma_support1}
		\pr\brk{\vs_x<10^{-4}k}&\leq 2^{-0.998k}.
	\end{align}
	Furthermore, due to the independence of the $\vs_x$ the sum $\sum_{x\in V_n}\vecone\{\vs_x<10^{-4}k\}$ is a binomial variable.
	Since \eqref{eqLemma_support1} shows that its mean is bounded by $n2^{-0.998k}$, the assertion follows from the Chernoff bound.
\end{proof}

For a variable $x$ let $\vu_x$ be the number of clauses of $\PHI^*$ in which $x$ occurs and that $\SIGMA^*$ fails to satisfy.

\begin{lemma}\label{Lemma_unsat}
We have $\pr\brk{\sum_{x\in V_n}\vecone\{\vu_x>0\}>2^{-0.997k}n}<\exp(-n/2^{k})$.
\end{lemma}
\begin{proof}
Let $\vm_0$ be the total number of clauses of $\PHI^*$ that $\SIGMA^*$ fails to satisfy.
If $\sum_{x\in V_n}\vecone\{\vu_x>0\}>2^{-0.997k}n$, then $\vm_0\geq 2^{-0.997k}n/k$.
But  {\bf PL2b}  ensures that $\vm_0\sim\Po(dn\eul^{-\beta}/(k(2^k-1+\eul^{-\beta}))).$
Therefore, by Bennett's inequality,
\begin{align*}
	\pr\brk{\sum_{x\in V_n}\vecone\{\vu_x>0\}>2^{-0.997k}n}&\leq\pr\brk{\Po(dn\eul^{-\beta}/(k(2^k-1+\eul^{-\beta})))>2^{-0.997k}n/k}\leq\exp(-n/2^k),
\end{align*}
providing that $\beta$ is sufficiently large, as claimed.
\end{proof}

The following lemma shows that the planted model possesses a large stable set with very high probability.

\begin{lemma}\label{Lemma_JeanPlanted}
	On the event $\vm\sim dn/k$ we have $\pr\brk{|S(\PHI^*,\SIGMA^*)|< 2^{-0.99k}n\mid\vm}\leq4\exp(-n/2^k).$ 
\end{lemma}
\begin{proof}
	Due to symmetry we may condition on the event $\SIGMA^*_x=1$ for all variables $x$.
	Starting from the set $S_0$ of all variables $x$ such that $\vs_x\geq10^{-4}k$ and $\vu_x=0$, we attempt to construct a large stable set.
	To this end, we iteratively obtain $S_{i+1}$ from $S_i$ by removing an arbitrary variable $y\in S_i$ that violates one of the conditions {\bf ST1}--{\bf ST2}.
	Hence, either $y$ supports fewer than $10^{-5}k$ clauses comprising variables from $S_i$ only, or $y$ appears negatively in more than $10^{-6}k$ clauses that contain at least one positive literal but whose positive literals stem from $V_n\setminus S_i$ only.
Of course, once no such variable $y$ is left the process stops.
Let $\vT$ be the stopping time of the process.
By \Lem s~\ref{Lemma_support} and~\ref{Lemma_unsat} we may assume that $|S_0|\leq 2^{-0.995k}n$.
Moreover, by construction the final set $S_{\vT}$ is stable and has size at least $|S_0|-\vT$.
Therefore, we just need to bound the probability of the event $\{\vT>2^{-0.991k}n\}$.

Hence, let $t=\lfloor 2^{-0.991k}n\rfloor$, set $\theta=t/n$ and let $R=V_n\setminus S_t$.
Then $R$ contains the set $S_0\setminus S_t$ of variables that our process removes by time $t$ as well as the variables $V_n\setminus S_0$ that were excluded from the beginning.
Since $|S_0|\leq 2^{-0.995k}n$ and $t\leq2^{-0.991k}n$ we have
\begin{align}\label{eqLemma_JeanPlanted1}
	|R|\leq 2t.
\end{align}
Further, let $\vX$ be the number of clauses that are supported by a variable from $R$ and contain a second variable from $R$.
\renewcommand{\fY}{\cC}
Also let $\fY$ be the set of clauses that contain at least one variable from $R$ positively and at least one variable from $R$ negatively but none from $S_t=V_n\setminus R$ positively.
Moreover, let $\vY$ be the number of $R$-$\fY$-edges in $G(\PHI)$.
By construction, if $\vT>t$ then either $\vX>10^{-7}kt$ or $\vY>10^{-7}kt$.
Thus, letting $\cE=\{\vm\sim dn/k,\, |S_0|\leq 2^{-0.995k}n\}$, we have
\begin{align}
	\pr\brk{\vT>2^{-0.991k}n\mid\cE}&\leq	\pr\brk{\vX>10^{-7}ktn\mid\cE}+	\pr\brk{\vY>10^{-7}kt\mid\cE}.  \label{eqLemma_JeanPlanted2}
\end{align}

In light of \eqref{eqLemma_JeanPlanted1}, to bound the first probability $\pr\brk{\vX>10^{-7}ktn\mid\cE}$ we estimate the probability that there {\em exists} a set $\cR\subset V_n$ of size $|\cR|=2t$ such that the number $\cX_{\cR}$ of clauses supported by a variable from $\cR$ that contain a second variable from $\cR$ exceeds $\ell=10^{-7}kt$.
Also let $\cX=\cX_{\{x_1,\ldots,x_{2t}\}}$.
Since $\pr[\cE]\sim1$, we obtain the upper bound
\begin{align}\label{eqLemma_JeanPlanted3}
	\pr\brk{\vX>10^{-7}ktn\mid\cE}&\leq2\pr\brk{\exists \cR:\cX_\cR>\ell\mid\vm\sim dn/k}\leq2\binom{n}{2t}\pr\brk{\cX>\ell\mid\vm\sim dn/k}.
\end{align}
Furthermore, the last probability is easy to estimate.
Indeed, due to {\bf PL2b} the probability that a single clause is supported by a variable from $\cR$ and features a second variable from $\cR$ negatively is bounded by
\begin{align*}
	p=\frac{4k(k-1)\theta^2}{2^k-1+\eul^{-\beta}}.
\end{align*}
Consequently, since the $\vm$ clauses are drawn independently, $\cX$ is stochastically dominated by a binomial variable $\Bin(\vm,p)$.
Combining \eqref{eqLemma_JeanPlanted3} with the Chernoff bound, we therefore obtain
\begin{align}\label{eqLemma_JeanPlanted4}
	\pr\brk{\vX>10^{-7}ktn\mid\cE}&\leq2\bcfr{\eul n}{2t}^{2t}\pr\brk{\Bin(2dn/k,p)>\ell}\leq\exp(-n/2^k).
\end{align}

Moving on to $\vY$, we consider an arbitrary set $\cR$ as above, define $\fY_\cR$ as above as the set of clauses that contain at least two variables from $\cR$ but in which no variable from $V_n\setminus\cR$ occurs positively and let $\cY_\cR$ be the number of $\cR$-$\fY_\cR$-edges in $G(\PHI)$.
Observe that $\cY_\cR/k\leq\fY_\cR\leq\cY_\cR/2$.
Thanks to symmetry it suffices to consider $\cY=\cY_{\{x_1,\ldots,x_{2t}\}}$ and we obtain
\begin{align}\label{eqLemma_JeanPlanted5}
	\pr\brk{\vY>10^{-7}tn\mid\cE}&\leq2\pr\brk{\exists \cR:\cY_\cR>\ell\mid\vm\sim dn/k}\leq2\binom{n}{2t}\pr\brk{\cY>\ell\mid\vm\sim dn/k}.
\end{align}
We bound the last probability by
\begin{align}\label{eqLemma_JeanPlanted6}
	\pr\brk{\cY>\ell\mid\vm\sim dn/k}&\leq\sum_{\ell/k\leq M\leq\ell/2}\binom{2dn/k}{M}\binom{kM}{\ell}2^{-kM}(2\theta)^\ell
	\leq2\binom{2dn/k}{\ell/2}\binom{k\ell/2}\ell2^{-k\ell/2}(2\theta)^{\ell}\leq(10^{10}k\theta)^{\ell/2}.
\end{align}
Finally, the assertion follows from \eqref{eqLemma_JeanPlanted2}, \eqref{eqLemma_JeanPlanted4} and~\eqref{eqLemma_JeanPlanted6}.
\end{proof}

\begin{proof}[Proof of \Lem~\ref{Lemma_Jean}]
The assertion is an immediate consequence of \Lem~\ref{Lemma_noisy}, \Lem~\ref{Lemma_quickUpper} and \Lem~\ref{Lemma_JeanPlanted}.
\end{proof}

\subsubsection{Proof of \Lem~\ref{Lemma_separable}}\label{Sec_Lemma_separable}
We treat two regimes of distances $\sum_{x\in S(\PHI,\SIGMA)}\vecone\{\SIGMA_{x}\neq\SIGMA'_x\}$ separately.
Let us begin with very small distances.

\begin{lemma}\label{Lemma_dist1}
\Whp\ the random formula $\PHI$ has the following property.
For any set $\cV\subset V_n$ of size $|\cV|\leq 10^{-9} k^{-1}2^{-k}n$ the number clauses in which at least two variables from $\cV$ occur is bounded above by $10^{-7}k|\cV|$ . 
\end{lemma}
\begin{proof}
	Fix any such set $\cV$ and let $v=|\cV|/n$ and $\lambda=10^{-7}k$. 
Clearly, we may condition on the event that $\vm\sim dn/k\leq 2^kn$. 
Because the clauses are drawn independently, given $\vm$ the number $\vX_\cV$ of clauses that contain two variables from $\cV$ is stochastically dominated by a $\Bin(\vm,k^2v^2)$ variable.
Therefore,
\begin{align}\label{eqLemma_dist1}
	\pr\brk{\vX_{\cV}>\lambda vn\mid\vm}&\leq\binom{\vm}{\lambda vn}(k^2v^2)^{\lambda vn}\leq\bcfr{\eul2^k k^2 v}{\lambda}^{\lambda vn}.
\end{align}
Thanks to the assumption on $v$, combining \eqref{eqLemma_dist1} with a union bound on sets $\cV$ completes the proof.
\end{proof}

\begin{corollary}\label{Cor_dist1}
	\Whp\ we have $\mu_{\PHI,\beta}\bc{\cbc{10^{-9}k^{-1}2^{-k}n> \sum_{x\in S(\PHI,\SIGMA)}\vecone\{\SIGMA_x\neq\SIGMA_x'\}>n\eul^{-10\beta}}}=o(1).  $
\end{corollary}
\begin{proof}
We may condition on $\PHI$ possessing the property quoted in \Lem~\ref{Lemma_dist1}.
For a pair of assignments $\SIGMA,\SIGMA'$ let $\SIGMA''$ be the assignment $\SIGMA''_x=\SIGMA_x$ for all $x\in S(\PHI,\SIGMA)$ and $\SIGMA''_x=\SIGMA_x'$ for all $x\not\in S(\PHI,\SIGMA)$.
Also let $\cV$ be the set of all variables $x\in S(\PHI,\SIGMA)$ such that $\SIGMA_x\neq\SIGMA'_x$.
We claim that
\begin{align}\label{eqCor_dist1_1}
\cH_{\PHI}(\SIGMA'')\leq \cH_{\PHI}(\SIGMA')-10^{-7}k|\cV|.
\end{align}
Indeed, by {\bf ST1} every $x\in S(\PHI,\SIGMA)$ supports at least $10^{-5}k$ clauses.
If $\SIGMA'_x\neq\SIGMA_x$, then $\SIGMA'$ can only satisfy those clauses supported by $x$ that contain a second variable from $\cV$.
But \Lem~\ref{Lemma_dist1} shows that there are no more than $10^{-7}k|\cV|$ such clauses.
Furthermore, {\bf ST2} ensures that there are no more than $10^{-6}k|\cV|$ clauses that $\SIGMA'$ satisfies and that $\SIGMA''$ fails to satisfy.
Thus, we obtain \eqref{eqCor_dist1_1}.
Finally, \eqref{eqCor_dist1_1} implies 
\begin{align*}
	\sum_{\eul^{-10\beta}n<t<10^{-9}k^{-1}2^{-k}n}\mu_{\PHI,\beta}\bc{\cbc{\sum_{x\in S(\PHI,\SIGMA)}\vecone\{\SIGMA_{x}\neq\SIGMA'_x \}=t}}&\leq \sum_{\eul^{-10\beta}n<t}\binom nt\exp(-10^{-7}\beta kt)=o(1),
\end{align*}
as desired.
\end{proof}

We proceed to assignment pairs that differ on an intermediate number of variables from the stable set.

\begin{lemma}\label{Lemma_dist2}
	The pair $(\PHI^*,\SIGMA^*)$ has the following property with probability at least $1-\exp(-n/2^k)$.
	Let $\sigma$ be any assignment such that $ t=\sum_{x\in S(\PHI^*,\SIGMA^*)}\vecone\{\sigma_x\neq\SIGMA^*_x\}\in(10^{-9}k^{-1}2^{-k}n,k^{-4}n).  $ 
	Furthermore, let $\sigma'_x=\SIGMA^*_x$ for all $x\in S(\PHI^*,\SIGMA^*)$ and set $\sigma_x'=\sigma_x$ for all $x\not\in S(\PHI^*,\SIGMA^*)$.
	Then $ \cH_{\PHI^*}(\sigma)\geq\cH_{\PHI^*}(\sigma')+10^{-6}kt.  $
\end{lemma}
\begin{proof}
	Let $\lambda=10^{-7}k$.
	We pursue a similar strategy as in the previous lemma, but this time we confine ourselves to the clauses supported by variables from $S(\PHI^*,\SIGMA^*)$.
	Indeed, by {\bf ST1} every variable $x\in S(\PHI^*,\SIGMA^*)$ supports at least $10^{-5}k$ clauses.
Hence, if $\sigma_x\neq\SIGMA^*_x$, then for $\sigma$ to satisfy such a clause, it must contain another variable $y$ such that $\sigma_y\neq\SIGMA^*_y$ negatively.

Consequently, to prove the assertion it suffices to show that $\PHI^*$ has the following property \whp\
Let $\PHI^\star$ be the sub-formula obtained by retaining only those clauses that contain a single true literal under $\SIGMA^*$.
Then the probability that there exists a set $\cV$ of variables of size $10^{-9}k^{-1}2^{-k}\leq|\cV|/n\leq k^{-4}$ such that the number $\vX_\cV$ of clauses of $\PHI^\star$ that contain one variables from $\cV$ negatively and one positively exceeds $\lambda |\cV|$ is upper bounded by $\exp(-n/2^k)$.

Thus, fix a set $\cV$ as above and let $v=|\cV|/n$.
The number of clauses of $\PHI^\star$ that contain one variable from $\cV$ positively and one negatively is stochastically dominated by $\Po(pdn/k)$ with $p=k^2v^2/(2^k-1+\eul^{-\beta})$.
Therefore, Bennett's inequality shows that
\begin{align}\label{eqLemma_dist2}
	\pr\brk{\vX_\cV>\lambda v n}&\leq\exp(10^{-8}kv\log(dp/(k\lambda v))).
\end{align}
Combining \eqref{eqLemma_dist2} with a union bound on sets $\cV$ completes the proof.
\end{proof}

\begin{corollary}\label{Cor_dist2}
	\Whp\ we have $\mu_{\PHI,\beta} \bc{\cbc{k^{-4}n> \sum_{x\in S(\PHI,\SIGMA)}\vecone\{\SIGMA_{x}\neq\SIGMA'_x\}\geq nk^{-1/2}2^{-k}n}}=o(1).$
\end{corollary}
\begin{proof}
	Invoking Lemmas~\ref{Lemma_noisy} and \ref{Lemma_quickUpper}, we extend the statement of \Lem~\ref{Lemma_dist2} from the planted model $(\PHI^*,\SIGMA^*)$ to the random pair $(\PHI,\SIGMA)$.
	Then we follow the steps of the proof of \Cor~\ref{Cor_dist1}.
\end{proof}

\begin{proof}[Proof of \Lem~\ref{Lemma_separable}]
	The assertion follows from  Corollaries~\ref{Cor_dist1} and~\ref{Cor_dist2}.
\end{proof}


\subsubsection{Proof of \Lem~\ref{Lemma_polarisedBethe}}\label{Sec_Lemma_polarisedBethe}

Let $\vec\mu$ be a random variable with distribution $\pi_{\PHI,\beta}$ and let $\vec J=\pm1$ be an independent random variable with $\ex[\vec J]=0$.
Moreover, let $\pi_{\PHI,\beta}'$ be the distribution of $(1+\vec J(2\vmu-1))/2$.
Further, let $(\vmu_{i,j})_{i,j}$ be a family of independent samples from $\pi_{\PHI,\beta}'$ and let $\vgamma^{\pm}$ be two independent $\Po(d/2)$ variables.
We are going to estimate the two contributions to the Bethe free energy separately.
The following claim deals with the second part.

\begin{claim}\label{Claim_Jean2}
	We have $-\frac{d(k-1)}{k}\ex\log1-(1-\eul^{-\beta})\prod_{j=1}^k\vmu_{1,j}\geq\frac{d(k-1+o_k(1))}{k2^k}\beta+o_\beta(1)$. 
\end{claim}
\begin{proof}
	Let $\cA$ be the event that $\vmu_{1,j}\geq1-\eul^{-\beta}$ for all $j\in[k]$ and let $\bar\cA$ be the complement of $\cA$.
	Then \eqref{eqO} implies together with the fact that $\vmu_{i,j}$ and $1-\vmu_{i,j}$ are identically distributed that 
	\begin{align}\label{eqClaim_Jean2_1}
		\pr\brk\cA&\geq2^{-k}+O(2^{-1.9k}).
	\end{align}
	Further, we have
	\begin{align}\label{eqClaim_Jean2_2}
		1-(1-\eul^{-\beta})\prod_{j=1}^k\vmu_{1,j}&\leq1-(1-\eul^{-\beta})^{k+1}\leq(k+1)\eul^{-\beta}\quad\mbox{on }\cA,&
		1-(1-\eul^{-\beta})\prod_{j=1}^k\vmu_{1,j}&\leq1\quad\mbox{on }\bar\cA.
	\end{align}
	Combining \eqref{eqClaim_Jean2_1} and \eqref{eqClaim_Jean2_2}, we obtain the assertion.
\end{proof}

Let
\begin{align*}
	\vec\Pi^+&=\prod_{i=1}^{\vgamma^+}\bc{1-(1-\eul^{-\beta})\prod_{j=1}^{k-1}\vmu_{i,j}},&\vec\Pi^-&=\prod_{i=1}^{\vgamma^-}\bc{1-(1-\eul^{-\beta})\prod_{j=1}^{k-1}\vmu_{i+\vgamma^+,j}}.
\end{align*}

\begin{claim}\label{Claim_Jean1}
	We have $\ex\log\brk{\vec\Pi^++\vec\Pi^-}\geq\beta\bc{-d2^{-k}+\Omega(\sqrt k)}+o_\beta(\beta).$ 
	\end{claim}
\begin{proof}

	As in the previous proof we are going to separate the clause terms $\prod_{j=1}^{k-1}\vmu_{i,j}$ with all $\vmu_{i,j}$ close to one from the rest.
	Specifically, let $\vg^+_1$ be the number of indices $i\leq\vgamma^+$ such that $\mu_{i,j}\geq1-\eul^{-\beta}$ for all $j\in[k-1]$.
	Moreover, let $\vg^+_0$ be the number $i\leq\vgamma^+$ such that $\mu_{i,j}\leq\eul^{-\beta}$ for some $j\in[k-1]$ and let $\vg^+_*=\vgamma^+-\vg^+_1-\vg^+_0$.
	Also define $\vg^-_0,\vg^-_1,\vg^-_*$ analogously for the second summand and let $\vg_*=\vg^+_*+\vg^-_*$, $\vg_0=\vg^+_0+\vg^-_0$.
	Then we obtain the lower bounds
	\begin{align*}
		\vec\Pi^{\pm}&\geq\exp(-\beta(\vg^{\pm}_1+\vg^{\pm}_*)-\vg^{\pm}_0).
	\end{align*}
	Hence,
	\begin{align}\label{eqClaim_Jean1_1}
		\log(\vec\Pi^++\vec\Pi^-)&\geq\log(\vec\Pi^+\vee\vec\Pi^-)\geq-\beta\bc{\vg^+_1\wedge\vg^-_1+\vg_*}-\vg_0.
	\end{align}
	Recalling \eqref{eqO} and using Poisson thinning, we can view $\vg^{\pm1}_1$, $\vg_*$ and $\vg_0$ as independent Poissons with means
	\begin{align}\label{eqClaim_Jean1_2}
		\lambda^+_1&=\lambda^-_1\leq\frac{d}{2^k+2^{-1.9k}},&\lambda_*&\leq d2^{-1.9k},&\lambda_0\leq d.
	\end{align}
	Additionally, invoking the normal approximation to the Poisson distribution, we obtain
	\begin{align}\label{eqClaim_Jean1_3}
		\ex\brk{\vg^+_1\wedge\vg^-_1}&\leq\lambda^+_1-\Omega\bc{ \sqrt{\lambda_+^1} }=2^{-k}d-\Omega(\sqrt k).
	\end{align}
	Finally, combining \eqref{eqClaim_Jean1_1}--\eqref{eqClaim_Jean1_3} we obtain the assertion.
\end{proof}

\begin{proof}[Proof of \Lem~\ref{Lemma_polarisedBethe}]
The lemma is an immediate consequence of Claims~\ref{Claim_Jean2} and~\ref{Claim_Jean1}.
\end{proof}

\subsubsection{Proof of \Lem~\ref{Lemma_rsint_Noela}}\label{Sec_Lemma_rsint_Noela}
Let $(\vmu_{i,j})_{i,j\geq1}$ be a sequence of independent samples from $\pi^\star_{d,\beta}$ and let $\vgamma^{\pm}$ be two independent Poisson variables with mean $d/2$.
Then
\begin{align*}
	\cB( \pi^\star_{d,\beta})&=\ex\brk{\log\prod_{i=1}^{\vgamma^+}\bc{1-(1-\eul^{ -\beta })\prod_{j=1}^{k-1}\vmu_{i,j}}+\prod_{i=1}^{\vgamma^-}\bc{1-(1-\eul^{-\beta})\prod_{j=1}^{k-1}\vmu_{\vgamma^++i,j}}}-\frac{d(k-1)}{k}\ex\brk{\log1-(1-\eul^{ -\beta })\prod_{j=1}^k\vmu_{1,j}}\nonumber.
\end{align*}
Hence, for large enough $\beta$ we obtain
\begin{align}\label{eqLemma_rsint_Noela2}
	\cB( \pi^\star_{d,\beta})&=\ex\brk{\log\prod_{i=1}^{\vgamma^+}\bc{1-\prod_{j=1}^{k-1}\vmu_{i,j}}+\prod_{i=1}^{\vgamma^-}\bc{1-\prod_{j=1}^{k-1}\vmu_{\vgamma^++i,j}}}-\frac{d(k-1)}{k}\ex\brk{\log1-\prod_{j=1}^k\vmu_{1,j}}+o(2^{-k}).
\end{align}
We expand the two terms on the r.h.s.\ separately
Due to the independence of the $\vmu_{1,j}$ and \eqref{Tail_bound0} we obtain
\begin{align*}
	\frac{d(k-1)}{k}\ex\brk{\log1-\prod_{j=1}^k\vmu_{1,j}}&=-\frac{d(k-1)}{k}\brk{\ex\bc{\prod_{j=1}^k\vmu_{1,j}}+\frac{1}{2}\ex\bc{\prod_{j=1}^k\vmu_{1,j}^2}+O\bc{\ex\bc{\prod_{j=1}^k\vmu_{1,j}^3} }}\\
														  &=-\frac{d(k-1)}{k}\brk{\ex[\vmu_{1,1}]^k+\frac{1}{2}\ex\bc{\vmu_{1,1}^2}^k+O\bc{\ex\bc{\vmu_{1,1}^3}^k }}.
\end{align*}
Hence, because the construction of  $\pi^\star_{d,\beta}$ ensures that $\vmu_{1,1}$ and $1-\vmu_{1,1}$ are identically distributed and thus $\ex[\vmu_{1,1}]=1/2$ and because  $\pi^\star_{d,\beta}$ satisfies \eqref{Tail_bound0}, we obtain
\begin{align}\label{eqLemma_rsint_Noela3}
	\frac{d(k-1)}{k}\ex\brk{\log1-\prod_{j=1}^k\vmu_{1,j}}&=-\frac{d(k-1)}{k}\brk{2^{-k}+2^{-2k-1}+o(4^{-k})}.
\end{align}

Moving on to the other term, we set $\Pi^+=\prod_{i=1}^{\vgamma^+}1-\prod_{j=1}^{k-1}\vmu_{i,j}$ and $\Pi^-=\prod_{i=1}^{\vgamma^-}1-\prod_{j=1}^{k-1}\vmu_{\vgamma^++i,j}$.
Then
\begin{align}
	\ex\brk{\log\prod_{i=1}^{\vgamma^+}\bc{1-\prod_{j=1}^{k-1}\vmu_{i,j}}+\prod_{i=1}^{\vgamma^-}\bc{1-\prod_{j=1}^{k-1}\vmu_{\vgamma^++i,j}}}
	&=\ex\brk{\log\Pi^++\Pi^-} =\ex\brk{\log\Pi^+}+\ex\brk{\log2+\frac{\Pi^-}{\Pi^+}-1}\nonumber\\
	&=\log(2)+\frac{d}{2}\ex\log\bc{1-\prod_{j=1}^{k-1}\vmu_{1,i}}+\ex\brk{\log\bc{1+\frac{1}{2}\bc{\frac{\Pi^-}{\Pi^+}-1}}}.\label{eqLemma_rsint_Noela4}
\end{align}
Further,
\begin{align}\label{eqLemma_rsint_Noela5}
	\frac{d}{2}\ex\log\bc{1-\prod_{j=1}^{k-1}\vmu_{1,i}}&=-\frac{d}{2}\brk{-2^{1-k}-2^{1-2k}+o(4^{-k})}=-d2^{-k}-d2^{-2k}+o(2^{-k}).
\end{align}
Moreover, using the inequality $\log(1+x)-x + x^2/2 \le |x|^3$, we obtain
\begin{align*}
	\ex\brk{\log\bc{1+\frac{1}{2}\bc{\frac{\Pi^-}{\Pi^+}-1}}}&\leq\frac{1}{2}\bc{\frac{\Pi^-}{\Pi^+}-1}-\frac{1}{8}\bc{\frac{\Pi^-}{\Pi^+}-1}^2+O\bc{\bc{\frac{\Pi^-}{\Pi^+}-1}^3}.
\end{align*}
Now, 
\begin{align*}
	\frac{\Pi^-}{\Pi^+}&=\exp\bc{\log\frac{\Pi^-}{\Pi^+}}=\exp\bc{\sum_{i=1}^{\vgamma^+}\log\bc{1-\prod_{j=1}^{k-1}\vmu_{i,j}}-\sum_{i=1}^{\vgamma^-}\log\bc{1-\prod_{j=1}^{k-1}\vmu_{i+\vgamma^+,j}}}\\
					   &=1+\sum_{i=1}^{\vgamma^+}\log\bc{1-\prod_{j=1}^{k-1}\vmu_{i,j}}-\sum_{i=1}^{\vgamma^-}\log\bc{1-\prod_{j=1}^{k-1}\vmu_{i+\vgamma^+,j}}
						+\frac{1}{2}\bc{\sum_{i=1}^{\vgamma^+}\log\bc{1-\prod_{j=1}^{k-1}\vmu_{i,j}}-\sum_{i=1}^{\vgamma^-}\log\bc{1-\prod_{j=1}^{k-1}\vmu_{i+\vgamma^+,j}}}^2\\
					   &\qquad+O\bc{\sum_{h\geq3}\frac{1}{h!}\bc{\sum_{i=1}^{\vgamma^++\vgamma^-}\log\bc{1-\prod_{j=1}^{k-1}\vmu_{i,j}}}^h}.
\end{align*}
Hence, using that the $\vgamma^{\pm}$ and the $\vmu_{1,i,j},\vmu_{2,i,j}$ are identically distributed, we obtain
\begin{align}\label{eqLemma_rsint_Noela6}
	\ex\brk{\log\bc{1+\frac{1}{2}\bc{\frac{\Pi^-}{\Pi^+}-1}}}&=
	\frac{1}{8}\ex\brk{\bc{\sum_{i=1}^{\vgamma^+}\log\bc{1-\prod_{j=1}^{k-1}\vmu_{i,j}}-\sum_{i=1}^{\vgamma^-}\log\bc{1-\prod_{j=1}^{k-1}\vmu_{i+\vgamma^+,j}}}^2}\\
															 &\qquad+O\bc{\sum_{h\geq3}\frac{1}{h!}\ex\brk{\bc{\sum_{i=1}^{\vgamma^++\vgamma^-}\log\bc{1-\prod_{j=1}^{k-1}\vmu_{i,j}}}^h}}.  \nonumber
\end{align}
Further, using the tail bound  \eqref{Tail_bound0} for $\pi^\star_{d,\beta}$ we obtain 
\begin{align}\label{eqLemma_rsint_Noela7}
	\ex\brk{\bc{\sum_{i=1}^{\vgamma^+}\log\bc{1-\prod_{j=1}^{k-1}\vmu_{i,j}}-\sum_{i=1}^{\vgamma^-}\log\bc{1-\prod_{j=1}^{k-1}\vmu_{i+\vgamma^+,j}}}^2}
		&=\ex\brk{\bc{\vgamma^+-\vgamma^-}^2}2^{2-2k}+o(2^{-k})=d2^{2-2k}+o(2^{-k}).
\end{align}
Similarly, the tail bound \eqref{Tail_bound0} and Bennett's inequality imply that
\begin{align}\label{eqLemma_rsint_Noela8}
	\sum_{h\geq3}\frac{1}{h!}\ex\brk{\bc{\sum_{i=1}^{\vgamma^++\vgamma^-}\log\bc{1-\prod_{j=1}^{k-1}\vmu_{i,j}}}^h}&=o(2^{-k}).
\end{align}
Combining \eqref{eqLemma_rsint_Noela6}--\eqref{eqLemma_rsint_Noela8} we get
\begin{align}\label{eqLemma_rsint_Noela9}
	\ex\brk{\log\bc{1+\frac{1}{2}\bc{\frac{\Pi^-}{\Pi^+}-1}}}&=2^{-1-2k}d+o(2^{-k}).
\end{align}
Finally, combining \eqref{eqLemma_rsint_Noela2}, \eqref{eqLemma_rsint_Noela3}, \eqref{eqLemma_rsint_Noela4}, \eqref{eqLemma_rsint_Noela5} and \eqref{eqLemma_rsint_Noela9}, we obtain the assertion.

\subsection{Proof of \Prop~\ref{Prop_rsbint}}\label{Sec_Prop_rsbint}
The proof hinges on the so-called ``1-step replica symmetry breaking interpolation method'' from mathematical physics.
Specifically, we seize upon the following result.
Recall that $\vgamma^{\pm}$ signify independent $\Po(d/2)$ variables and that $(\vec\mu_{\pi,i,j})_{i,j}$ is a sequence of independent samples from a distribution $\pi$.

\begin{theorem}[{\cite{PanchenkoTalagrand}}]\label{Thm_PT_rsb}
	For any $y>0,\beta>0$, any probability distribution $\pi$ on $[0,1]$ and any $n\geq1$ we have
	\begin{align*}
		\frac{y}{n}\Erw[\log Z(\PHI,\beta)]&\leq\Erw\brk{\log \Erw\brk{\bc{\prod_{i=1}^{\vgamma^+}1-(1-\eul^{-\beta})\prod_{j=1}^{k-1}\vmu_{\pi,i,j}+\prod_{i=1}^{\vgamma^-}1-(1-\eul^{-\beta})\prod_{j=1}^{k-1}\vmu_{\pi,i+\vgamma^+,j}}^y\mid\vgamma^+,\vgamma^-}}\\
										   &\qquad-\frac{d(k-1)}{k}\log\Erw\brk{\bc{1-(1-\eul^{-\beta})\prod_{j=1}^k\vmu_{\pi,1,j}}^y}.
	\end{align*}
\end{theorem}

\noindent
We apply \Thm~\ref{Thm_PT_rsb} with the specific choice $\pi=\frac{1}{2}(\delta_1+\delta_0).$
For the last expression we obtain 
\begin{align}
	-\frac{d(k-1)}{k}\log\Erw\brk{\bc{1-(1-\eul^{-\beta})\prod_{j=1}^k\vmu_{\pi,1,j}}^y}&=
	-\frac{d(k-1)}{k}\log\bc{1-2^{-k}}+o(2^{-k})=
	-\frac{d(k-1)}{k}\bc{-2^{-k}-2^{-2k-1}}+o(2^{-k})\nonumber\\&=
	2^{-k}d+2^{-2k-1}d-\log(2)+c2^{-k}-2^{-k-1}\log(2)+o(2^{-k}).
	\label{eqProp_rsbint1}
\end{align}
Further, to estimate the first term let $(\vmu_{\pi,i,j,h})_{i,j,h}$ be additional independent samples from $\pi$ and
\begin{align*}
	\Pi_+&=\prod_{i=1}^{\vgamma^+}\bc{1-\prod_{j=1}^{k-1}\vmu_{\pi,i,j,1}},&
	\Pi_-&=\prod_{i=1}^{\vgamma^-}\bc{1-\prod_{j=1}^{k-1}\vmu_{\pi,i,j,2}}.
\end{align*}
Then for large $\beta$ we have
\begin{align}
	\Erw&\brk{\log \Erw\brk{\bc{\prod_{i=1}^{\vgamma^+}1-(1-\eul^{-\beta})\prod_{j=1}^{k-1}\vmu_{\pi,i,j,1}	
	+\prod_{i=1}^{\vgamma^-}1-(1-\eul^{-\beta})\prod_{j=1}^{k-1}\vmu_{\pi,i,j,2}}^y\mid\vgamma^\pm}}=\Erw\brk{\log\Erw\brk{\bc{\Pi^++\Pi^-}^y}\mid\vgamma^\pm}+o(2^{-k}).
	\label{eqProp_rsbint2}
\end{align}
Furthermore, $\Pi^{\pm}$ are $\{0,1\}$-valued random variables and $ \Erw[\Pi_{\pm}\mid\vgamma^{\pm}]=(1-2^{1-k})^{\vgamma^{\pm}} $.
Therefore,
\begin{align}\nonumber
	\Erw&\brk{\log\Erw\brk{\bc{\Pi^++\Pi^-}^y}\mid\vgamma^+,\vgamma^-}=
	\Erw\brk{\log\bc{\sum_{s=\pm1}(1-2^{1-k})^{\vgamma^{s}}\bc{1-(1-2^{1-k})^{\vgamma^{-s}}} +2^y(1-2^{1-k})^{\vgamma^{+}+\vgamma^{-}} }}\\&=
	\Erw\brk{\vgamma^+\log\bc{1-2^{1-k}}}+\Erw\brk{\log\bc{1-(1-2^{1-k})^{\vgamma^{-}}+\bc{ 1-(1-2^{1-k})^{\vgamma^{+}} } (1-2^{1-k})^{\vgamma^{-}-\vgamma^{+}} +2^y (1-2^{1-k})^{\vgamma^{-}} }}\nonumber\\
																																		   &=\frac{d}{2}\log\bc{1-2^{1-k}}+\Erw\brk{\log\bc{1-(1-2^{1-k})^{\vgamma^{-}}+\bc{ 1-(1-2^{1-k})^{\vgamma^{+}} } (1-2^{1-k})^{\vgamma^{-}-\vgamma^{+}} +2^y (1-2^{1-k})^{\vgamma^{-}} }}
																																		   \label{eqProp_rsbint3}
																																	   .\end{align}
Applying Bennett's inequality, we obtain
\begin{align}
\Erw&\brk{\log\bc{1-(1-2^{1-k})^{\vgamma^{-}}+\bc{ 1-(1-2^{1-k})^{\vgamma^{+}} } (1-2^{1-k})^{\vgamma^{-}-\vgamma^{+}} +2^y (1-2^{1-k})^{\vgamma^{-}} }}\nonumber\\&=
\Erw\brk{\log\bc{1-2^{-k} + (1-2^{-k}) (1-2^{1-k})^{\vgamma^{-}-\vgamma^{+}}+2^{y-k} }}+o(2^{-k})\nonumber\\&=
\log(2)+\Erw\brk{\log\bc{1-2^{-k} +2^{y-1-k}+(1-2^{1-k})^{\vgamma^{-}-\vgamma^{+}} }}+o(2^{-k})\nonumber\\&=
\log(2)-2^{-k}+2^{y-1-k}   +\Erw\brk{\log\bc{1+\frac{1}{2}\bc{(1-2^{1-k})^{\vgamma^{-}-\vgamma^{+}}-1} }}+o(2^{-k}).
\label{eqProp_rsbint4}
\end{align}
Further, once more by Bennett's inequality,
\begin{align}
   \Erw\brk{\log\bc{1+\frac{1}{2}\bc{(1-2^{1-k})^{\vgamma^{-}-\vgamma^{+}}-1} }}&=
   \frac{1}{2}\Erw\bc{(1-2^{1-k})^{\vgamma^{-}-\vgamma^{+}}- 1}- \frac{1}{4}\Erw\brk{\bc{(1-2^{1-k})^{\vgamma^{-}-\vgamma^{+}}-1}^2} +o(2^{-k}).
   \label{eqProp_rsbint5}
\end{align}
Since $\vgamma^+,\vgamma^-$ are $\Po(d/2)$ variables, we obtain
\begin{align}
   \Erw\bc{(1-2^{1-k})^{\vgamma^{-}-\vgamma^{+}}- 1}&=
\Erw\brk{\exp\bc{\bc{\vgamma^{-}-\vgamma^{+}} \log\bc{1-2^{1-k}} }-1}\nonumber\\
												&=\Erw\brk{(\vgamma^+-\vgamma^-)\log\bc{1-2^{1-k}} +\frac{1}{2}(\vgamma^+-\vgamma^-)^2\log^2\bc{1-2^{1-k}}}\nonumber\\
												&=\Erw\brk{(\vgamma^+-\vgamma^-)^2}2^{1-2k}=d2^{1-2k}.
												\label{eqProp_rsbint6}
\end{align}
Similarly,
\begin{align}
\label{eqProp_rsbint7}
\Erw\brk{\bc{(1-2^{1-k})^{\vgamma^{-}-\vgamma^{+}}- 1}^2}&= \Erw\brk{(\vgamma^+-\vgamma^-)^2}2^{1-2k}=d2^{1-2k}.
\end{align}
Plugging \eqref{eqProp_rsbint6} and  \eqref{eqProp_rsbint7} into  \eqref{eqProp_rsbint5}, we obtain
\begin{align}
   \Erw\brk{\log\bc{1+\frac{1}{2}\bc{(1-2^{1-k})^{\vgamma^{-}-\vgamma^{+}}-1} }}&=d2^{-1-2k}.
   \label{eqProp_rsbint8}
\end{align}
Finally, combining \eqref{eqProp_rsbint1}, \eqref{eqProp_rsbint2}, \eqref{eqProp_rsbint3},  \eqref{eqProp_rsbint4} and  \eqref{eqProp_rsbint8}, we get
\begin{align}\nonumber
   \Erw&\brk{\log \Erw\brk{\bc{\prod_{i=1}^{\vgamma^+}1-(1-\eul^{-\beta})\prod_{j=1}^{k-1}\vmu_{i,j,1}	+\prod_{i=1}^{\vgamma^-}1-(1-\eul^{-\beta})\prod_{j=1}^{k-1}\vmu_{i,j,2}}^y\mid\vgamma^+,\vgamma^-}} -\frac{d(k-1)}{ky}\log\Erw\brk{\bc{1-(1-\eul^{-\beta})\prod_{j=1}^k\mu_{j}}^y}\\
   &\leq \frac{c-1+2^{y-1}-\log(2)/2}{2^ky}+o(2^{-k}).
\label{eqProp_rsbint9}	\end{align}
To complete the proof, we observe that the function $y\mapsto \bc{c-1+2^{y-1}-\log(2)/2}/{y}$ attains its minimum at $y<1$ if $c<3\log2/2$.
Since the function value for $y=1$ comes to $c-\log(2)/2$, \eqref{eqProp_rsbint9} shows together with \Thm~\ref{Thm_PT_rsb} that for any $c<3\log2/2$ we have
$n^{-1}\Erw\brk{\log Z(\PHI,\beta)}\leq 2^{-k}\bc{c-\log(2)/2-\Omega(1)}$.
Hence, the assertion follows from \Prop~\ref{Prop_rsint}.

\begin{appendix}

\section{Proof of \Lem~\ref{Lemma_ASS}}\label{Sec_Lemma_ASS}

\noindent
We include a full proof of \Lem~\ref{Lemma_ASS} for the sake of completeness.
The argument is an adaptation of the proofs from~\cite{Will}.
Recall that $G(\Phi)$ is the factor graph obtained from a CNF-formula $\Phi$ and set $\Omega=\cbc{\pm 1}$.

\begin{lemma}[\cite{Will}{, Lemma 3.1}] \label{addingfactor}
	For any integer $L>0$ and any $\alpha>0$ there exist $\epsilon=\epsilon(\alpha,L)$, $n_{0}=n_0(\epsilon, L)$ such that the following is true. Suppose $G(\Phi)$ is the factor graph corresponding to any formula $\Phi$ with $n > n_0$ variables. Moreover, assume that $\mu_{\Phi, \beta}$ is $\epsilon$-extremal. Let $G^\star(\Phi)$ be obtained from $G(\Phi)$ by adding $L$ constraints nodes $b_1, \ldots, b_L$ arbitrarily and denote by $\Phi^\star$ the formula corresponding to $G^\star(\Phi)$. Then, $\mu_{\Phi^\star, \beta}$ is $\alpha$-extremal and
	\begin{align} \label{eqTVaddedfact}
	\sum_{x \in V(G(\Phi))} \dTV \bc{\mu_{\Phi,\beta, x}, \mu_{\Phi^\star,\beta, x} } < \alpha n.
	\end{align}
\end{lemma}
In the following, let $\PHI_n$ denote a random formula with $n$ variables $x_1, \ldots, x_n$. Now, we will  proceed to the proof  of \Lem~\ref{Lemma_ASS}. Following Aizenman-Sims-Starr~\cite{Aizen}, we are going to show that
\begin{align} \label{Aizentrick}
\liminf_{n\to\infty}\ex\brk{ \log \frac{Z(\PHI_n,\beta)}{Z(\PHI_{n-1},\beta)}}\geq\liminf_{n\to\infty}\ex[\fB_{d,\beta}(\pi_{\PHI,\beta})].
\end{align}
The assertion then follows by summing on $n$. To prove \eqref{Aizentrick}, we will couple the random variables $Z(\PHI_n,\beta)$ and $Z(\PHI_{n+1},\beta)$ by way of a third formula $ \hat{\PHI}$. Specifically, let $\hat{\PHI}$ be the random formula with variables $x_1, \ldots, x_n$ obtained by including $\vm=\Po(\hat{d}n/k)$ independent random clauses, where
$$ \hat{d} = d\frac{n+k-1}{n}.$$
Further, set $ \q= n / (n+k-1)$ and let $\PHI'$ be a random formula obtained from $\hat{\PHI}$ by deleting each clause with probability $1-\q$ independently. Let $A$ be the set of clauses removed from $\hat{\PHI}$ to obtain $\PHI'$. In addition, obtain $\PHI''$ from $\hat{\PHI}$ by selecting a variable $\mathbf{x}$ uniformly at random and removing all constraints $a \in \partial_{\hat{\G}} \mathbf{x}$ along with $\mathbf{x}$ itself. Then $\PHI'$ is distributed as $\PHI_n$ and $\PHI''$ is distributed as $\PHI_{n-1}$. Thus, $Z(\PHI_n,\beta)$ is distributed as $Z(\PHI',\beta)$ and $Z(\PHI_{n-1},\beta)$ is distributed as $Z(\PHI'', \beta)$.
\begin{fact} \label{TVPHI}
	The two random formulas $\hat{\PHI}, \PHI_n$ have total variance distance $o(1)$. 
\end{fact}
\begin{proof}
Given that $\hat{\vm} = \vm$ both formulas are identically distributed. Moreover, the random variable $\vm$ is Poisson distributed with mean $dn/k$, which has total variation distance $o(1)$ from the distribution of $\hat{\vm}$. 
\end{proof}
For a clause $b$ of $\hat{\PHI}$, we define 
\begin{align*}
	S(b)= \log\brk{ \sum_{\sigma \in \Omega^{\partial b}} \eul^{-\beta \vecone\cbc{ \sigma \nvDash b}} \prod_{y \in \partial b} \mu_{\hat{\PHI}, \beta, y \rightarrow b}\bc{\sigma_y} }.
\end{align*}
\begin{lemma} \label{aizenClause}
$A.a.s.$ we have $\ln \bc{ Z(\hat{\PHI},\beta) / Z (\PHI', \beta)} = o(1)+ \sum_{a \in A} S(a)$.
\end{lemma}
\begin{proof}
	Given $\epsilon>0$ let $L=L(\epsilon) > 0$ be a large enough number, let $\gamma=\gamma(\epsilon, L)>\delta=\delta(\gamma) >0 $ be small enough and assume that $n$ is sufficiently large. Let $X=\abs{ A}$, $X$ is distributed as $\Po((1-\q) \hat{d} n / k) = \Po(d (k-1)/k)$. Then, the construction of $\PHI'$ ensures  that
	\begin{align} \label{eqXnotlarge}
		\Pr[X > L] < \epsilon.
	\end{align}
	Instead of thinking of $\PHI'$ as being obtained from $\hat{\PHI}$ by removing $X$ random clauses, we can think of $\hat{\PHI}$ as being obtained from $\PHI'$ by adding $X$ random clauses $a_1, \ldots, a_X$. More precisely, let $\PHI'_0= \PHI'$ and $\PHI'_i = \PHI'_{i-1} \wedge a_i$ for $i \in [X]$. Then given $X$ the triple $(\PHI', \hat{\PHI},  A)$ has the same distribution as $(\PHI', \PHI_X', \cbc{a_1, \ldots, a_X}).$  Moreover, because $\q \hat{d}n/k= dn/k$, $\PHI'$ has the same distribution as $\PHI_n$. Therefore, our assumption \eqref{eqRS} implies that $\PHI'$ is $o(1)$-extremal a.a.s. Hence, Lemma \ref{addingfactor} implies that $\PHI'_{i-1}$ remains $o(1)$-extremal a.a.s for any $1 \leq i \leq \min \cbc{X,L}$. Consequently, Lemma \ref{Lemma_Victor} implies that $\PHI'_{i-1}$ is $(o(1), k)$-extremal a.a.s. Since $\partial a_i$ is chosen uniformly and independently of $a_1, \ldots, a_{i-1}$, Markov's inequality shows that for every $1 \leq i \leq \min\cbc{ X,L}$,
	\begin{align*}
	\Pr\brk{ \sum_{\tau \in \Omega^k} \abs{ \sum_{\sigma \in \Omega^n} \vecone\cbc{ \forall y \in \partial a_i : \sigma_y = \tau_y} \mu_{\PHI'_{i-1},\beta}(\sigma) - \prod_{y \in \partial a_i} \mu_{\PHI'_{i-1},\beta, y} (\tau_y)}\geq \delta } < \epsilon, 
	\end{align*}
	for $n$ large  enough. Further, since the clauses $(a_{i})_{i \in [X]}$ are chosen independently and because $\mu_{\hat{\PHI}, y \rightarrow a_i}(\tau_y)$ is the marginal in the formula $\hat{\PHI}-a_i$, \eqref{eqTVaddedfact} and \eqref{eqXnotlarge} imply that
	\begin{align*}
		\Pr\brk{ \forall i \in [X]: \sum_{\tau \in \Omega^k} \abs{ \prod_{y \in \partial a_i} \mu_{\hat{\PHI},\beta, y \rightarrow a_i} (\tau_y)  - \prod_{y \in \partial a_i} \mu_{\PHI'_{i-1}, \beta,y} (\tau_y)}\geq \delta } < 2 \epsilon.
	\end{align*}
	Hence,  with probability at least $1-3\epsilon$ the bound 
	\begin{align} \label{eqkfactorize}
		\sum_{\tau \in \Omega^k} \abs{ \sum_{\sigma \in \Omega^n} \vecone\cbc{ \forall y \in \partial a_i : \sigma_y = \tau_y} \mu_{\PHI'_{i-1}, \beta}(\sigma) - \prod_{y \in \partial a_i} \mu_{\hat{\PHI},\beta, y \rightarrow a_i} (\tau_y)} < 2\delta 
	\end{align}
	holds for all $i \in [X]$ simultaneously. Further, the definition \eqref{eqBoltz} of the partition function shows that for any $i \in [X]$,
	\begin{align*}
		\frac{Z(\PHI'_i, \beta)}{Z(\PHI'_{i-1}, \beta)}= \sum_{\sigma \in \Omega^{\partial a_i}} \eul^{-\beta \vecone\cbc {\sigma \nvDash a_i} } \sum_{\tau \in \Omega^n} \vecone\cbc{ \forall y \in \partial a_i :  \tau_y = \sigma_y} \mu_{\PHI'_{i-1}, \beta}(\tau). 
	\end{align*}
Thus, if \eqref{eqkfactorize}holds and if $\delta$ is chosen sufficiently small then
\begin{align*}
	\abs{\frac{Z(\PHI'_i, \beta)}{Z(\PHI'_{i-1}, \beta)} - \sum_{\sigma \in \Omega^{\partial a_i}} \eul^{-\beta \vecone\cbc {\sigma \nvDash a_i} }  \prod_{y \in \partial a_i} \mu_{\hat{\PHI},\beta, y \rightarrow a_i} (\sigma_y) } < \gamma.
\end{align*}
Finally, the assertion follows by taking logarithms and summing over $i=1, \ldots, X$.
\end{proof}

\begin{lemma} \label{Aizenvariable}
	Let $U= \bigcup _{a \in \partial_{\hat{\PHI} \mathbf{x}}} \partial a .$ Then a.a.s we have
	\begin{align*}
		\log \frac{Z( \hat{\PHI}, \beta)}{Z( \PHI'', \beta)} = o(1) + \log \sum_{\tau \in \Omega^U} \prod_{a \in \partial_{\hat{\PHI}}   \mathbf{x}} \eul^{-\beta \vecone\cbc{\tau \nvDash a}} \prod_{y \in \partial a \setminus \mathbf{x}} \mu_{\hat{\PHI}, \beta, y \rightarrow a } (\tau_y).
	\end{align*}
\end{lemma}
\begin{proof}
	Given $\epsilon > 0$ let $L= L(\epsilon)>0$ be large enough, let $\gamma=\gamma(\epsilon,L)> \delta=\delta(\gamma) > 0$ be small enough and assume that $n$ is sufficiently large. Letting $Y = \abs{ \partial_{\hat{\PHI} } \mathbf{x} }$, we can pick $L$ large enough so that
	\begin{align} \label{eqXnotlarge2}
		\Pr\brk{ Y > L} < \epsilon.
	\end{align}
	As in the previous proof, we think of $\hat{\PHI}$ as being obtained from $\PHI''$ by adding a new variable $\mathbf{x}$ and $Y$ independent clauses $a_1, \ldots, a_Y$ such that $\mathbf{x} \in \partial a_i$ for all $i$. Then assumption \eqref{eqRS}, Lemma \ref{addingfactor} and Lemma \ref{Lemma_Victor} imply that 
	\begin{align} \label{eqvariablefactorize}
	 \Pr\brk{  \sum_{\tau \in \Omega^{U \setminus \cbc{\mathbf{x}}}} \abs{ \sum_{\sigma \in \Omega^n} \vecone\cbc{ \forall y \in U \setminus \cbc{\mathbf{x}} : \sigma_y= \tau_y  } \mu_{\PHI'', \beta} (\sigma)- \prod_{i=1}^{Y} \prod_{ y \in \partial a_i \setminus \cbc{\mathbf{x}}   }  \mu_{\hat{\PHI}, \beta, y \rightarrow a_i } (\tau_y)  } \geq \delta \Big\vert Y \leq L } = o(1).
	\end{align}
	In addition, \eqref{eqBoltz} yields
	\begin{align*}
		\frac{Z(\hat{\PHI}, \beta )}{Z(\PHI'', \beta)} = \sum_{\tau \in \Omega^U} \prod_{i=1}^Y \eul^{ - \beta \vecone\cbc{ \tau \nvDash a_i}} \sum_{\sigma \in \Omega^n}\vecone\cbc{ \forall y \in U \setminus \set{\mathbf{x}} : \sigma_y = \tau_y }  \mu_{\PHI'', \beta}(\sigma). 
	\end{align*}
	Hence, \eqref{eqXnotlarge2} and \eqref{eqvariablefactorize} show that with probability at least $1-2\epsilon$,
	\begin{align*}
		\abs{\frac{Z(\hat{\PHI}, \beta )}{Z(\PHI'', \beta)} - \sum_{\tau \in \Omega^U} \prod_{i=1}^Y \eul^{ - \beta \vecone\cbc{ \tau \nvDash a_i}} \prod_{ y \in \partial a_i \setminus \cbc{ \mathbf{x}}   }  \mu_{\hat{\PHI}, \beta, y \rightarrow a_i } (\tau_y)  }< \gamma.
	\end{align*}
The assertion follows by taking logarithm .
\end{proof}

\begin{claim}\label{Claim_Jean} If $a_1, \ldots, a_Y$ are the clauses containing $\mathbf{x}$ and $U= \cup_{i=1}^Y \partial_{\hat{\PHI}} a_i$ then
\begin{align*} 
\sum_{\tau \in \Omega^U} \prod_{i=1}^Y \eul^{ - \beta \vecone\cbc{ \tau \nvDash a_i}} \prod_{ y \in \partial a_i \setminus \cbc{ \mathbf{x}}   }  \mu_{\hat{\PHI}, \beta, y \rightarrow a_i } (\tau_y) = \sum_{\tau(x) = \pm 1} \prod_{i=1}^Y  \sum_{ \tau \in \Omega^{\partial a_i \setminus \set{x}}}\eul^{ - \beta \vecone\cbc{ \tau \nvDash a_i}} \prod_{ y \in \partial a_i \setminus \cbc{ \mathbf{x}}   }  \mu_{\hat{\PHI}, \beta, y \rightarrow a_i } (\tau_y). 
\end{align*}
\begin{proof}
With probability $1-o(1)$ for all $1 \leq i <j \leq Y$ we have $\partial a_i \cap \partial a_j \setminus \set{\mathbf{x}} = \emptyset.$
\end{proof}
\end{claim}

\begin{proof}[Proof of \Lem~\ref{Lemma_ASS}]
Recall that $ \pi_{\PHI,\beta}=\frac1n\sum_{i=1}^n\delta_{\mu_{\PHI,\beta}(\{\SIGMA_{x_i}=1\})}$.
Moreover, let $(\vec\rho_{\pi,i,j})_{i,j\geq1}$ be an array of independent random variables with distribution $\pi_{\PHI,\beta}$ and define $(\MU_{i,j})_{i,j \geq 1}$ as in \eqref{eqmupij}.
Additionally, let $(\hat\MU_{i,j})_{i,j}$ be a family of independent random variables defined accordingly for $\pi_{\hat\PHI,\beta}$.
Then \Lem~\ref{addingfactor} shows that $W_2(\pi_{\hat\PHI,\beta},\pi_{\PHI,\beta})=o(1)$.
Therefore, using \Lem s~\ref{addingfactor} and~\ref{aizenClause} and Wald's identity, we can write
\begin{align}\label{eqfinal1}
	\ex\ln \frac{ Z(\hat{\PHI},\beta) }{ Z (\PHI', \beta)} &= \frac{d(k-1)}{k}\ex\log\brk{ 1-(1-\eul^{-\beta})\prod_{i=1}^k\hat\MU_{1,i}}+o(1)= \frac{d(k-1)}{k}\ex\log\brk{ 1-(1-\eul^{-\beta})\prod_{i=1}^k\MU_{1,i}}+o(1).
\end{align}
Similarly, \Lem s~\ref{addingfactor} and~\ref{Aizenvariable} and Claim~\ref{Claim_Jean} yield
\begin{align}\nonumber
	\ex\ln \frac{ Z(\hat{\PHI},\beta) }{ Z (\PHI'', \beta)} &=\ex\brk{\prod_{i=1}^{\vgamma^+}1-(1-\eul^{-\beta})\prod_{j=1}^{k-1}\hat\MU_{i,j}+\prod_{i=1}^{\vgamma^-}1-(1-\eul^{-\beta})\prod_{j=1}^{k-1}\hat\MU_{i+\vgamma^+,j}}+o(1)\\
															&=\ex\brk{\prod_{i=1}^{\vgamma^+}1-(1-\eul^{-\beta})\prod_{j=1}^{k-1}\MU_{i,j}+\prod_{i=1}^{\vgamma^-}1-(1-\eul^{-\beta})\prod_{j=1}^{k-1}\MU_{i+\vgamma^+,j}}+o(1)
\label{eqfinal2}
\end{align}
Finally, combining \eqref{eqfinal1} and \eqref{eqfinal2} completes the proof.
\end{proof}

\end{appendix}


\begin{thebibliography}{29}

\bibitem{Barriers}
D.~Achlioptas, A.~Coja-Oghlan: 
Algorithmic barriers from phase transitions.
Proc.~49th FOCS (2008) 793--802.

\bibitem{nae}
D.~Achlioptas, C.~Moore:
Random $k$-SAT: two moments suffice to cross a sharp threshold.
SIAM Journal on Computing {\bf 36} (2006) 740--762.

\bibitem{Aizen}
M.~Aizenman, R.~Sims, S. ~Starr: An extended variational principle for the SK spin-glass model. Phys.~Rev.~ B {\bf 68} (2003) 214403

\bibitem{nature}
D.~Achlioptas, A.~Naor, Y.~Peres:
Rigorous location of phase transitions in hard optimization problems.
Nature {\bf 435} (2005) 759--764.

\bibitem{yuval}
D.~Achlioptas, Y.~Peres:
The threshold for random $k$-SAT is $2^k \ln 2 - O(k)$.
Journal of the AMS \textbf{17} (2004) 947--973.

\bibitem{AchSorkin}
D.~Achlioptas, G.~Sorkin: Optimal myopic algorithms for random 3-SAT.
Proc.\ 41st FOCS (2000) 590--600.

\bibitem{harnessing}
V.~Bapst, A.~Coja-Oghlan: Harnessing the Bethe free energy. Random Structures and Algorithms~{\bf49} (2016) 694--741.

\bibitem{victor}
V.~Bapst, A.~Coja-Oghlan:
The condensation phase transition in the regular $k$-SAT model.
Proc.\ 20th RANDOM (2016) \#22

\bibitem{Bartha}
Z.\ Bartha, N.\ Sun, Y.\ Zhang: Breaking of 1RSB in random regular MAX-NAE-SAT. Proc.\ 60th FOCS (2019) 1405--1416.

%
%



\bibitem{Broder}
A.~Broder, A.~Frieze, E.~Upfal: On the satisfiability and maximum satisfiability of random 3-CNF formulas. Proc.~4th SODA (1993) 322--330.

\bibitem{Chao}
M.~Chao, J.~Franco: Probabilistic analysis of two heuristics for the 3-satisfiability problem. SIAM J.~Comput.~{\bf15} (1986) 1106--1118.

\bibitem{Cheeseman}
P.~Cheeseman, B.~Kanefsky, W.~Taylor: Where the {\em really} hard problems are.
Proc.\ IJCAI (1991) 331--337.

\bibitem{BetterAlg}
A.~Coja-Oghlan: A better algorithm for random $k$-SAT. SIAM Journal on Computing {\bf39} (2010) 2823--2864.

\bibitem{BPDec}
A.~Coja-Oghlan: Belief Propagation fails on random formulas. Journal of the ACM {\bf63} (2017) \#49.

\bibitem{Charting}
A.\ Coja-Oghlan, C.\ Efthymiou, N.\ Jaafari, M.\ Kang, T.\ Kapetanopoulos: Charting the replica symmetric phase. Communications in Mathematical Physics 359 (2018) 603--698.

\bibitem{Max}
A.~Coja-Oghlan, Max Hahn-Klimroth: The cut metric for probability distributions.	arXiv:1905.13619.

\bibitem{2SAT}
D.\ Achlioptas, A.\ Coja-Oghlan, M.\ Hahn-Klimroth, J.\ Lee, N.\ M\"uller, M.\ Penschuck, G.\ Zhou: The random 2-SAT partition function. Random Structures and Algorithms, in press.

\bibitem{CR}
V.~Chvatal, B.~Reed:
Mick gets some (the odds are on his side). 
Proc.~33th FOCS (1992) 620--627.

\bibitem{kSAT}
A.~Coja-Oghlan, K.\ Panagiotou:
The asymptotic $k$-SAT threshold.
Advances in Mathematics {\bf288} (2016) 985--1068.


\bibitem{Will}
A.~Coja-Oghlan, W.~Perkins:
Belief Propagation on replica symmetric random factor graph models.
Annales de l\textquotesingle institut Henri Poincare D {\bf5} (2018) 211--249.

\bibitem{Will2}
A.~Coja-Oghlan, W.~Perkins:
Bethe states of random factor graphs.
Communications in Mathematical Physics {\bf 366} (2019) 173--201.

\bibitem{ACOWorm}
A.~Coja-Oghlan, N.~Wormald:
The number of satisfying assignments of random regular $k$-SAT formulas.
Combinatorics, Probability and Computing {\bf27} (2018) 496--530.


%

\bibitem{Cuckoo}
M.\ Dietzfelbinger, A.\ Goerdt, M.\ Mitzenmacher, A.\ Montanari, R.\ Pagh, M.\ Rink: 
Tight thresholds for cuckoo hashing via XORSAT. 
arXiv:0912.0287 (2009).

\bibitem{DSS3}
J.~Ding, A.~Sly, N.~Sun:
Proof of the satisfiability conjecture for large $k$.
Proc.\ 47th STOC (2015) 59--68.

\bibitem{DuboisMandler}
O.\ Dubois, J.\ Mandler: The 3-XORSAT threshold. Proc.\  43rd FOCS (2002) 769--778.

\bibitem{Charis}
C.~Efthymiou, T.~Hayes, D.~Stefankovic, E.~Vigoda, Y.~Yin: Convergence of MCMC and loopy BP in the tree uniqueness region for the hard-core model. SIAM J.\ Comput.\ {\bf48} (2019) 581--643.

\bibitem{Ehud}
E.~Friedgut: Sharp thresholds of graph properties, and the {$k$-SAT} problem.  J.\ AMS {\bf12} (1999) 1017--1054.
 

\bibitem{FriezeWormald}
A.~Frieze, N.~Wormald: Random $k$-Sat: a tight threshold for moderately growing $k$. Combinatorica {\bf 25} (2005) 297--305.

\bibitem{FS}
A.~Frieze, S.~Suen: Analysis of two simple heuristics on a random instance of $k$-SAT. J.~Algorithms {\bf20} (1996) 312--355.

\bibitem{GGGY}
A.~Galanis, L.A.~Goldberg, H.~Guo, K.~Yang: Counting solutions to random CNF formulas. arXiv: 1911.07020 (2019).


\bibitem{Goerdt}
A.\ Goerdt: A threshold for unsatisfiability. J.\ Comput.\ Syst.\ Sci.\ {\bf53} (1996) 469--486

\bibitem{Samuel}
S.~Hetterich: Analysing survey propagation guided decimation on random formulas. Proc.~43rd ICALP (2016) \#65.

%
%

\bibitem{pnas}
F.~Krzakala, A.~Montanari, F.~Ricci-Tersenghi, G.~Semerjian, L.~Zdeborov\'a:
Gibbs states and the set of solutions of random constraint satisfaction problems.
Proc.~National Academy of Sciences {\bf104} (2007) 10318--10323.

\bibitem{MM}
M.~M\'ezard, A.~Montanari:
Information, physics and computation.
Oxford University Press~2009.

\bibitem{MPZ}
M.~M\'ezard, G.~Parisi, R.~Zecchina:
Analytic and algorithmic solution of random satisfiability problems.
Science {\bf 297} (2002) 812--815.

\bibitem{Moitra}
A.~Moitra: Approximate counting, the Lov\'asz local lemma and inference in graphical models. Journal of the ACM {\bf 66} (2019) 1--25.


\bibitem{MZ}
R.~Monasson, R.~Zecchina: The entropy of the $k$-satisfiability problem.
Phys.\ Rev.\ Lett.\ {\bf76} (1996) 3881.

\bibitem{MS}
A.~Montanari, D.~Shah:
Counting good truth assignments of random $k$-SAT formulae.
Proc.\ 18th SODA (2007) 1255--1264.

\bibitem{PanchenkoBook}
D.\ Panchenko: The Sherrington-Kirkpatrick model. Springer (2013).

\bibitem{PanchkSAT}
D.\ Panchenko: On the replica symmetric solution of the $K$-sat model.  Electron.\ J.\ Probab.\ {\bf 19} (2014) \#67.

\bibitem{Panchenko}
D.\ Panchenko: Spin glass models from the point of view of spin distributions.  Annals of Probability {\bf41} (2013) 1315--1361.

\bibitem{PanchenkoSATRSB}
D.\ Panchenko: On the $K$-sat model with large number of clauses.
Random Structures and Algorithms {\bf 52} 536--542.

\bibitem{PanchenkoTalagrand} D.\ Panchenko, M.\ Talagrand: Bounds for diluted mean-fields spin glass models. Probab.\ Theory Relat.\ Fields {\bf130} (2004) 319--336.

\bibitem{Pearl}
J.\ Pearl: Probabilistic reasoning in intelligent systems: networks of plausible inference. Morgan Kaufmann 1988.

\bibitem{PittelSorkin}
B.\ Pittel, G.\ Sorkin:  The satisfiability threshold for $k$-XORSAT.  Combinatorics, Probability and Computing {\bf25} (2016) 236--268.





\bibitem{SSZ}
A.\ Sly, N.\ Sun, Y.\ Zhang:
The number of solutions for random regular NAE-SAT.
Proc.\ 57th FOCS (2016)  724-731.


\bibitem{Talagrand}
M.~Talagrand: The high temperature case for the random K-sat problem. Probab.\ Theory Related Fields {\bf 119} (2001) 187--212.


\bibitem{Valiant}
L.~Valiant: The complexity of enumeration and reliability problems. SIAM Journal on Computing {\bf8}  (1979) 410--421.

\end{thebibliography}
\end{document}